\documentclass{article}

\usepackage[english]{babel}
\usepackage[utf8]{inputenc}
\usepackage{longtable}
\usepackage[mathscr]{eucal}
\usepackage{amsfonts}
\usepackage{amsmath}
\usepackage{amsthm}
\usepackage{color}
\usepackage{amssymb}
\usepackage{enumitem}
\usepackage{graphicx}
\usepackage[centering]{geometry}
\usepackage{multicol}
\usepackage{fancyhdr}
\usepackage{mathtools}
\usepackage{float} 
\usepackage{graphicx}
\usepackage{bbm}

\allowdisplaybreaks

\usepackage{tikz}
\usetikzlibrary{patterns}  	
\usetikzlibrary{arrows}
\usepackage{fp}  		
\usetikzlibrary{calc, shapes, positioning, decorations}
\usepackage{pgfplots}
\usetikzlibrary{math} 
\usetikzlibrary{angles,quotes}
\usetikzlibrary{intersections}  
\pgfplotsset{compat=newest}

\usetikzlibrary{external}

\definecolor{colTspec}{rgb}{0.0, 0.72, 0.92}  
\definecolor{colTAspec}{rgb}{0.5, 0.5, 0.5} 
\colorlet{colPlane}{colTspec!20}

\usepackage[hypertexnames=false, breaklinks,
   colorlinks=true, 
   citecolor=red!70!black,
   linkcolor=blue,
   anchorcolor=cyan,
   urlcolor=green!50!black,
   pagebackref=false,
   ]{hyperref}

\usepackage[breakable]{tcolorbox}
\tcbset{shield externalize} 
\tcbuselibrary{theorems}
\newtcolorbox{bluebox}[1][]%
{left=0mm, right=0mm, bottom=0mm, top=0mm, sharp corners, boxrule=.8pt, before skip=\topsep, after skip=\topsep, colback=cyan!5, colframe=cyan, coltitle=black, fonttitle=\bfseries, title=#1, breakable}
\tcbset{tcbox raise base, sharp corners, left=0mm,right=0mm,top=0mm,bottom=0mm,nobeforeafter, boxrule =.5pt} 
\tcbset{tcbox raise base, sharp corners, left=0mm,right=0mm,top=0mm,bottom=0mm,nobeforeafter, boxrule =.5pt}

\usepackage[color=yellow!30, linecolor=orange, colorinlistoftodos]{todonotes}

\newcommand{\monikatodo}[2][]{\tikzexternaldisable\todo[color=blue!20, linecolor=blue]{#2}\tikzexternalenable}

\newcommand{\monika}[1]{{\color{blue}\footnotesize #1}}

\newcommand{\R}{\mathbb{R}}
\newcommand{\C}{\mathbb{C}}
\newcommand{\Z}{\mathbb{Z}}
\newcommand{\N}{\mathbb{N}}

\newcommand{\J}{\mathcal{J}}

\newcommand{\M}{\mathcal{M}}
\newcommand{\rd}{\mathrm d}
\newcommand{\id}{I}
\newcommand{\e}{\mathrm{e}} 
\newcommand{\I}{\mathrm{i}} 
\newcommand{\mD}{\mathcal{D}}
\newcommand{\wgamma}{\widetilde \gamma}
\newcommand{\weta}{\widetilde \eta}

\DeclareMathOperator{\im}{Im}
\DeclareMathOperator{\re}{Re}

\DeclareMathOperator{\dist}{dist}

\DeclareMathOperator{\Ran}{Ran}
\DeclareMathOperator{\codim}{codim}

\DeclareMathOperator{\diag}{diag}
\DeclareMathOperator{\rank}{rank}

\newcommand{\define}[1]{\emph{#1}}
\newcommand{\strip}{\mathrm{Str}}
\newcommand{\sector}[2]{S_{#1}(#2)} 
\newcommand{\hyperbola}{\mathrm{Hyp}}
\newcommand{\parabola}{\Pi}
\newcommand{\Hind}{\delta}
\newcommand{\sigmaess}{\sigma_{\mathrm{ess} }}  

\theoremstyle{plain}
\newtheorem{theorem}{Theorem}[section]
\newtheorem{theoremCT}{Theorem}

\newtheorem{lemma}[theorem]{Lemma}
\newtheorem*{lemma*}{Lemma}
\newtheorem{proposition}[theorem]{Proposition}
\newtheorem{corollary}[theorem]{Corollary}

\theoremstyle{definition}
\newtheorem{definition}[theorem]{Definition}
\newtheorem{remark}[theorem]{Remark}
\newtheorem*{remark*}{Remark}

\newtheorem{assumption}{Assumption}
\newtheorem{example}[theorem]{Example}

\begin{document}

\title{Spectral inclusions of perturbed normal operators and applications}
\author{%
Javier Moreno
\thanks{Departamento de Matem\'aticas, Universidad de los Andes, Bogot\'a, Colombia.\newline jd.morenop@uniandes.edu.co},
Monika Winklmeier%
\thanks{Departamento de Matem\'aticas, Universidad de los Andes, Bogot\'a, Colombia.\newline mwinklme@uniandes.edu.co}
}
\maketitle

\begin{abstract}
   \noindent
   We consider a normal operator $T$ on a Hilbert space $H$.
   Under various assumptions on the spectrum of $T$, we give bounds for the spectrum of $T+A$
   where $A$ is $T$-bounded with relative bound less than 1 but we do not assume that $A$ is symmetric or normal.
   If the imaginary part of the spectrum of $T$ is bounded, then the spectrum of $T+A$ is contained in the region between two hyperbolas whose asymptotic slope depends on the $T$-bound of $A$.
   If the spectrum of $T$ is contained in a bisector, then the spectrum of $T+A$ is contained in the area between certain rotated hyperbolas.
   The case of infinite gaps in the spectrum of $T$ is studied.
   Moreover, we prove a stability result for the essential spectrum of $T+A$.
   If $A$ is even $p$-subordinate to $T$, then we obtain stronger results for the localisation of the spectrum of $T+A$.

  \smallskip


\end{abstract}

\section{Introduction} 
Perturbation theory of linear operator has a long history.
Vast parts of the literature are dedicated to the perturbation theory of selfadjoint linear operators
where the perturbation is symmetric and small in some sense with respect to the unperturbed operator, see for instance the classic monograph~\cite{Kato95}.
\smallskip

Recently however, there is growing interest in the spectral theory of non-selfadjoint operators.
Such operators may be used to model open physical systems where certain quantities are not conserved.
For instance, in \cite{riviere2019, Krejcirik2015} the Laplace operator on a quantum graph with non-selfadjoint boundary conditions is studied.
Here the unperturbed operator is not selfadjoint and not even similar to a selfadjoint operator.
Perturbations of non-selfadjoint operators have been studied for example in the influential books 
\cite{GohbergKrein69} by Gohberg and Krein, \cite{DunfordSchwartzII, DunfordSchwartzIII} by Dunford and Schwartz
and \cite{Markus1988} by Markus.  
The case of a normal discrete operator with a $p$-subordinate perturbation was studied by
Markus and Matsaev (\cite{MarkusMatsaev1981}, \cite[Theorem 6.12]{Markus1988}).
Wyss \cite{wyss2010} could extend their results by allowing for weaker hypotheses on the asymptotic behaviour of eigenvalues of the normal discrete operator.
Shkalikov also studied $p$-subordinate perturbations of normal operators in \cite{shkalikov2016perturbations} where he considered other kinds of perturbations, for example relatively bounded perturbations with relative bound equal to zero.
\smallskip

Analytic bounds for the spectrum of non-selfadjoint operators are valuable because it is in general difficult to obtain good numerical bounds.
In this paper we will study the spectra of perturbations of normal operators where the perturbation is either relatively bounded or $p$-subordinate.

\begin{definition}
   \label{def:relbounded}
   Let $T$ and $A$ be linear operators on a Banach space. We say that $A$ is \define{$T$-bounded} or \define{relatively bounded with respect to $T$} if $\mD(T)\subseteq \mD(A)$ and if there exist $a,b\geq 0$ such that
   \begin{align}
      \label{eq:definicion2}
      \|Ax\|^2\leq a^2\|x\|^2+b^2\|Tx\|^2,\qquad x\in \mD(T).
   \end{align}
   The infimum $\delta_A$ of all $b\geq 0$ such that \eqref{eq:definicion2} is satisfied for some $a\geq 0$ is called the $T$-\define{bound} of $A$.
\end{definition}

\begin{definition}
   Let $T$ and $A$ be linear operators on a Banach space.
   We say that $A$ is $p$\define{-subordinate} to $T$ for some $p\in [0,1]$ if $\mD(T)\subseteq \mD(A)$ and if there exists $c\geq 0$ such that 
   \begin{align}
      \label{eq:defpsubord}
      \|Ax\|\leq c\|x\|^{1-p}\|Tx\|^{p},
      \qquad x\in \mD(T).
   \end{align}
   The infimum over all constants $c\geq 0$ such that \eqref{eq:defpsubord} holds is called the \define{$p$-subordination bound} of $A$ to $T$.
\end{definition}

\begin{remark}
   \label{rem:psub:relbounded}
   If $A$ is $p$-subordinate to $T$ with $p<1$, then $A$ is $T$-bounded with $T$-bound zero; if in addition $0\in \rho(T)$ and $q>p$, then $A$ is $q$-subordinate to $T$ (see, e.g., \cite[Theorem 12.3]{krasnoselski}).
\end{remark}

Cuenin and Tretter studied in \cite{CueninTretter2016} the case when the unperturbed operator $T$ is selfadjoint and the additive perturbation $A$ is $T$-bounded with relative bound less than 1.
The perturbation is not assumed to be symmetric so that $T+A$ is in general not selfadjoint or even symmetric.
The authors give sufficient conditions for gaps in the spectrum of $T$ to remain open under the perturbation.
They show the following localisation of the spectrum between two hyperbolas.

\begin{theoremCT}[see Theorem 2.1 in \cite{CueninTretter2016}]
   \label{thm:theorem21}
   Let $T$ be a selfadjoint operator on a Hilbert space $H$ and let $A$ be $T$-bounded with $T$-bound $<1$ and let $a,b \geq 0$, $b <1$, as in \eqref{eq:definicion2}.
   Then the following is true:
   \begin{enumerate}[label={\upshape(\roman*)}]

      \item 
      \label{item:theorem21:i}
      The spectrum of $T+A$ lies between two hyperbolas, more precisely,
      \begin{align*}
	 \left\{z\in\C: |\im z|^2> \frac{a^2+b^2|\re z|^2}{1-b^2}\right\} \subseteq \rho(T+A).
      \end{align*}

      \item 
      \label{item:theorem21:ii}
      If $T$ has a spectral gap $(\alpha_{T},\beta_{T})\subset \R$, i.e. $\sigma(T)\cap (\alpha_{T},\beta_{T})=\emptyset$, then
      \begin{align*}
	 \{z\in\C: \alpha_{T+A} <\re z < \beta_{T+A}\} \subseteq \rho(T+A)
      \end{align*}
      where $\alpha_{T+A}:=\alpha_{T}+\sqrt{a^2+b^2\alpha_{T}^2}\,,\ \beta_{T+A}:= \beta_T -\sqrt{a^2+b^2\beta_{T}^2}$.


      \item 
      \label{item:theorem21:iii}
      If $T$ has an essential spectral gap $(\alpha_{T},\beta_{T})\subset \R$, i.e. $\sigmaess(T)\cap (\alpha_{T},\beta_{T})=\emptyset$  with $\alpha_{T},\beta_{T}\in \sigmaess(T)$, and $\alpha_{T+A} \le \beta_{T+A}$ holds, then the set
      $$\sigma(T+A)\cap \{z\in\C: \alpha_{T+A} <\re z < \beta_{T+A}\}$$ 
      consists of at most countably many isolated eigenvalues of finite algebraic multiplicity which may accumulate at most at the points $\alpha_{T+A}$ and $\beta_{T+A}$.
   \end{enumerate}
\end{theoremCT}

Moreover, the following estimate for the norm of the resolvent operator is proved.

\begin{theoremCT}[\cite{CueninTretter2016}, Proposition 2.8]
   \label{p:proposition28}
   Let $T$ be a selfadjoint operator in a Hilbert space $H$, let $A$ be $T$-bounded with $T$-bound $<1$ and let $a,b \geq 0$, $b <1$, as in \eqref{eq:definicion2}. 
   Then the following statements hold: 
   \begin{enumerate}[label={\upshape(\roman*)}]
      \item For $z\in \C$ such that $|\im z|^2>\frac{a^2+b^2|\re z|^2}{1-b^2}$, we have that
      \begin{equation}
	 \label{eq:prop28:i}
	 \|(T+A-z)^{-1}\|\leq \frac{1}{|\im z|-\sqrt{a^2+b^2|z|^2}}.
      \end{equation}

      \item If $(\alpha_T,\beta_T)$ is a spectral gap of $T$ then for $z\in \C$ such that  $\alpha_{T+A}<\re z<\beta_{T+A}$ with $\alpha_{T+A}$ and $\beta_{T+A}$ as in Theorem~\ref{thm:theorem21}, we have that 
      \begin{align}
	 \label{eq:prop28:ii}
	 \|(T+A-z)^{-1}\|
	 \leq 
	 \frac{
	 \left( \sqrt{\min\{\re z-\alpha_T,\beta_T-\re z\}^2+(\im z)^2} \right)^{-1}
	 }{1-\max\left\{b,~\frac{\sqrt{a^2+b^2\alpha_T^2}}{\mu-\alpha_T},~\frac{\sqrt{a^2+b^2\beta_T^2}}{\beta_T-\mu}\right\}}.
      \end{align}
   \end{enumerate}
\end{theoremCT}
The case of infinitely many gaps in the spectrum is also considered.
\smallskip

In the present work however we assume that the unperturbed operator is only normal instead of selfadjoint which makes the situation more complicated and appropriate conditions on its spectrum have to be identified that allow us to establish perturbational results about its spectrum.
In Proposition~\ref{prop:imupperbdd} we assume that $T$ is normal and that the imaginary part of the spectrum of $T$ is bounded.
Then the spectrum of $T+A$ is contained in the region between two hyperbolas which are shifted away from the real axis.
If in addition we assume that the resolvent set of $T$ contains a strip parallel to the imaginary axis, we can provide a condition under which the resolvent set of $T+A$ still contains a (possibly smaller) strip and we give estimates for $\|(T+A-z)^{-1}\|$ for $z$ in that strip, see Theorem~\ref{thm:bddstripgap}.
If the bound on the imaginary part of the spectrum of $T$ is 0, we recover the spectral inclusion from Theorem~\ref{thm:theorem21} and the estimate for the resolvent operator, see \eqref{eq:resestimupperbdd}.
If the real part of the spectrum of $T$ is lower bounded, then we show in Corollary~\ref{cor:ourconjecture1} that the real part of the spectrum of $T+A$ remains lower bounded, which is a generalisation of Corollary~2.4 in \cite{CueninTretter2016}.

We can even allow for normal operators with unbounded imaginary part of the spectrum.
We show in Theorem~\ref{thm:bisecgap} that if the spectrum of the normal operator $T$ is contained in two sectors symmetric around the real axis, then the spectrum of $T+A$ is contained in two possibly bigger sectors and some neighbourhood of $0$ if the perturbation $A$ is small enough.
More precisely, the exterior of certain shifted and rotated hyperbolas is contained in the resolvent set of $T+A$.
If we set the opening angles of the sectors equal to $0$, we again recover the results from Theorem~\ref{thm:theorem21}.
We can also show that a spectral free strip

\begin{equation}
   \label{eq:spectralgapgeneral}
   \{\lambda\in\C : \alpha\le \re(\lambda) \le \beta\} \subseteq \rho(T)
\end{equation}
of the unperturbed operator leads to a possibly smaller spectral free strip of $T+A$ if the perturbation is small enough.
That the strips under consideration are parallel to the real or the imaginary axis is a choice of convenience.
Clearly, if $T$ is normal, then so is $e^{\I\phi}T$ for any $\phi \in\R$ and we could as well consider other infinite strips which are not necessarily parallel to the axes.
\medskip

The paper is organised as follows.
In Section~\ref{sec:perturbation} we use a Neumann series argument and the spectral theorem to establish a basic criterion for $z$ to belong to $\rho(T+A)$ and we give estimates for $\|(T+A-z)^{-1}\|$ for such $z$.
An important ingredient here is that 
$\|(T-z)^{-1}\| = \dist(z,\sigma(T))^{-1}$ for all $z\in\rho(T)$ since $T$ is a normal operator.
Moreover we prove a perturbation theorem for the essential spectrum of normal operators.
In Section~\ref{sec:boundedimaginary} we consider normal operators whose imaginary part is bounded. 
Under a relatively bounded perturbation, the imaginary part of the spectrum may become unbounded but it is contained in the region between two shifted hyperbolas.
Theorem~\ref{thm:bddstripgap} shows that a spectral free strip of $T$ may give  rise to spectral free strip of $T+A$ if the perturbation $A$ is small enough.
In Section~\ref{sec:sectorial} we investigate normal (bi)sectorial-like operators. 
Our main theorem, Theorem~\ref{thm:bisecgap}, shows that the spectrum of $T+A$ away from 0 is contained in two sectors if the relative bound of $A$ is small enough. 
The amount by which the opening angle of the new sector increases depends on the $T$-bound of $A$.
Moreover, we give sufficient conditions for a spectral free strip around the imaginary axis to remain open.

The special case when the spectrum of $T$ is contained in only one sector has implications for semigroup generator properties of $T_A$ which will be dealt with in more detail in a forthcoming paper.
In Section~\ref{sec:misc} we apply our method to several other situations. 
Most notably we study the case when infinitely many strips parallel to the imaginary axis belong to the resolvent set of $T$, see Theorems~\ref{thm:infgapsbdd} and \ref{thm:infgaps}.
In Section~\ref{sec:psub} we consider $p$-subordinate perturbations.
In this case, we can prove spectral enclosures even if the spectrum of $T$ is bounded only by parabolas as opposed to linear functions in the sectorial case.
Section~\ref{sec:applications} contains applications to a non-selfadjoint quantum graph and to a first order system of differential equations.
Several estimates that are used throughout this work are colleted in Appendix~\ref{app:estimates}.


\bigskip
\noindent
{\bf Acknowledgement.}
The authors are grateful to Lyonell Boulton for fruitful discussions.
\bigskip

\section{Perturbation of normal operators} 
\label{sec:perturbation}
In this section we will establish our notation and provide some basic facts on normal operators.
The main result in this section is Proposition~\ref{prop:ourfirsprop3} on which most of our results are based.
Recall that a closed densely defined operator $T$ on a Hilbert space $H$ is called \define{normal} if $TT^* = T^*T$.
Note that for a normal operator $\mD(T)= \mD(T^*)$ and $\|Tx\| = \|T^*x\|$ for all $x\in\mD(T)$.
As for selfadjoint operators, there is a spectral theorem for normal operators:
There exists a unique projection valued measure $E$ supported on the spectrum $\sigma(T)$ of $T$ such that
\begin{equation*}
   T = \int_{\sigma(T)} \lambda\, \rd E(\lambda),
\end{equation*}
see for instance, \cite{Schmuedgen2012} or \cite{vanneerven}.
An important consequence which we will use frequently is that
\begin{equation*}
   \| (T-\lambda)^{-1} \| = \frac{1}{\dist(\lambda, \sigma(T))} .
\end{equation*}

In this work, the operators $T$ and $A$ and the constants $a,b, c$ will always satisfy one of the following assumptions.

\begin{assumption}
   \label{ass:relbounded}
   $T$ is a closed normal operator on a Hilbert space $H$ and $A$ is $T$-bounded with $T$-bound $\delta_A <1$ and the constants $a\ge 0$ and $0\le b \leq \delta_A$ satisfy \eqref{eq:definicion2}.
\end{assumption}

\begin{assumption}
   \label{ass:psubord}
   $T$ is a closed normal operator on a Hilbert space $H$ and $A$ is $p$-subordinate to $T$ for some $p\in [0,1]$ and the constant $c>0$ satisfies \eqref{eq:defpsubord}.
\end{assumption}

\noindent
Note that under either of the assumptions the perturbed operator $T+A$ is closed and that Assumption~\ref{ass:psubord} is stronger than Assumption~\ref{ass:relbounded} by Remark~\ref{rem:psub:relbounded}.
\smallskip

Recall that the resolvent set $\rho(T)$ of a closed linear operator consists of all $z\in\C$ such that $T-z$ is bijective.
In this case, the resolvent operator $(T-z)^{-1}$ is bounded. 
Much of perturbation theory is based on the well-known fact that $\id + S$ is bijective and $\|(I+S)^{-1}\|\le \frac{1}{1-\|S\|}$ if $S$ is a bounded linear operator with $\|S\|< 1$.
This allows us to prove the next Proposition which is fundamental for most of our results.

\begin{proposition}
   \label{prop:ourfirsprop3}
   Let $T$ and $A$ be as in Assumption~\ref{ass:relbounded}.
   Then for every $z \in \rho(T)$ the following holds.
   \begin{enumerate}[label={\upshape (\roman*)}]
      \item \label{item:ourfirstprop3:i}
      If $\displaystyle \|A(T-z)^{-1}\| < 1$, then $T+A-z$ is invertible and
      \begin{equation*}
	 \| (T+A-z)^{-1} \| \le  \frac{ \| (T-z)^{-1} \| }{ 1 - \| A(T-z)^{-1} \|} .
      \end{equation*}

      \item \label{item:ourfirstprop3:ii}
      $\displaystyle \|A(T-z)^{-1}\|\leq \sup_{t\in\sigma(T)} \frac{\sqrt{a^2+b^2|t|^2}}{|t-z|}.$

      \item \label{item:ourfirstprop3:iii}
      If $\displaystyle \sup_{t\in\sigma(T)} \frac{\sqrt{a^2+b^2|t|^2}}{|t-z|} <1$, 
      then $T+A-z$ is invertible and
      \begin{equation*}
	 \| (T+A-z)^{-1} \| \le  \frac{ \| (T-z)^{-1} \| }{ 1 - \| A(T-z)^{-1} \|} 
	 \le \frac{1}{\dist(z, \sigma(T))}  \frac{1}{ 1 - \sup_{t\in\sigma(T)} \frac{\sqrt{a^2+b^2|t|^2}}{|t-z|}} .
      \end{equation*}
   \end{enumerate}
\end{proposition}
\begin{proof}

   \ref{item:ourfirstprop3:i} is well known and follows from the Neumann series for the resolvent:
   If $z\in\rho(T)$, we can write $T+A-z = (\id + A (T-z)^{-1})(T-z)$.
   Therefore, if $\| A (T-z)^{-1}\| < 1$, then $T+A-z$ is invertible with inverse
   \begin{equation*}
      (T+A-z)^{-1} = (T-z)^{-1} \sum_{n=0}^\infty (-A (T-z)^{-1})^n
   \end{equation*}
   and 
   \begin{equation*}
      \|T+A-z)^{-1}\| \le \| (T-z)^{-1}\| \sum_{n=0}^\infty \|A (T-z)^{-1}\|^n
      =  \frac{\| (T-z)^{-1}\|}{ 1-\|A (T-z)^{-1}\|}.
   \end{equation*}
   \smallskip

   \ref{item:ourfirstprop3:ii}
   Since $T$ is normal, we have $\||T|x\| = \|Tx\|$ for all $x\in \mD(T) = \mD(|T|)$.
   Using that $A$ is $T$-bounded, we obtain from \eqref{eq:definicion2} for all $x\in \mD(|T|^2)$ and $z\in\rho(T)$ that 
   \begin{align*}
      \|Ax\|^2&\leq a^2\|x\|^2 +b^2\||T|x\|^2
      =a^2\langle x,\, x\rangle +b^2\langle |T|x,\, |T|x\rangle
      \\
      &= \langle (a^2+b^2|T|^2)x,\, x\rangle
      \\
      &=\|\sqrt{a^2+b^2|T|^2}\, x\|^2.
   \end{align*}
   Therefore, for all $x\in\mD(T)$ with $\|x\|=1$,
   \begin{align}
      \label{eq:ATresolventestimate}
      \|A(T-z)^{-1}x\|
      \leq \|\sqrt{a^2+b^2|T|^2}(T-z)^{-1} x\|
      \le 
      \sup_{t\in\sigma(T)} \frac{\sqrt{a^2+b^2|t|^2}}{|t-z|}
   \end{align}
   where the last inequality follows from the spectral theorem for normal operators.
   Since $A(T-z)^{-1}$ is bounded on the dense subset $\mD(T)$, it extends uniquely to a bounded operator on $H$ 
   and \eqref{eq:ATresolventestimate} holds for all $x\in H$.
   \smallskip

   \ref{item:ourfirstprop3:iii}
   follows from 
   \ref{item:ourfirstprop3:i} and \ref{item:ourfirstprop3:ii}.
\end{proof}
The proposition shows that in order to establish $z\in\rho(T+A)$ for a given $z\in\rho(T)$, it is crucial to find good estimates for 
$\dist(z, \sigma(T))$ and for the supremum of the function
\begin{equation}
   \label{eq:Hdef}
   H_z(t) := \frac{a^2 + b^2|t|^2 }{ |t-z|^2 } 
\end{equation}
where the supremum is taken over $t\in\sigma(T)$.
We show in Lemma~\ref{lem:nolocalextrema} that for fixed $z$, the function $H_z$ has no local maximum in open subsets of $\C\setminus\{z\}$.
Therefore it suffices to take the supremum over the boundary of $\sigma(T)$ or any set containing $\sigma(T)$ which may include the point at infinity.
In Appendix~\ref{app:estimates} we give bounds for $\sup_{t\in U} H_z(t)$ for various sets $U$ which then lead to the characterisation of regions contained in the resolvent set of $T+A$.
\bigskip

We finish this section with the following theorem about the essential spectrum\footnote{%
It should be noted that there are many different definitions in the literature for the essential spectrum of a non-selfadjoint operator; see for instance \cite{EdmundsEvans}. The definition we use corresponds to $\sigma_{e1}$ in \cite{EdmundsEvans}.
}
which for a closed operator $S$ is defined in \cite[Sec. XVII.2]{GohGolMarKaa90} as
\begin{align*}
   \sigmaess(S) := \{\lambda\in\C : S-\lambda\ \text{ is not a Fredholm operator} \}.
\end{align*}
If $U$ is an open connected set in $\C\setminus\sigmaess(S)$ which contains at least one point from $\rho(S)$, then
$\sigma(S)\cap U$ consists of at most countably many eigenvalues of $S$ of finite type which cannot accumulate in $U$, see \cite[Chapter~XVII, Theorem~2.1]{GohGolMarKaa90}.

\begin{theorem}
   \label{thm:essentialgaps}
   Let $T$ and $A$ and the constants $a,b$ be as in Assumption~\ref{ass:relbounded}.
   Assume that there exists a subspace $\widetilde H\subseteq \mD(T)$ such that 
   $\dim \widetilde H < \infty$, 
   $\mD(T) = \widetilde H \oplus \mD_2$ where $\mD_2 = \mD(T) \cap \widetilde H^\perp$
   and $\widetilde H$ and $\mD_2$ are $T$-invariant so that we can write $T = T_1\oplus T_2$ where $T_1$ is defined on $\widetilde H$ and $T_2$ on $\mD_2 \subseteq \widetilde H^\perp$.
   Let $P_1$ and $P_2$ be the orthogonal projections onto $\widetilde H$ and $\widetilde H^\perp$ respectively.
   Then $A_{22} := P_2 A|_{\mD(T)\cap\widetilde H^\perp}$, viewed as an operator on $\widetilde H^\perp$, is $T_2$-bounded with the same constants $a,b$ and
   \begin{equation}
      \label{eq:sigmaess}
      \sigmaess(T+A) = \sigmaess(T_2+A_{22}).
   \end{equation}
   In any open connected subset of $\C\setminus\sigmaess(T+A)$ which contains at least one point from the resolvent set of $T+A$, there are at most countably many eigenvalues of $T+A$ and they may accumulate only at its boundary.
\end{theorem}

\begin{proof}
   Since $\mD(T)\subseteq \mD(A)$, we can write 
   $\mD(T+A)=\mD(T_1)\oplus \mD(T_2) = \widetilde H \oplus \mD_2$ and, if we define 
   \begin{equation*}
      A_{ij}:=P_i A|_{\mD(T)\cap H_j},
      \qquad i,j=1,2,
   \end{equation*}
   we can represent both $T$ and $T+A$ as block operator matrices
   \begin{align*}
      T = 
      \begin{pmatrix} T_1 & 0 \\ 0 & T_2
      \end{pmatrix}
      : \mD(T_1)\oplus \mD(T_2) \subseteq H \longrightarrow \widetilde H \oplus \widetilde H^\perp,
   \end{align*}
   and 
   \begin{align*}
      T + A = 
      \begin{pmatrix} T_1  +  A_{11} & A_{12} \\ A_{21} & T_2 + A_{22}
      \end{pmatrix}
      : \mD(T_1)\oplus \mD(T_2) \subseteq H \longrightarrow \widetilde H \oplus \widetilde H^\perp.
   \end{align*}
   We decompose the sum as
   \begin{equation*}
      T+A=\mathcal{T}+\mathcal{A}
      \qquad\text{with }\quad
      \mathcal{T}:=
      \begin{pmatrix} T_1 & 0 \\ 0 & T_2+A_{22}
      \end{pmatrix},
      \quad 
      \mathcal{A}:= 
      \begin{pmatrix} A_{11} & A_{12} \\ A_{21} & 0
      \end{pmatrix}.
   \end{equation*}
   Note that for $j=1, 2$ and $x\in\mD(T_j)$ we have that
   \begin{equation*}
      \|A_{j1}x\|^2 + \|A_{j2}x\|^2
      = \|Ax\|^2 
      \le a^2 \|x\|^2 + b^2 \|Tx\|^2
      = a^2 \|x\|^2 + b^2 \|T_jx\|^2,
   \end{equation*}
   hence $A_{j\ell}$ is $T_j$-bounded with relative bound less or equal to the $T$-bound of $A$.
   Since $\widetilde H$ is finite dimensional, an easy calculation shows that $\mathcal{A}$ is $\mathcal{T}$-compact.
   Thus, by \cite[Chapter~XVII, Theorem~4.3]{GohGolMarKaa90},
   \begin{align}
      \label{eq:espectroess}
      \sigmaess(T+A)=\sigmaess(\mathcal{T}+\mathcal{A})=\sigmaess(\mathcal{T}).
   \end{align}
   Observe that for $\lambda \in \C$
   \begin{align*}
      \ker(\mathcal{T}-\lambda)
      & =\ker(T_1-\lambda)\oplus \ker(T_2+A_{22} - \lambda),
      \\
      \Ran(\mathcal{T}-\lambda)
      & =\Ran(T_1-\lambda)\oplus \Ran(T_2+A_{22} - \lambda).
   \end{align*}
   Since $\ker(T_1-\lambda)$ and $\Ran(T_1-\lambda)$ are finite dimensional, we have that 
   \begin{align*}
      \dim(\ker(\mathcal{T}-\lambda))< \infty
      & \quad \Longleftrightarrow \quad 
      \dim(\ker(T_2+A_{22}-\lambda))< \infty ,
      \\
      \codim(\Ran(\mathcal{T}-\lambda))< \infty 
      & \quad\Longleftrightarrow\quad 
      \codim(\Ran(T_2+A_{22}-\lambda))<\infty,
   \end{align*}
   and therefore
   \begin{align}
      \label{eq:espectroess2}
      \sigmaess(\mathcal{T})=\sigmaess(T_2+A_{22})
   \end{align}
   which together with \eqref{eq:espectroess} shows \eqref{eq:sigmaess}.
   The claim about nonaccumulation of the spectrum follows from \cite[Chapter~XVII, Theorem~2.1]{GohGolMarKaa90}.
\end{proof}

Let us remark that we need to assume conditions on $\sigma(T)$ in order to be able to control $\sigma(T+A)$.
For instance, the condition \eqref{eq:spectralgapgeneral} in general will not be enough to guarantee that a spectral gap remains open under a perturbation even in the case of an arbitrarily large gap and a $p$-subordinate perturbation.

\begin{example}
   \label{ex:gapcloses}
   Let $H=L^2(\R;\C)\oplus L^2(\R;\C)$.
   For $k>0$ we define the operator $T$ by 
   \begin{equation*}
      \mD(T):= \left\{
      \begin{pmatrix}
	 f \\ g
      \end{pmatrix}
      \in L^2(\R;\C)\oplus L^2(\R;\C): xf(x),~xg(x)\in L^2(\R;\C) \right\}
   \end{equation*}
   and 
   \begin{equation*}
      T \begin{pmatrix} f \\ g
      \end{pmatrix}
      := 
      \begin{pmatrix}
	 \mathcal{M}_{k+\I x} & 0 \\
	 0 & \mathcal{M}_{-k+\I x}
      \end{pmatrix}
      \begin{pmatrix} f \\ g
      \end{pmatrix}
      = 
      \begin{pmatrix} (k+\I x)f\\ (-k+\I x)g
      \end{pmatrix}
      \qquad
      \begin{pmatrix} f \\ g
      \end{pmatrix} \in \mD(T).
   \end{equation*}
   It is clear that $T$ is a normal operator with spectrum
   \begin{equation*}
      \sigma(T)=\left\{k+\I x:x\in \R\right\}\cup\left\{-k+\I x:x\in \R\right\}.
   \end{equation*}
   In particular, $T$ has a spectral free strip 
   \begin{equation*}
      \left\{z\in \C: -k<\re z<k\right\}\subseteq \rho(T).
   \end{equation*}
   Consider the operator $A$ defined by 
   \begin{equation*}
      A
      \begin{pmatrix} f \\ g
      \end{pmatrix}
      := 
      \begin{pmatrix}
	 \mathcal{M}_{-\sqrt{|x|}} & 0 \\
	 0 & \mathcal{M}_{\sqrt{|x|}}
      \end{pmatrix}
      \begin{pmatrix} f \\ g
      \end{pmatrix}
      = 
      \begin{pmatrix}
	 -\sqrt{|x|}f \\ \sqrt{|x|}g
      \end{pmatrix},
      \qquad
      \begin{pmatrix} f \\ g
      \end{pmatrix}
      \in \mathcal \mD(A) := \mD(T).
   \end{equation*}
   It is easy to see that $A$ is $p$-subordinate to $T$ for $p=\frac{1}{2}$, in particular $A$ is also $T$-bounded with relative bound $0$.
   However, the spectrum of $T+A$ is
   \begin{equation*}
      \sigma(T+A)
      =\left\{ k-\sqrt{|x|} + \I x : x\in \R\right\}\cup\left\{ -k+\sqrt{|x|} + \I x : x\in \R\right\},
   \end{equation*}
   hence $T+A$ does not have any spectral free strip.  
\end{example}


\section{Normal operators with bounded imaginary part} 
\label{sec:boundedimaginary}

Except for Proposition~\ref{prop:imupperbdd} we will always assume in this section that the imaginary part of the spectrum of the normal operator $T$ is bounded.

For $\gamma\in\R$ we define the sets
\begin{align}
   \label{eq:defHyper}
   \hyperbola_{\gamma} &= 
   \left\{ z\in\C :
   \frac{ a^2 + b^2\gamma^2 }{1-b^2}
   + 
   \frac{ b^2 (\re z )^2 }{1-b^2}
   < ( \im z )^2,\
   \im z > 0
   \right\}.
\end{align}
$\hyperbola_{\gamma}$ is the set of all points above the hyperbola in the complex which passes through $\I \sqrt{\frac{ a^2 + b^2\gamma^2 }{1-b^2} }$ and has the asymptotes $\im z = \pm \frac{b}{\sqrt{1-b^2}} \re z$ for $\re z \to\pm \infty$.

\begin{proposition}
   \label{prop:imupperbdd}
   Let $T$ and $A$ be as in Assumption~\ref{ass:relbounded}.
   Moreover, we assume that $\im\sigma(T)\le \gamma$ for some $\gamma\in\R$.
   Then
   \begin{equation}
      \label{eq:imupperbdd}
      \I\gamma + \hyperbola_\gamma
      \subseteq \rho(T+A).
   \end{equation}
   For $z\in \I \gamma + \hyperbola_\gamma$ we have the resolvent estimate
   \begin{equation}
      \label{eq:resestimupperbdd}
      \|(T+A-z)^{-1}\|
      \leq \frac{1}{\im z -\gamma -\sqrt{a^2 + 2b^2\gamma^2 - 2b^2 \gamma \im z +b^2|z|^2}}.
   \end{equation}
\end{proposition}
\begin{proof}
   Let $z=\mu + \I\nu\in\C$ with $\nu > \gamma$.
   Then $z\in\rho(T)$ and it belongs to $\rho(T+A)$ if 
   $\sup_{\im t \le \gamma} H_z(t) < 1$.
   By Lemma~\ref{lem:nolocalextrema}, $H_z$ has no local extrema, so we obtain
   \begin{equation}
      \label{eq:imbddabove}
      \sup_{\im t \le \gamma} H_z(t)
      = \sup_{\im t = \gamma} H_z(t)
      = \sup_{\im t = \gamma} \frac{a^2 + b^2 |t|^2}{|t-z|^2} 
      \le b^2 + \frac{b^2 \mu^2}{(\nu - \gamma)^2} +  \frac{a^2 + b^2 \gamma^2}{(\nu - \gamma)^2}
   \end{equation}
   where the last inequality follows from \eqref{eq:distanceline} in Lemma~\ref{lem:app:Hest}.
   Clearly, the right hand side is smaller than 1 if and only if 
   \begin{equation*}
      \frac{b^2\mu^2}{1-b^2} + \frac{a^2 + b^2\gamma^2}{1-b^2} 
      < (\im z -\gamma)^2,
   \end{equation*}
   which is true if and only if $z\in\I\gamma + \hyperbola_\gamma$.
   The resolvent estimate \eqref{eq:resestimupperbdd}
   follows from Proposition~\ref{prop:ourfirsprop3}\ref{item:ourfirstprop3:iii}
   together with \eqref{eq:imbddabove} and 
   $\|(T-z)^{-1}\|^2 = \frac{1}{\dist(z,\sigma(T))^2}\leq \frac{1}{(\im z -\gamma)^2}$.
\end{proof}

For the rest of this section we assume that the imaginary part of the spectrum of $T$ is bounded.
Our next proposition shows that in this case the spectrum of the perturbed operator $T+A$ is confined between two shifted hyperbolas whose asymptotic slopes $\pm \frac{b}{\sqrt{1-b^2}}$ depend only on the $T$-bound of $A$, 
while the amount of the shift depends also on the bound for the imaginary part of $\sigma(T)$.
Note that if $\gamma_1 = \gamma_2 =0$ in the Proposition below, then our spectral inclusion coincides with 
the one for selfadjoint operators proved in \cite[Theorem 2.1]{CueninTretter2016}.

\begin{proposition}  
   \label{prop:bddstrip}
   Let $T$ and $A$ be as in Assumption~\ref{ass:relbounded}.
   If
   \begin{equation}
      \label{eq:assumptionboundedim}
      \sigma(T)\subseteq \{ z\in\C : \gamma_1 \le \im z \le \gamma_2\}.
   \end{equation}
   for some $\gamma_1 < \gamma_2 \in \R$.
   Set $\wgamma = \max\{ |\gamma_1|,\, |\gamma_2|\}$.
   Then
   \begin{align}
      \label{eq:hiperbolas1}
      (\I\gamma_1 - \hyperbola_{\wgamma})
      \cup
      (\I\gamma_2 + \hyperbola_{\wgamma})
      \subseteq \rho(T+A)
   \end{align}
   and for all $z$ in the set on the left hand side of \eqref{eq:hiperbolas1} we have the resolvent estimate
   \begin{align}
      \label{eq:resestimateboundedim}
      \|(T+A-z)^{-1}\|
      \le
      \frac{1}{ |\im z-\gamma| - \sqrt{ 
      b^2( |\im z-\gamma|^2 + (\re z)^2 ) + a^2 + b^2 \wgamma^2} }
   \end{align}
   with
   $\gamma = \gamma_1$ if $\im z < \gamma_1$ and $\gamma = \gamma_2$ if $\im z > \gamma_2$,
   see Figure~\ref{fig:figura3}.
\end{proposition}
\begin{proof} 
   Let $z = \mu + \I\nu\in \C$ with $\nu =\im z \in \R\setminus [\gamma_1, \gamma_2]$ so that  $z\in \rho(T)$.
   Note that by \eqref{eq:distancestrip} in Lemma~\ref{lem:app:Hest} we have
   \begin{align}
      \label{eq:bddstrip}
      \nonumber
      \sup_{t\in\sigma(T)} \frac{ a^2 + b^2 |t|^2}{|t-z|^2} 
      & \le \sup_{t\in \R + \I[\gamma_1, \gamma_2]} \frac{ a^2 + b^2 |t|^2}{|t-z|^2} 
      \le b^2 + \max_{\gamma\in\{\gamma_1, \gamma_2\}} 
      \frac{ b^2 \mu^2}{(\nu - \gamma)^2} + \frac{ a^2 + b^2 \gamma^2}{(\nu - \gamma)^2}
      \\
      & \le b^2 + \max_{\gamma\in\{\gamma_1, \gamma_2\}} 
      \frac{ b^2 \mu^2}{(\nu - \gamma)^2} + \frac{ a^2 + b^2 \wgamma^2}{(\nu - \gamma)^2}.
   \end{align}
   If $\nu < \gamma_1$ then the maximum is attained for $\gamma=\gamma_1$ 
   and if $\nu > \gamma_2$, then it is attained for $\gamma=\gamma_2$.
   Therefore, the right hand side of \eqref{eq:bddstrip} is less than 1 if and only if 
   \begin{equation*}
      (\nu - \gamma)^2 > 
      \frac{ b^2}{1-b^2} \mu^2  + \frac{ a^2 + b^2 \wgamma^2}{1-b^2}
      \qquad\text{with}\quad
      \begin{cases}
	 \gamma = \gamma_1\ &\text{if } \nu < \gamma_1,\\
	 \gamma = \gamma_2\ &\text{if } \nu > \gamma_2,
      \end{cases}
   \end{equation*}
   hence \eqref{eq:hiperbolas1} is proved.
   Finally the resolvent estimate \eqref{eq:resestimateboundedim}
   follows from Proposition~\ref{prop:ourfirsprop3}\ref{item:ourfirstprop3:iii}
   together with \eqref{eq:bddstrip} and 
   \begin{align*}
      \|(T-z)^{-1}\| &= \frac{1}{\dist(z,\sigma(T))}\leq \frac{1}{ |\nu-\gamma| }
      \qquad\text{with}\quad
      \begin{cases}
	 \gamma = \gamma_1\ &\text{if } \nu < \gamma_1,\\
	 \gamma = \gamma_2\ &\text{if } \nu > \gamma_2.
      \end{cases}
   \end{align*}
\end{proof}

Next we show that if the imaginary part of the spectrum of $T$ is bounded, then a gap in the spectrum of $T$ remains open if the perturbation is small enough.

\begin{theorem} 
   \label{thm:bddstripgap}
   Let $T$ and $A$ be as in Assumption~\ref{ass:relbounded}.
   As in Proposition~\ref{prop:bddstrip} we assume that there exist $\gamma_1 < \gamma_2\in\R$ such that 
   $\sigma(T)\subseteq \R + \I[\gamma_1, \gamma_2]$.
   In addition, we assume that $T$ has a spectral free strip 
   \begin{equation*}
      \{z\in\C: \alpha_{T} <\re z < \beta_{T}\} \subseteq \rho(T).
   \end{equation*}
   Set $\wgamma := \max\{|\gamma_1|, |\gamma_2|\}$
   and 
   \begin{align}
      \label{eq:extremosdelatira2}
      \alpha'_{T+A}:=\alpha_{T}+\sqrt{a^2+b^2\wgamma^2+b^2\alpha_{T}^2}\, ,\qquad
      \beta'_{T+A}:= \beta_T -\sqrt{a^2+b^2\wgamma^2+b^2\beta_{T}^2}\, .
   \end{align}
   If $\alpha_{T+A}' < \beta_{T+A}'$, then
   $T+A$ has a spectral free strip
   \begin{align}
      \label{eq:tira2}
      \{z\in\C: \alpha'_{T+A} <\re z < \beta'_{T+A}\} \subseteq \rho(T+A).
   \end{align}
   See Figure~\ref{fig:figura3}.
   Moreover, if $\alpha_{T+A}' \le \re z \le \beta_{T+A}'$, then
   \begin{equation}
      \label{eq:resestbddstrip}
      \|(T+A-z)^{-1}\|
      \leq \frac{\left(\sqrt{(\min\{\re z-\alpha_T,\beta_T-\re z\})^2+
      (\min\{0, \gamma_2-\im z, \im z - \gamma_1 \})^2}\right)^{-1}}%
      {1-\max\left\{
      b,\ \frac{\sqrt{a^2+b^2\wgamma^2+b^2\alpha_T^2}}{\re z-\alpha_T},\ \frac{\sqrt{a^2+b^2\wgamma^2+b^2\beta_T^2}}{\beta_T-\re z}\right\}}.
   \end{equation}
\end{theorem}

\begin{proof} 
   Let $z=\mu + \I\nu$ with $\mu \in (\alpha_{T},\beta_{T})$ and $\nu \in\R\setminus [\gamma_1, \gamma_2]$. 
   Then
   \begin{align}
      \nonumber
      \sup_{t\in\sigma(T)}
      \frac{a^2+b^2|t|^2}{|t-z|^2}
      &\le
      \sup_{\sigma\in [\gamma_1, \gamma_2] }
      \sup_{t\in(\R\setminus(\alpha_T, \beta_T)) + \I\sigma}
      \frac{a^2+b^2|t|^2}{|t-z|^2}
      \\
      \nonumber
      &\le 
      \sup_{\sigma\in [\gamma_1, \gamma_2] }
      \sup_{\tau\in \R\setminus(\alpha_T, \beta_T)}
      \frac{a^2+b^2(\tau^2+\sigma^2)}{(\tau-\mu)^2}
      \\
      \nonumber
      &= \sup_{\tau\in \R\setminus(\alpha_T, \beta_T)}
      \frac{a^2 + b^2\wgamma^2 + b^2 \tau^2}{(\tau-\mu)^2}
      \\
      \label{eq:eq3delcorolario2:5}
      & \leq \max\left\{b^2,~\frac{{a^2+b^2\wgamma^2+b^2\alpha_T^2}}{(\mu-\alpha_T)^2},~\frac{{a^2+b^2\wgamma^2+b^2\beta_T^2}}{(\beta_T-\mu)^2} \right\}
   \end{align}
   where the last inequality follows from Lemma~\ref{lem:app:lemma0} with $a^2+b^2\wgamma^2$ instead of $a^2$.
   Note that the expression in \eqref{eq:eq3delcorolario2:5} is less than 1 if
   $\alpha_{T+A}' < \re z  = \mu < \beta_{T+A}'$.
   Therefore \eqref{eq:tira2} follows from Proposition~\ref{prop:ourfirsprop3}\ref{item:ourfirstprop3:iii}.

   The resolvent estimate \eqref{eq:resestbddstrip} follows again from 
   Proposition~\ref{prop:ourfirsprop3}\ref{item:ourfirstprop3:iii}
   together with \eqref{eq:eq3delcorolario2:5} and 
   $\dist(z, \sigma(T)) \ge \sqrt{(\min\{\re z-\alpha_T,\beta_T-\re z\})^2+(\min\{0, \gamma_2-\im z, \im z - \gamma_1 \})^2}$.
\end{proof}
%
In Theorem~\ref{thm:EVmultiplicitybdd} below we extend Theorem~\ref{thm:bddstripgap} and allow for finitely many eigenvalues outside the strip.
We will also prove a stability result for the total algebraic multiplicity of these eigenvalues.

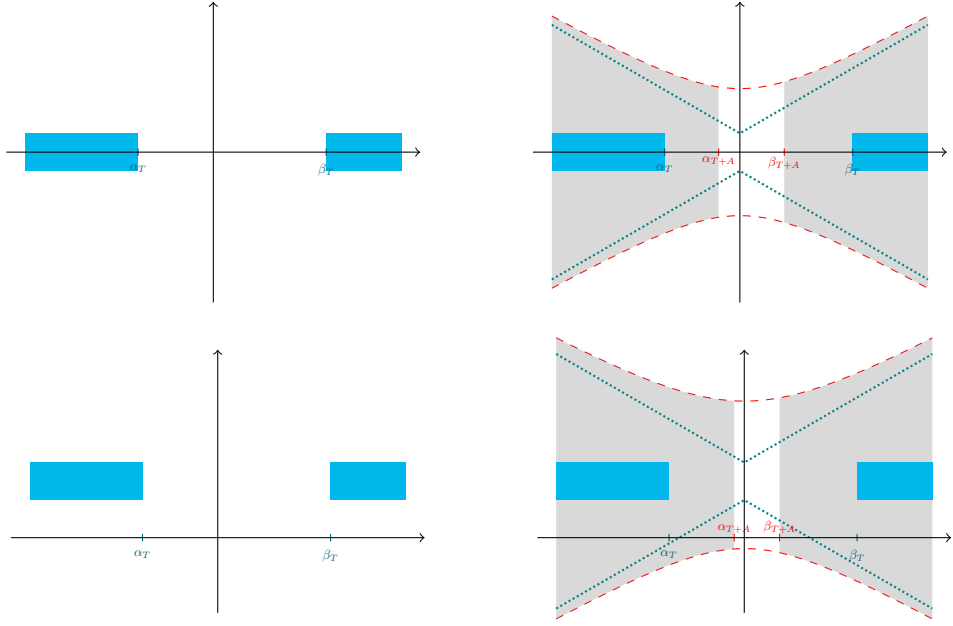
\begin{figure}
   \centering
   \begin{tikzpicture}[ 
      declare function={ hyperbola(\x) = \gT + sqrt( (\abound^2 + \bbound^2*\gT^2 + \bbound^2*\x*\x)/(1-\bbound^2) ); 
      hyperbolaright(\x) = \betaT - sqrt( (\abound^2 + \bbound^2*(\betaT)^2 + \bbound^2*\x*\x)/(1-\bbound^2) );  
      hyperbolaleft(\x) = \alphaT + sqrt( (\abound^2 + \bbound^2*(\alphaT)^2 + \bbound^2*\x*\x)/(1-\bbound^2) ); 
      } ]

      \tikzmath{
      \gT = .5;  
      \abound = 1;   
      \bbound = .5;  
      \alphaT = -2;      
      \betaT = 3;      
      \alphaTA = \alphaT + sqrt( \abound^2 + \bbound^2*\gT^2 + \bbound^2*(\alphaT)^2); 
      \betaTA = \betaT - sqrt( \abound^2 + \bbound^2*\gT^2 + \bbound^2*\betaT^2);      
      }

      \begin{scope}[scale=0.5, transform shape, xshift=-7cm] 

      \path [fill=colTspec, opacity=1] (-5, -\gT) rectangle (\alphaT, \gT);
      \path [fill=colTspec, opacity=1] (5, -\gT) rectangle (\betaT, \gT);

      \draw[colTspec!50!black] (\alphaT,-0.1) --  node[below=.2] {$\alpha_{T}$} (\alphaT,0.1);
      \draw[colTspec!50!black] (\betaT,-0.1) -- node[below=.2] {$\beta_{T}$} (\betaT,0.1);

      \draw[->] (-5.5,0)--(5.5,0);
      \draw[->] (0,-4)--(0,4);
      \end{scope}

      \begin{scope}[scale=0.5, transform shape, xshift=7cm] 
      \draw [red, dashed] plot[domain=-5.0:5, smooth] ({\x},{hyperbola(\x)});
      \draw [red, dashed] plot[domain=-5.0:5, smooth] ({\x},{-hyperbola(\x)});

      \fill [colTAspec, opacity=.3]
      plot[domain=-5.0:\alphaTA, smooth] ({\x},{hyperbola(\x)}) -- plot[domain=\alphaTA:-5.0, smooth] ({\x},{-hyperbola(\x)}) -- cycle;
      
      \fill [colTAspec, opacity=.3]
      plot[domain=5.0:\betaTA, smooth] ({\x},{hyperbola(\x)}) -- plot[domain=\betaTA:5.0, smooth] ({\x},{-hyperbola(\x)}) -- cycle;

      \draw [densely dotted, thick,teal] (-5, {\gT + 5*\bbound/sqrt(1-\bbound^2)}) -- (0,\gT) -- (5, {\gT + 5*\bbound/sqrt(1-\bbound^2)});
      \draw [densely dotted, thick,teal] (-5, {-\gT - 5*\bbound/sqrt(1-\bbound^2)}) -- (0,-\gT) -- (5, {-\gT-5*\bbound/sqrt(1-\bbound^2)});

      \path [fill=colTspec, opacity=1] (-5, -\gT) rectangle (\alphaT, \gT);
      \path [fill=colTspec, opacity=1] (5, -\gT) rectangle (\betaT, \gT);

      \draw[red] (\alphaTA,-0.1) --  node[below] {$\alpha_{T+A}$} (\alphaTA,0.1);
      \draw[red] (\betaTA,-0.1) -- node[below] {$\beta_{T+A}$} (\betaTA,0.1);
      \draw[colTspec!50!black] (\alphaT,-0.1) --  node[below=.2] {$\alpha_{T}$} (\alphaT,0.1);
      \draw[colTspec!50!black] (\betaT,-0.1) -- node[below=.2] {$\beta_{T}$} (\betaT,0.1);

      \draw[->] (-5.5,0)--(5.5,0);
      \draw[->] (0,-4)--(0,4);
      \end{scope}

   \end{tikzpicture}
   \bigskip

   \begin{tikzpicture}[ 
      declare function={ 
      hyperbolaupper(\x) = \gupper + sqrt( (\abound^2 + \bbound^2*(\gupper)^2 + \bbound^2*\x*\x)/(1-\bbound^2) ); 
      hyperbolalower(\x) = \glower - sqrt( (\abound^2 + \bbound^2*(\glower)^2 + \bbound^2*\x*\x)/(1-\bbound^2) );  
      hyperbolaright(\x) = \betaT - sqrt( (\abound^2 + \bbound^2*(\betaT)^2 + \bbound^2*\x*\x)/(1-\bbound^2) );  
      hyperbolaleft(\x) = \alphaT + sqrt( (\abound^2 + \bbound^2*(\alphaT)^2 + \bbound^2*\x*\x)/(1-\bbound^2) ); 
      } ]

      \tikzmath{
      \gupper = 2.0;  
      \glower = 1.0;  
      \gT = .5;  
      \wg = max(abs(\gupper), abs(\glower));  
      \abound = 1;   
      \bbound = .5;  
      \alphaT = -2;      
      \betaT = 3;      
      \alphaTA = \alphaT + sqrt( \abound^2 + \bbound^2*\wg^2 + \bbound^2*(\alphaT)^2); 
      \betaTA = \betaT - sqrt( \abound^2 + \bbound^2*\wg^2 + \bbound^2*\betaT^2);      
      }

      \begin{scope}[scale=0.5, transform shape, xshift=-7cm] 

      \path [fill=colTspec, opacity=1] (-5,\glower) rectangle (\alphaT, \gupper);
      \path [fill=colTspec, opacity=1] (5, \glower) rectangle (\betaT, \gupper);

      \draw[colTspec!50!black] (\alphaT,-0.1) --  node[below=.2] {$\alpha_{T}$} (\alphaT,0.1);
      \draw[colTspec!50!black] (\betaT,-0.1) -- node[below=.2] {$\beta_{T}$} (\betaT,0.1);

      \draw[->] (-5.5,0)--(5.5,0);
      \draw[->] (0,-2)--(0,5);
      \end{scope}

      \begin{scope}[scale=0.5, transform shape, xshift=7cm] 
      \draw [red, dashed] plot[domain=-5.0:5, smooth] ({\x},{hyperbolaupper(\x)});
      \draw [red, dashed] plot[domain=-5.0:5, smooth] ({\x},{hyperbolalower(\x)});
      
      \fill [colTAspec, opacity=.3]
      plot[domain=-5.0:\alphaTA, smooth] ({\x},{hyperbolaupper(\x)}) -- plot[domain=\alphaTA:-5.0, smooth] ({\x},{hyperbolalower(\x)}) -- cycle;
      
      \fill [colTAspec, opacity=.3]
      plot[domain=5.0:\betaTA, smooth] ({\x},{hyperbolaupper(\x)}) -- plot[domain=\betaTA:5.0, smooth] ({\x},{hyperbolalower(\x)}) -- cycle;

      \draw [densely dotted, thick,teal] (-5, {\gupper + 5*\bbound/sqrt(1-\bbound^2)}) -- (0,\gupper) -- (5, {\gupper + 5*\bbound/sqrt(1-\bbound^2)});
      \draw [densely dotted, thick,teal] (-5, {\glower - 5*\bbound/sqrt(1-\bbound^2)}) -- (0,\glower) -- (5, {\glower-5*\bbound/sqrt(1-\bbound^2)});

      \path [fill=colTspec, opacity=1] (-5,\glower) rectangle (\alphaT, \gupper);
      \path [fill=colTspec, opacity=1] (5, \glower) rectangle (\betaT,  \gupper);

      \draw[red] (\alphaTA,-0.1) --  node[above] {$\alpha_{T+A}$} (\alphaTA,0.1);
      \draw[red] (\betaTA,-0.1) -- node[above] {$\beta_{T+A}$} (\betaTA,0.1);

      \draw[colTspec!50!black] (\alphaT,-0.1) --  node[below=.2] {$\alpha_{T}$} (\alphaT,0.1);
      \draw[colTspec!50!black] (\betaT,-0.1) -- node[below=.2] {$\beta_{T}$} (\betaT,0.1);

      \draw[->] (-5.5,0)--(5.5,0);
      \draw[->] (0,-2)--(0,5);
      \end{scope}

   \end{tikzpicture}

   \caption{Spectral inclusion from Proposition~\ref{prop:bddstrip} and Theorem~\ref{thm:bddstripgap}.
   The left graphic shows the location of the spectrum of $T$, the gray area in the graphics on the right shows the spectral enclosure of $T+A$ proved in Proposition~\ref{prop:bddstrip} and Theorem~\ref{thm:bddstripgap}.
   In the upper row we have $\gamma_1 = -\gamma_2$, 
   in the lower row $0 < \gamma_1 < \gamma_2$.
   The upper red dashed line is the boundary of $\I\gamma_2 + \hyperbola_{\gamma_2}$,
   the lower red dashed line is the boundary of $\I\gamma_1 - \hyperbola_{\gamma_1}$.
   The green dotted lines are their asymptotes.
   }
   \label{fig:figura3}
\end{figure}

\begin{remark} 
   \label{prop:bddstripsymm}
   In the special case when $\gamma_2 = -\gamma_1 =: \gamma$, our formulas simplify slightly. 
   Proposition~\ref{prop:bddstrip} yields that
   \begin{multline}
      \label{eq:hiperbolas1sym}
      ( -\I\gamma-\hyperbola_{\gamma} )
      \cup ( \I\gamma+\hyperbola_{\gamma} )
      \\[1ex]
      = \left\{z\in\C:  
      \frac{a^2+b^2\left((\re z)^2+\gamma^2\right)}{1-b^2} < \left(|\im z|-\gamma\right)^2,\
      |\im z| > \gamma
      \right\}
      \subseteq \rho(T+A)
   \end{multline}
   and for all $z$ in the set on the left hand side of \eqref{eq:hiperbolas1sym} we have the resolvent estimate
   \begin{align}
      \label{eq:resestimateboundedimsym}
      \|(T+A-z)^{-1}\|
      \leq \frac{1}{|\im z|-\gamma -\sqrt{ a^2+2b^2\gamma^2-2b^2\gamma |\im z| +b^2|z|^2}}.
   \end{align}
   If in addition $(\alpha_{T},\beta_{T})+\I \R \subseteq\rho(T)$ and $\alpha'_{T+A} < \beta'_{T+A}$ (with $\gamma$ instead of $\wgamma$), then
   $T+A$ has a spectral free strip
   \begin{align}
      \label{eq:tira2sym}
      \{z\in\C: \alpha'_{T+A} <\re z < \beta'_{T+A}\} \subseteq \rho(T+A)
   \end{align}
   and for $\alpha_{T+A}' < \re z < \beta_{T+A}'$ we have the resolvent estimate
   \begin{equation}
      \label{eq:resestbddstripsym}
      \|(T+A-z)^{-1}\|
      \leq \frac{\left(\sqrt{(\min\{\re z-\alpha_T,\beta_T-\re z\})^2+
      (\min\{0, \gamma-|\im z|\})^2}\right)^{-1}}%
      {1-\max\left\{
      b,\ \frac{\sqrt{a^2+b^2\gamma^2+b^2\alpha_T^2}}{\re z-\alpha_T},\ \frac{\sqrt{a^2+b^2\gamma^2+b^2\beta_T^2}}{\beta_T-\re z}\right\}}.
   \end{equation}
   Note that for $\gamma=0$, the estimates \eqref{eq:resestimateboundedimsym} and \eqref{eq:resestbddstripsym}
   coincide with the estimates \eqref{eq:prop28:i} and \eqref{eq:prop28:ii}.
\end{remark}

If the spectrum of $T$ is bounded from below, then we can choose $\alpha_T$ in Theorem~\ref{thm:bddstripgap} arbitrarily small and the spectrum of $T+A$ is also bounded from below as the following corollary shows.
\begin{corollary}
   \label{cor:ourconjecture1}
   Let $T$ and $A$ be as in Assumption~\ref{ass:relbounded}.
   As in Proposition~\ref{prop:bddstrip} we assume that $\sigma(T)\subseteq \R + \I [\gamma_1,\gamma_2]$ for $\gamma_1 < \gamma_2 \in\R$ and we set $\wgamma := \max\{|\gamma_1|, |\gamma_2|\}$.
   \begin{enumerate}[label={\upshape(\roman*)}]

      \item 
      \label{item:semib}
      If $\re\sigma(T) \ge \beta_{T}$ for some $\beta_T\in\R$, then
      \begin{align*}
	 \re \sigma(T+A)\ge
	 \beta'_{T+A}:= \beta_T -\sqrt{a^2+b^2\wgamma^2+b^2\beta_{T}^2}\, .
      \end{align*}

      \item \label{item:semia}
      If $\re\sigma(T) \le \alpha_{T}$ for some $\alpha_T\in\R$, then
      \begin{align*}
	 \re \sigma(T+A)\le
	 \alpha'_{T+A}:= \alpha_T + \sqrt{a^2+b^2\wgamma^2+b^2\alpha_{T}^2}\, .
      \end{align*}

   \end{enumerate}
\end{corollary}
\begin{proof}
   We proof only \ref{item:semib}.
   Let $\beta_{T+A}'$ be as in \eqref{eq:extremosdelatira2}.
   Since the spectrum of $T$ is bounded from below, we can choose $\alpha_T$, and hence also $\alpha_{T+A}'$ from \eqref{eq:extremosdelatira2} arbitrarily small because $0\le b< 1$ and therefore
   \begin{equation*}
      \alpha_{T+A}'
      = \alpha_T+\sqrt{a^2+b^2\gamma_T^2+b^2\alpha_{T}^2}
      \ \longrightarrow -\infty
      \quad\text{for }\
      \alpha_T\to -\infty.
   \end{equation*}
   Hence \eqref{eq:tira2} shows that
   \begin{equation*}
      (-\infty,~\beta'_{T+A})+\I \R
      \subseteq \rho(T+A).
      \qedhere
   \end{equation*}
\end{proof}

Under the conditions from Corollary~\ref{cor:ourconjecture1}\ref{item:semia}, the spectrum of $T+A$ is contained in a sector.
In fact, if $\phi_b = \arctan( \frac{b}{\sqrt{1-b^2}})$, we can find an $\omega\in\R$ such that $\sigma(T+A)\subseteq -\sector{\phi_b}{\omega}$.
It is even true that $T+A-\omega$ is sectorial and therefore generates an analytic semigroup.
Semigroup generator properties of $T+A$ will be discussed in a forthcoming work.
\bigskip

The last two theorems in this section give a perturbation result for the essential spectrum and a stability result for the total algebraic multiplicity of eigenvalues in a gap of the essential spectrum.

\begin{theorem}
   \label{thm:essspecbdd}
   Let $T$ and $A$ be as in Assumption~\ref{ass:relbounded}.
   For
   $\alpha < \beta$  and 
   $\gamma_1 \le \gamma_2$ we define the gapped strip
   $\strip = \{ z\in\C : \re z \in \R\setminus (\alpha, \beta),\ \gamma_1 \le \im z \le \gamma_2 \}$.
   Assume that 
   $\sigmaess(T)\subseteq \strip$
   and 
   $\sigma(T) = \sigmaess(T) \cup \Sigma$ 
   where $\Sigma \subseteq\C\setminus\strip$ is a bounded countable set which may accumulate at most at the boundary of $\sigmaess(T)$
   Let $\alpha'_{T+A}$ and $\beta'_{T+A}$ as in \eqref{eq:extremosdelatira2}
   and set $\wgamma = \max\{ |\gamma_1|,\, |\gamma_2|\}$.
   If $\alpha_{T+A}' < \beta_{T+A}'$
   then 
   \begin{align}
      \label{eq:bddstripessspec}
      \begin{aligned}
	 \{z\in \C: \alpha'_{T+A} < \re z < \beta'_{T+A}\}
	 \cup
	 (\I\gamma_1 - \hyperbola_{\wgamma}) \cup (\I\gamma_2 + \hyperbola_{\wgamma})
	 \subseteq \C\setminus\sigmaess(T+A)
      \end{aligned}
   \end{align}
   and the set on the left hand side consists of at most 
   countably many isolated eigenvalues of finite algebraic multiplicity which may accumulate only at its boundary.
\end{theorem}

\begin{proof}
   For $n\in\N$ define the sets
   \begin{equation*}
      S_n := \left\{ z\in\C : \alpha+\frac{1}{n} \le \re z \le \beta-\frac{1}{n},\ \gamma_1 - \frac{1}{n} \le \im z \le \gamma_2 + \frac{1}{n} \right\}
   \end{equation*}
   and $U_n := \C\setminus S_n$.
   Then $\sigmaess(T)\subseteq S_n$ for every $n\in\N$ and $U_n\cap \sigma(T)$ is finite and contains only finite eigenvalues of finite multiplicity.
   Let $P_1$ be the spectral projection of $T$ for the set $U_n$ and set 
   $\widetilde H_n := \Ran(P_1) = \bigoplus_{\lambda\in U_n\cap \sigma(T)} \ker(T-\lambda)$.
   Set $\mD_2 = \mD(T)\cap \widetilde H_n^\perp$.
   Then $\dim \widetilde H_n < \infty$, $\mD(T) = \widetilde H_n \oplus \mD_2$
   and both $\widetilde H_n$ and $\widetilde H_n^\perp$ are $T$-invariant, see \cite[pag.~83]{Schmuedgen2012}).
   Let us define $T_j^{(n)}$ and $A_{ij}^{(n)}$ as in Theorem~\ref{thm:essentialgaps}.
   Then $A_{22}^{(n)}$ is $T_2^{(n)}$-bounded with the same constants $a,b$ that work for $T$ and $A$ and, by the spectral theorem, $\sigma(T_2^{(n)})\subseteq S_n$.
   If we apply Proposition~\ref{prop:bddstrip} to $T_2^{(n)}$ and $A_{22}^{(n)}$, we obtain that 
   \begin{align*}
      F_n & :=
      \textstyle
      (\I(\gamma_1 - \frac{1}{n} ) - \hyperbola_{\wgamma + \frac{1}{n} }) \cup 
      (\I(\gamma_2 + \frac{1}{n} ) + \hyperbola_{\wgamma + \frac{1}{n} })
      \subseteq \rho(T_2^{(n)} + A_{22}^{(n)}).
   \end{align*}
   Moreover, Theorem~\ref{thm:bddstripgap} applied to $T_2^{(n)}$ and $A_{22}^{(n)}$ shows that 
   \begin{align*}
      G_n := \left\{ z\in\C : \alpha_{n}' < \re z < \beta_{n}' \right\}
      \subseteq \rho(T_2^{(n)} + A_{22}^{(n)})
   \end{align*}
   where 
   \begin{align*}
      \alpha_{n}' & := \alpha + \frac{1}{n} + \sqrt{ \textstyle a^2 + b^2(\wgamma+\frac{1}{n})^2 + b^2 (\alpha + \frac{1}{n})^2},
      \\
      \beta_{n}'  & := \beta - \frac{1}{n} + \sqrt{ \textstyle a^2 + b^2(\wgamma+\frac{1}{n})^2 + b^2 (\beta - \frac{1}{n})^2}.
   \end{align*}
   In summary,
   $\sigmaess(T_2^{(n)}+A_{22}^{(n)}) \subseteq \sigma(T_2^{(n)}+A_{22}^{(n)})\subseteq \C\setminus (F_n\cup G_n)$.
   Since by Theorem~\ref{thm:essentialgaps} the essential spectra of $T+A$ and $T_2^{(n)} + A_{22}^{(n)}$ coincide for every $n\in\N$, we obtain that
   \begin{equation*}
      \sigmaess(T+A) \subseteq  \bigcap_{n\in\N} \C\setminus (F_n\cup G_n).
   \end{equation*}
   Since $\bigcup_{n\in\N} (F_n\cup G_n)$ is equal to the right hand side of \eqref{eq:bddstripessspec}, that inclusion is proved.
   The claim about nonaccumulation of the spectrum follows from \cite[Chapter~XVII, Theorem~2.1]{GohGolMarKaa90} because by Proposition~\ref{prop:bddstrip} the exterior of certain hyperbolas belongs to the resolvent set of $T+A$, hence the complement of the left hand side of \eqref{eq:bddstripessspec} contains points in the resolvent set of $T+A$.
\end{proof}

The situation in our next theorem is similar to that in Theorem \ref{thm:bddstripgap} but here we allow for additional finitely many eigenvalues.
\begin{theorem}
   \label{thm:EVmultiplicitybdd}
   Let $T$ and $A$ be as in Assumption~\ref{ass:relbounded}.
   For
   $\alpha < \beta$  and 
   $\gamma_1 \le \gamma_2$ we define the gapped strip
   $\strip = \{ z\in\C : \re z \in \R\setminus (\alpha, \beta),\ \gamma_1 \le \im z \le \gamma_2 \}$.
   Assume that 
   $\sigma(T) = \Sigma_0 \cup \Sigma_1$
   where
   $\Sigma_0 \subseteq \strip$
   and $\Sigma_1 = \{\lambda_1,\,\dots,\, \lambda_N\}$ is a finite set of eigenvalues of $T$, each of finite multiplicity, which satisfies
   $\Sigma_0 \cap \strip = \emptyset$.
   Set $\wgamma = \max\{ |\gamma_1|,\, |\gamma_2|\}$ and 
   let $\alpha'_{T+A}$ and $\beta'_{T+A}$ as in \eqref{eq:extremosdelatira2}.
   Moreover, for each $j=1,\dots,n$ we define $r_j = \sqrt{ a^2 + b^2 |\lambda_j|^2}$.

   \begin{enumerate}[label={\upshape(\roman*)}]

      \item\label{item:EVmultiplidicadbdd:i}
      For the resolvent set of $T+A$ we have that 
      \begin{align}
	 \label{eq:bddstripEV}
	 \begin{aligned}
	    \Big( 
	    \{z\in \C: \alpha'_{T+A} \le \re z \le \beta'_{T+A}\}
	    \cup
	    (\I\gamma_1 - \hyperbola_{\wgamma}) 
	    \cup (\I\gamma_2 + \hyperbola_{\wgamma})
	    \Big)
	    \setminus \bigcup_{j=1}^N K_{r_j}(\lambda_j)
	    \subseteq \rho(T+A)
	 \end{aligned}
      \end{align}
      where $K_{r_j}(\lambda_j)$ is the closed disk of radius $r_j$ centred in $\lambda_j$.

      \item\label{item:EVmultiplidicadbdd:ii}
      If there exists $k\in\{1,\,\dots,\, N\}$ and $\epsilon > 0$ such that 
      $K_{r_k+2\epsilon}(\lambda_k)$ is contained in the set
      \begin{multline}
	 \label{eq:bddstripEVnotk}
	 \Big( 
	 \{z\in \C: \alpha'_{T+A} \le \re z \le \beta'_{T+A}\}
	 \cup
	 (\I\gamma_1 - \hyperbola_{\wgamma}) \cup (\I\gamma_2 + \hyperbola_{\wgamma})
	 \Big)
	 \setminus \bigcup_{\substack{j=1\\ j\neq k}}^N K_{r_j}(\lambda_j), 
      \end{multline}
      then $\sigma(T+A)\cap K_{r_k}(\lambda_k)$ consists of finitely many eigenvalues and their total algebraic multiplicity is equal to the multiplicity of $\lambda_k$ as eigenvalue of $T$.

      \item\label{item:EVmultiplidicadbdd:iii}
      If there exists $\epsilon > 0$ such that for all $k\in\{1,\,\dots,\, N\}$ the closed disk $K_{r_k+2\epsilon}(\lambda_k)$ is contained in the set in \eqref{eq:bddstripEVnotk},
      then $\sigma(T+A)\cap \left( \bigcup_{j=1}^N K_{r_j}(\lambda_j) \right)$ consists of finitely many eigenvalues and their 
      total algebraic multiplicity is equal to the total multiplicity of $\lambda_1,\, \dots,\, \lambda_N$ as eigenvalues of $T$.

   \end{enumerate}
\end{theorem}
\begin{proof}
   In order to prove \ref{item:EVmultiplidicadbdd:i}, we start with $z\in\rho(T)$.
   As in the proofs of Proposition~\ref{prop:bddstrip} and Theorem~\ref{thm:bddstripgap}, we have to show that 
   $\sup_{t\in\sigma(T)} H_z(t) < 1$ for $z$ in the set on the left side of \eqref{eq:bddstripEV}.
   Note that 
   \begin{equation*}
      \sup_{t\in\sigma(T)} H_z(t)
      = \max\left\{ \sup_{t\in\Sigma_0} H_z(t),\  \sup_{t\in\Sigma_1} H_z(t) \right\}
      \le \max\left\{ \sup_{t\in\strip} H_z(t),\  \sup_{t\in\Sigma_1} H_z(t) \right\}.
   \end{equation*}
   By Proposition~\ref{prop:bddstrip} and Theorem~\ref{thm:bddstripgap} we know that $\sup_{t\in\strip} H_z(t) < 1$ if $z$ belongs to the union of the first three sets on the left hand side in \eqref{eq:bddstripEV}.
   Moreover,
   \begin{equation*}
      \sup_{t\in\Sigma_1} H_z(t) 
      = \max_{j=1}^N H_z(\lambda_j) 
      = \max_{j=1}^N \frac{a^2 + b^2 |\lambda_j|^2}{ |\lambda_j - z |^2 } 
   \end{equation*}
   which is smaller than $1$ if $z$ is outside $\bigcup_{j=1}^N K_{r_j}(\lambda_j)$.
   \medskip

   We proceed to prove \ref{item:EVmultiplidicadbdd:ii}.
   For $s\in [0,1]$ define $A_s := sA$.
   Then all $A_s$ are $T$-bounded and if $a,b$ are the constants from \eqref{eq:definicion2}, then we have the corresponding inequality for $A_s$ if we use $a_s:=\sqrt{s}a$ and $b_s:=\sqrt{s}b$ instead of $a$ and $b$.
   Note that the numbers $\alpha'_{T+A_s}, \beta'_{T+A_s}$ and $r_{j,s} := \sqrt{a_s^2 + b_s^2|\lambda_j|^2}$
   are non-decreasing in $s$.
   In particular we have \eqref{eq:bddstripEV} for all $A_s$ with the corresponding 
   $\alpha'_{T+A_s}, \beta'_{T+A_s}$.
   Therefore, if for the moment we denote the set \eqref{eq:bddstripEVnotk} by $U(A)$, we have that $U(sA)\supset U(s'A)$ for all $0\le s \le s' \le 1$
   and for any $s,s'\in [0,1]$ we have that $K_{r_{k,s}+2\epsilon} \subseteq U(s'A)$.
   Therefore we can define the Riesz projections
   \begin{equation}
      \label{eq:Rieszprojection}
      P_s := \frac{1}{2\pi \I} \int_{\partial K_{r_k+\epsilon}(\lambda_k)} (T+A_s-\lambda)^{-1}\,\rd\lambda,
      \qquad s\in[0,1].
   \end{equation}
   The dimension of the range of $P_s$ is equal to the total algebraic multiplicity of the eigenvalues of $A_s$ contained in $B_{r_k+\epsilon}(\lambda_k)$, hence in $K_{r_k}(\lambda_k)$.
   So we have to show that $\dim\Ran P_0 = \dim\Ran P_1$.
   To this end, observe that for $\lambda\in\partial K_{r_k+\epsilon}(\lambda_k)$
   \begin{equation*}
      (T+A_s-\lambda)^{-1}
      = (T-\lambda)^{-1} \sum_{n=0}^\infty s^n(-A(T-\lambda)^{-1})^n
   \end{equation*}
   which is continuous in $s$, uniformly with respect to $\lambda$.
   Therefore $P_s$ is continuous in $s$, hence $\dim\Ran P_s$ is constant in $s$
   by \cite[Ch.I, Lemma 4.10]{Kato95}.
   Therefore, applying the above equality $m$ times, we obtain that $\dim\Ran P_0 = \dim\Ran P_1$.
   \smallskip

   The claim in \ref{item:EVmultiplidicadbdd:iii} follows as in \ref{item:EVmultiplidicadbdd:ii} if we replace the contour of integration in \eqref{eq:Rieszprojection} by $\Gamma = \bigcup_{j=1}^N \partial K_{r_j+\epsilon}(\lambda_j)$.

\end{proof}

As a corollary, we obtain a generalization of \cite[Theorem 2.12]{CueninTretter2016}.

\begin{corollary}
   \label{cor:EVmultiplicity}
   Let $T$ and $A$ be as in Assumption~\ref{ass:relbounded}.
   For
   $\alpha < \beta$  and 
   $\gamma_1 \le \gamma_2$ we define the ``gapped strip''
   $\strip = \{ z\in\C : \re z \in \R\setminus (\alpha, \beta),\ \gamma_1 \le \im z \le \gamma_2 \}$.
   Assume that 
   $\sigma(T) = \Sigma_0 \cup \Sigma_1$
   where
   $\Sigma_0 \subseteq \strip$
   and $\Sigma_1 = \{\lambda_1,\,\dots,\, \lambda_N\}$ is a finite set of eigenvalues counted with multiplicity that satisfies
   $\Sigma_0 \cap \strip = \emptyset$ and $\gamma_1< \im\lambda_j < \gamma_2$ for $j=1,\dots, N$.
   Without restriction we may assume that $\re \lambda_1 \le \re\lambda_2 \le \dots \le \re \lambda_N$.
   Set $\wgamma = \max\{ |\gamma_1|,\, |\gamma_2|\}$ and, as in \eqref{eq:extremosdelatira2}, let
   \begin{equation*}
      \alpha'_{T+A}:=\alpha_{T}+\sqrt{a^2+b^2\wgamma^2+b^2\alpha_{T}^2},\qquad
      \beta'_{T+A}:= \beta_T -\sqrt{a^2+b^2\wgamma^2+b^2\beta_{T}^2}.
   \end{equation*}
   Moreover, we set 
   \begin{equation*}
      R_1 := \re \lambda_1 - \sqrt{a^2+b^2\wgamma^2+b^2(\re \lambda_1)^2},\qquad
      R_N :=  \re\lambda_N + \sqrt{a^2+b^2\wgamma^2+b^2(\re\lambda_N)^2}.
   \end{equation*}
   If $\alpha'_{T+A} < R_1$ and $R_N < \beta'_{T+A}$, then for any 
   $\widetilde R_1\in (\alpha'_{T+A},\, R_1)$,
   $\widetilde R_2\in (R_N, \, \beta'_{T+A})$
   and large enough $\eta$ the rectangular curve with corners $\widetilde R_1\pm\I\eta$ and $\widetilde R_2\pm\I\eta$ belongs to $\rho (T+sA)$ for any $s\in [0,1]$, its interior contains only eigenvalues of $T+A$ and their total algebraic multiplicity is equal to $N$.
\end{corollary}
\begin{proof}
   As in the proof of \ref{item:EVmultiplidicadbdd:iii} in Theorem~\ref{thm:EVmultiplicitybdd} it can be shown that
   the vertical parts of $\Gamma$ belong to $\rho(T+sA)$ for all $s\in [0,1]$ by Theorem~\ref{thm:bddstripgap}, applied to the spectral gaps $\{\alpha < \re z < \re\lambda_1\}$ and $\{\re\lambda_N < \re z < \beta\}$.
   The same is true for the horizontal lines from $\widetilde R_1\pm\I\eta$
   if we choose $\eta$ large enough such that they lie in 
   $\I\gamma_2 + \hyperbola_{\wgamma}$ and $\I\gamma_1 - \hyperbola_{\wgamma}$ respectively.
   Now the same continuity argument shows that the rank of the Riesz projections in \eqref{eq:Rieszprojection} with contour $\Gamma$ is constant in $s$.
\end{proof}


\section{Bisectorial and sectorial operators} 
\label{sec:sectorial}

We will denote the closed sector of angle $\theta$ and centred at $\lambda \in \R$ by
\begin{equation*}
   \sector{\theta}{\lambda}:=\{z\in \C: |\arg(z-\lambda)|\leq \theta\} \cup \{\lambda\}.
\end{equation*}
For $\alpha, \beta$, with $\alpha\leq \beta$ we define
\begin{equation}
   \label{eq:defOmega}
   \Omega (\alpha, \beta, \theta):= \sector{\theta}{\beta}\cup-\sector{\theta}{-\alpha}.
\end{equation}

\begin{figure}[H] 
  \centering
  \begin{tikzpicture}[scale=1]

      \tikzmath{
      \alphaT = -.5;      
      \betaT = 1.5;      
      \thetaT = 20;      
      }

      \path [fill=colTspec, opacity=1, shift={(\alphaT,0)}] (\thetaT:-3) -- (0,0) --  (-\thetaT:-3) -- cycle;
      \path [fill=colTspec, opacity=1, shift={(\betaT,0)}] (\thetaT:3) -- (0,0) --  (-\thetaT:3) -- cycle;
      \node [below, colTspec!50!black] at (\alphaT, -.1) {$\alpha$};
      \node [below, colTspec!50!black] at (\betaT, -.1) {$\beta$};
      \node [right, colTspec!50!black] at (4.2, .8) {$\sector{\theta}{\beta}$};
      \node [left, colTspec!50!black] at (-3.2, .8) {$\sector{\theta}{\alpha}$};

      \draw[->] (-3.5,0)--(4.5,0);
      \draw[->] (0,-2)--(0,2);

   \end{tikzpicture}
   \caption{The set $\Omega (\alpha, \beta, \theta) = \sector{\theta}{\beta} \cup -\sector{\theta}{\alpha}$.}
  \label{fig:figura4}
\end{figure}
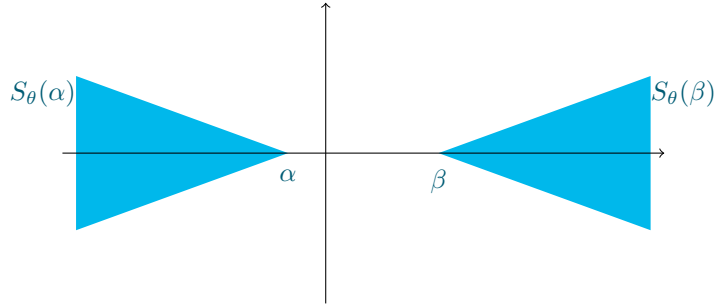

\begin{definition}
   \label{def:sectorial}
   An operator $T$ on a Hilbert space $H$ is called \define{sectorial} if $\sigma(T)\subseteq -\sector{\theta}{0}$ for some $\theta\in [0,\pi)$ and if for every $\theta' \in (\theta, \pi)$ there exists a constant $M\ge 0$ such that 
   $\|(T-z)^{-1}\| \le \frac{M}{|z|}$ for all $z\in\C\setminus (-\sector{\theta'}{0})$.
   An operator $T$ is called \define{quasi-sectorial} if there exists some $c \in\R$ such that $T+c I$ is sectorial.

   We call the operator $T$ \define{bisectorial} if 
   $\sigma(T)\subseteq \Omega(0,0,\theta)$ for some $\theta \in [0,\pi)$ and if for every $\theta' \in (\theta, \pi)$ there exists a constant $M\ge 0$ such that 
   $\|(T-z)^{-1}\| \le \frac{M}{|z|}$ for all $z\in\C\setminus \Omega(0, 0, \theta')$.
\end{definition}

Since $\|(T-z)^{-1}\| = \dist(z,\sigma(T))$ for all $z\in\rho(T)$ if $T$ is a normal operator, we obtain the following lemma.
\begin{lemma}
   \label{lem:normalsectorial}
   Let $T$ be a closed normal  operator on a Hilbert space $H$.
   \begin{enumerate}[label={\upshape(\roman*)}]

      \item 
      If $\sigma(T)\subseteq \Omega(\alpha_T, \beta_T, \theta)$, then, for any $c\in(\alpha_T,\beta_T)$, the operator $T -cI$ is bisectorial with a possibly larger angle $\theta'>\theta$.

      \item 
      If there exists $\mu, M \in\R$ such that $(\mu+\I\R)\setminus\{0\} \in \rho(T)$ and $\| (T-\mu-\I\nu)^{-1}\| \le \frac{M}{|\nu|}$ for all $\nu\in\R\setminus\{0\}$, then $T$ is bisectorial.
   \end{enumerate}
\end{lemma}

In our first theorem we will show that the spectrum of $T+A$ is contained in a set of the form \eqref{eq:defOmega} if $T$ is a normal bisectorial operator and $A$ is a sufficiently small perturbation.
We saw in Section~\ref{sec:boundedimaginary} that the guaranteed bounds for the spectrum of $T+A$ bend away from the boundary of $\sigma(T)$ by the angle $\arctan \frac{b}{\sqrt{1-b^2}}$.
Therefore, if $\sigma(T)$ is contained in some $\Omega(\alpha,\beta,\theta)$ and if we want a half-plane or at least a strip to be contained in $\rho(T+A)$, then 
it is reasonable to assume that the constant $b$ should satisfy $b<\cos\theta$ in order to guarantee that $\theta + \arctan \frac{b}{\sqrt{1-b^2}} < \pi/2$.
\medskip

In the following we will need sets of the form
\begin{equation*}
   \e^{\I\phi}( \Hind + \hyperbola_\gamma) 
   = 
   \left\{ z\in\C :
   \begin{aligned}
      & \frac{ a^2 + b^2\gamma^2 + b^2 (\re (\e^{ -\I\phi}(z-\Hind)) )^2 ] }{1-b^2}
      < (\im (\e^{ -\I\phi} (z - \Hind)) )^2,
      \\
      & \im (\e^{ -\I\phi} (z - \Hind)) > 0
   \end{aligned}
   \right\}
\end{equation*}
which is the region above a hyperbola shifted away from the real axis by $\Hind$ and then rotated by $\phi$,
see Figure~\ref{fig:rotatedhyperbolas}.

Let $\phi_b=\arctan \frac{b}{\sqrt{1-b^2}}$.
Then the asymptotes of the hyperbola 
$\e^{\I\phi}( \Hind + \hyperbola_\gamma)$ are
\begin{align*}
   \im z &= (\re z - \Hind)\tan(\phi + \phi_b) \text{ for }\re z \to\infty,\\
   \im z &= (\re z - \Hind)\tan(\phi - \phi_b) \text{ for }\re z \to -\infty.
\end{align*}
\smallskip

If $b < \cos\theta$, we define
\begin{align}
   \label{eq:betaTAsect}
   \alpha_{T+A}:=\alpha_{T}+\frac{\sqrt{a^2+b^2\alpha_{T}^2}}{\sqrt{1-b^2\tan^2\theta}},
   \qquad
   \beta_{T+A}:= \beta_T-\frac{\sqrt{a^2+b^2\beta_{T}^2}}{\sqrt{1-b^2\tan^2\theta}}.
\end{align}

Let us first generalise Corollary~\ref{cor:ourconjecture1} to the case when $-T$ is quasi-sectorial.

\begin{theorem}
   \label{thm:sectorial} 
   Let $T$ and $A$ and the constants $a,b$ be as in Assumption~\ref{ass:relbounded}
   and assume that 
   \begin{equation*}
      \sigma(T)\subseteq \sector{\theta}{\beta_T}
   \end{equation*}
   for some $\theta\in[0,\pi/2)$.
   Then
   \begin{align}
      \label{eq:resolvent:secthyperbola}
      \e^{\I\theta}(\beta_T + \hyperbola_{\beta_T}) \cup \e^{-\I\theta}(\beta_T - \hyperbola_{\beta_T})
      \subseteq \rho(T+A).
   \end{align}
   If in addition $b < \cos\theta$ holds, then the real part of $\sigma(T+A)$ is lower bounded by $\beta_{T+A}$ and we even have
   \begin{align}
      \label{eq:resolvent:sectstrip}
      \{ z\in\C : \re z < \beta_{T+A}\}
      \cup 
      \e^{\I\theta}(\beta_T + \hyperbola_{\beta_T}) \cup \e^{-\I\theta}(\beta_T - \hyperbola_{\beta_T})
      \subseteq \rho(T+A)
   \end{align}
   where $\beta_{T+A}$ as defined in \eqref{eq:betaTAsect}
   See Figure~\ref{fig:rotatedhyperbolas}.
\end{theorem}
\begin{remark}
   The analogous inclusions for the case when $\sigma(T)\subseteq -\sector{\theta}{-\alpha_T}$ are:
   \begin{align}
      \label{eq:resolvent:secthyperbolaA}
      \e^{-\I\theta}(\alpha_T + \hyperbola_{\alpha_T}) \cup \e^{\I\theta}(\alpha_T - \hyperbola_{\alpha_T})
      \subseteq \rho(T+A).
   \end{align}
   Moreover, if $b<\cos\theta$, then the real part of $\sigma(T+A)$ is upper bounded by $\alpha_{T+A}$ and we obtain the stronger result
   \begin{align}
      \label{eq:resolvent:sectstripA}
      \{ z\in\C : \re z > \alpha_{T+A}\}
      \cup 
      \e^{-\I\theta}(\alpha_T + \hyperbola_{\alpha_T}) \cup \e^{\I\theta}(\alpha_T - \hyperbola_{\alpha_T})
      \subseteq \rho(T+A).
   \end{align}
\end{remark}

\begin{remark}
   If \eqref{eq:resolvent:sectstrip} holds, then clearly there exists $\omega\in\R$ such that $\sigma(T+A) \subseteq \sector{\theta+\phi_b}{\omega}$.
\end{remark}

\begin{proof}[Proof of Theorem~\ref{thm:sectorial}] 
   Let $z = \mu + \I\nu \in\C\setminus\sector{\theta}{\beta_T}$ and define the function $H_z$ as in \eqref{eq:Hdef}.
   Then, by Proposition~\ref{prop:ourfirsprop3}\ref{item:ourfirstprop3:iii},
   $z\in\rho(T+A)$ if $\sup_{\sigma(T)} H_z(t) < 1$.
   Inequality \eqref{eq:Hsectorb} shows that for any $z\in\C\setminus \sector{\theta}{\beta_T}$ with $\arg(z-\beta_T)\in (\theta,\theta+\pi)$
   \begin{align*}
      \sup_{t\in \sigma(T)} H_z(t)
      \le
      \sup_{t\in \sector{\theta}{\beta_T}} 
      H_z(t)
      = \sup_{t\in \sector{\theta}{\beta_T}} 
      \frac{a^2+b^2|t|^2}{|t-z|^2}
      \le
      b^2 + 
      \frac{ a^2  + b^2\beta_T^2 + b^2 (\re( \e^{ -\I\theta}(z-\beta_T)) )^2 }{ (\im (\e^{ -\I\theta} (z - \beta_T)) )^2 } .
   \end{align*}
   Clearly, the right hand side is less than 1 if and only if
   \begin{align}
      \label{eq:ourfirsttheo2:b2}
      \frac{ a^2 +  b^2 \beta_T^2 + b^2 (\re (\e^{ -\I\theta}(z-\beta_T)) )^2 }{1-b^2}
      < 
      (\im (\e^{ -\I\theta} (z - \beta_T)) )^2.
   \end{align}
   Since $\arg(z-\beta_T)\in (\theta,\theta+\pi)$, this is the case if and only if
   $z\in \e^{\I\theta}(\beta_T + \hyperbola_\beta)$.
   Analogously, if $\arg(z-\beta_T)\in (-\theta-\pi, -\theta)$, then 
   $\sup_{t\in \sector{\theta}{\beta_T}} \frac{a^2+b^2|t|^2}{|t-z|^2} < 1$ if 
   $z\in \e^{-\I\theta}(\beta_T - \hyperbola_\beta)$.
   Therefore \eqref{eq:resolvent:secthyperbola} is proved.

   Now let us assume that $\mu = \re z < \beta_T$.
   Then \eqref{eq:Hsector2b} shows that
   \begin{align}
      \label{eq:thm:bisecgap:1}
      \sup_{t\in \partial\sector{\theta}{\beta_T}} 
      H_z(t)
      \le
      \max\left\{b^2,\ \frac{{a^2+b^2\beta_T^2}}{(\beta_T-\mu)^2} \right\}
      + b^2\tan^2\theta.
   \end{align}
   Under the additional assumption $\beta < \cos\theta$ we have that $b^2(1+\tan^2\theta)<1$.
   So the right hand side in \eqref{eq:thm:bisecgap:1} is smaller than 1 if
   $\frac{{a^2+b^2\beta_T^2}}{(\beta_T-\mu)^2}+b^2\tan^2\theta < 1$.
   This is the case
   if and only if 
   $\mu < \beta_T - \frac{\sqrt{a^2+b^2\beta_T^2}}{1-b^2\tan^2\theta} = \beta_{T+A}$.
   It follows that $\{z\in\C : \re z < \beta_{T+A} \}\subseteq\rho(T+A)$ and 
   \eqref{eq:ourfirsttheo2:b2} is proved.
\end{proof}

\begin{theorem}
   \label{thm:bisecgap} 
   Let $T$ and $A$ and the constants $a,b$ be as in Assumption~\ref{ass:relbounded}
   and assume that there are $\alpha_T < \beta_T$ and $\theta\in [0, \pi/2)$ such that 
   \begin{equation*}
      \sigma(T)\subseteq \sector{\theta}{\beta_T}\cup-\sector{\theta}{-\alpha_T}.
   \end{equation*}
   Then
   \begin{multline}
      \label{eq:bisec:hyperbola}
      \Big(
      \e^{\I\theta}(\beta_T + \hyperbola_{\beta_T}) \cup \e^{-\I\theta}(\beta_T - \hyperbola_{\beta_T})
      \Big)
      \cap
      \Big(
      \e^{-\I\theta}(\alpha_T + \hyperbola_{\alpha_T}) \cup \e^{\I\theta}(\alpha_T - \hyperbola_{\alpha_T})
      \Big)
      \\
      \subseteq \rho(T+A).
   \end{multline}
   If $b < \cos\theta$, then we have in addition
   \begin{align}
      \label{eq:bisec:strip}
      \{z\in\C : \alpha_{T+A} < \re z < \beta_{T+A} \}
      \subseteq \rho(T+A)
   \end{align}
   where $\alpha_{T+A}$ and $\beta_{T+A}$ are defined in \eqref{eq:betaTAsect}.
   See Figure~\ref{fig:bisectorgap}.

\end{theorem}
\begin{proof} 
   Let $z = \mu + \I\nu \in\C\setminus\Omega(\alpha_T, \beta_T, \theta)$ and define the function $H_z$ as in \eqref{eq:Hdef}.
   Proposition~\ref{prop:ourfirsprop3}\ref{item:ourfirstprop3:iii} shows that $z\in\rho(T+A)$ if $\sup_{\sigma(T)} H_z(t) < 1$.
   Observe that 
   \begin{align}
      \label{eq:bisecgap}
      \sup_{t\in \sigma(T)} H_z(t)
      \le \max\left\{ \sup_{t\in -\sector{\theta}{-\alpha_T}} H_z(t),\
      \sup_{t\in \sector{\theta}{\beta_T}} H_z(t) \right\}.
   \end{align}
   From the proof of Theorem~\ref{thm:sectorial} we know that 
   $\sup_{t\in \sector{\theta}{\alpha_T}} H_z(t) < 1$
   if $z$ belongs to the left hand side of \eqref{eq:resolvent:secthyperbola}
   and that 
   $\sup_{t\in -\sector{\theta}{-\alpha_T}} H_z(t) < 1$
   if $z$ belongs to the left hand side of \eqref{eq:resolvent:secthyperbolaA}.
   Therefore the right hand side of \eqref{eq:bisecgap} is less than 1 if $z$ belongs to both sets which proves \eqref{eq:bisec:hyperbola}.

   If $b<\cos\theta$, then the proofs of \eqref{eq:resolvent:sectstrip} and \eqref{eq:resolvent:sectstripA} show that the right hand side of \eqref{eq:bisecgap} is less than 1 if 
   $\alpha_{T+A} < \re z < \beta_{T+A}$.
   Hence \eqref{eq:bisec:strip} holds.
\end{proof}

\begin{figure}[H] 
   \begin{tikzpicture}[ 
      declare function={ 
      hyperbolarotb(\x) = sqrt( (\abound^2 + \bbound^2*(\betaT)^2  + \bbound^2*(\x)^2)/(1-\bbound^2) ); 
      hyperbolarota(\x) = sqrt( (\abound^2 + \bbound^2*(\alphaT)^2 + \bbound^2*(\x)^2)/(1-\bbound^2) ); 
      } ]

      \tikzmath{
      \gT = 0;  
      \abound = .5;   
      \bbound = .3;  
      \alphaT = -2;   
      \betaT = 1;      
      \thetaT = 15;      
      \alphaTA = \alphaT + sqrt( \abound^2 + \bbound^2*\gT^2 + \bbound^2*(\alphaT)^2);   
      \betaTA = \betaT - sqrt( \abound^2 + \bbound^2*\gT^2 + \bbound^2*(\betaT)^2);      
      \thetaB = atan(\bbound/sqrt(1-\bbound^2));
      }

      \begin{scope}[scale=.6, xshift=-13cm] 

	 \begin{scope}
	    \clip (-5,-3) rectangle (5, 3);

	    \path [fill=colPlane, shift={(\betaT,0)}, rotate=\thetaT] (-8,0) -- (8,0) --  (8, -8) -- (-8,-8) -- cycle;

	    \path [fill=colTspec, shift={(\betaT,0)}] (\thetaT:6) -- (0,0) --  (-\thetaT:6) -- cycle;

	 \end{scope}

	 \draw[->] (-5.5,0)--(5.5,0);
	 \draw[->] (0,-3)--(0,3);

	 \node at (-5,3) {(a)};
      \end{scope}


      \begin{scope}[scale=.6, xshift=0cm] 

	 \begin{scope}
	    \clip (-5,-3) rectangle (5, 3);


	    \fill [colTAspec, opacity=.2, shift={(\betaT,0)} ]
	    plot[domain=-7:7, smooth,  rotate=\thetaT] ({\x},{hyperbolarotb(\x)}) -- (7,-7) -- (-7,-7)-- cycle;

	    \path [fill=colTspec, shift={(\betaT,0)}] (\thetaT:6) -- (0,0) --  (-\thetaT:6) -- cycle;

	    \draw [red, densely dashed] plot [domain=\betaT:7.0, smooth] ({\x},{(\x-\betaT) * tan(\thetaT+\thetaB)});
	    \draw [red, densely dashed] plot [domain=-7.0:\betaT, smooth] ({\x},{(\x-\betaT) * tan(\thetaT-\thetaB)});
	 \end{scope}

	 \draw[->] (-5.5,0)--(5.5,0);
	 \draw[->] (0,-3)--(0,3);

	 \node at (-5,3) {(b)};
      \end{scope}


      \begin{scope}[scale=.6, xshift=-13cm, yshift=-8cm] 

	 \begin{scope}
	    \clip (-5,-3) rectangle (5, 3);

	    \path [fill=colPlane, shift={(\betaT,0)}, rotate={180-\thetaT}] (-8,0) -- (8,0) --  (8, -8) -- (-8,-8) -- cycle;

	    \path [fill=colTspec, shift={(\betaT,0)}] (\thetaT:6) -- (0,0) --  (-\thetaT:6) -- cycle;
	 \end{scope}

	 \draw[->] (-5.5,0)--(5.5,0);
	 \draw[->] (0,-3)--(0,3);

	 \node at (-5,3.5) {(c)};
      \end{scope}


      \begin{scope}[scale=.6, xshift=0cm, yshift=-8cm] 

	 \begin{scope}
	    \clip (-5,-3) rectangle (5, 3);


	    \fill [colTAspec, opacity=.2, shift={(\betaT,0)} ]
	    plot[domain=-7:7, smooth,  rotate={180-\thetaT}] ({\x},{hyperbolarotb(\x)}) -- (-7,7) -- (7,7) -- cycle;

	    \path [fill=colTspec, shift={(\betaT,0)}] (\thetaT:6) -- (0,0) --  (-\thetaT:6) -- cycle;

	    \draw [red, densely dashed] plot [domain=\betaT:-7.0, smooth] ({\x},{(\x-\betaT) * tan(-\thetaT+\thetaB)});
	    \draw [red, densely dashed] plot [domain=7.0:\betaT, smooth] ({\x},{(\x-\betaT) * tan(-\thetaT-\thetaB)});
	 \end{scope}

	 \draw[->] (-5.5,0)--(5.5,0);
	 \draw[->] (0,-3)--(0,3);

	 \node at (-5,3.5) {(d)};
      \end{scope}


      \begin{scope}[scale=0.6, xshift=-13cm, yshift=-15cm] 

	 \begin{scope}
	    \clip (-5,-3) rectangle (5, 3);

	    \begin{scope}
	       \clip plot[domain=-7:7, smooth, shift={(\betaT,0)},  rotate=\thetaT] ({\x},{hyperbolarotb(\x)}) -- (7,-7) -- (-7,-7)-- cycle;
	       \fill [colTAspec, opacity=.2 ]
	       plot[domain=-7:7, smooth, shift={(\betaT,0)},  rotate={180-\thetaT}] ({\x},{hyperbolarotb(\x)}) -- (-7,7) -- (7,7) -- cycle;
	    \end{scope}

	    \path [fill=colTspec, shift={(\betaT,0)}] (\thetaT:6) -- (0,0) --  (-\thetaT:6) -- cycle;

	    \draw [blue] (\betaTA, -5) -- (\betaTA, 5);

	 \end{scope}

	 \draw[->] (-5.5,0)--(5.5,0);
	 \draw[->] (0,-3)--(0,3);

	 \node at (-5,3.5) {(e)};

      \end{scope}


      \begin{scope}[scale=0.6, xshift=0cm, yshift=-15cm] 

	 \tikzmath{
	 \thetaT = 40;      
	 \betaTA = \betaT - sqrt( \abound^2 + \bbound^2*\gT^2 + \bbound^2*(\betaT)^2);      
	 }

	 \begin{scope}
	    \clip (-5,-3) rectangle (5, 3);

	    \begin{scope}
	       \clip plot[domain=-7:7, smooth, shift={(\betaT,0)},  rotate=\thetaT] ({\x},{hyperbolarotb(\x)}) -- (7,-7) -- (-7,-7)-- cycle;
	       \fill [colTAspec, opacity=.2 ]
	       plot[domain=-7:7, smooth, shift={(\betaT,0)},  rotate={180-\thetaT}] ({\x},{hyperbolarotb(\x)}) -- (-7,7) -- (7,7) -- cycle;
	    \end{scope}

	    \path [fill=colTspec, shift={(\betaT,0)}] (\thetaT:6) -- (0,0) --  (-\thetaT:6) -- cycle;

	    \draw [blue] (\betaTA, -5) -- (\betaTA, 5);

	 \end{scope}

	 \draw[->] (-5.5,0)--(5.5,0);
	 \draw[->] (0,-3)--(0,3);

	 \node at (-5,3.5) {(f)};

      \end{scope}


   \end{tikzpicture}
   
   \caption[Illustration for the proof of Theorem~\ref{thm:sectorial}.]
   {%
   Illustration for the proof of Theorem~\ref{thm:bisecgap} and Lemma~\ref{lem:app:Hest}~\ref{item:lem:app:Hsectorb}.
   In (a), the lightblue half-plane is the plane $P$ above which the estimate \eqref{eq:halfplaneestimate} holds.
   The white regions in (b) and (d) correspond to the sets $\e^{\I\theta}(\beta_T + \hyperbola_{\beta_T})$ and $\e^{-\I\theta}(\beta_T - \hyperbola_{\beta_T})$, respectively.
   There, $\sup_{t\in\partial(\sector{\theta}{\beta_T})} H_z(t) < 1$ holds.
   The red dashed lines indicate the asymptotes of the hyperbola.
   If we combine the guaranteed regions for the resolvent set from (b) and (d), we find that the spectrum of $T+A$ must be contained in the grey area in (e).
   The vertical blue line in (e) and (f) is the lower bound $\beta_{T+A}$ for the real part of the spectrum of $T+A$ in the case $b<\cos\theta$, so $\sigma(T+A)$ is contained in the grey area to its right.
   The graphic in (f) shows the spectral inclusion for a larger angle $\theta$.
   }
  \label{fig:rotatedhyperbolas}
\end{figure}
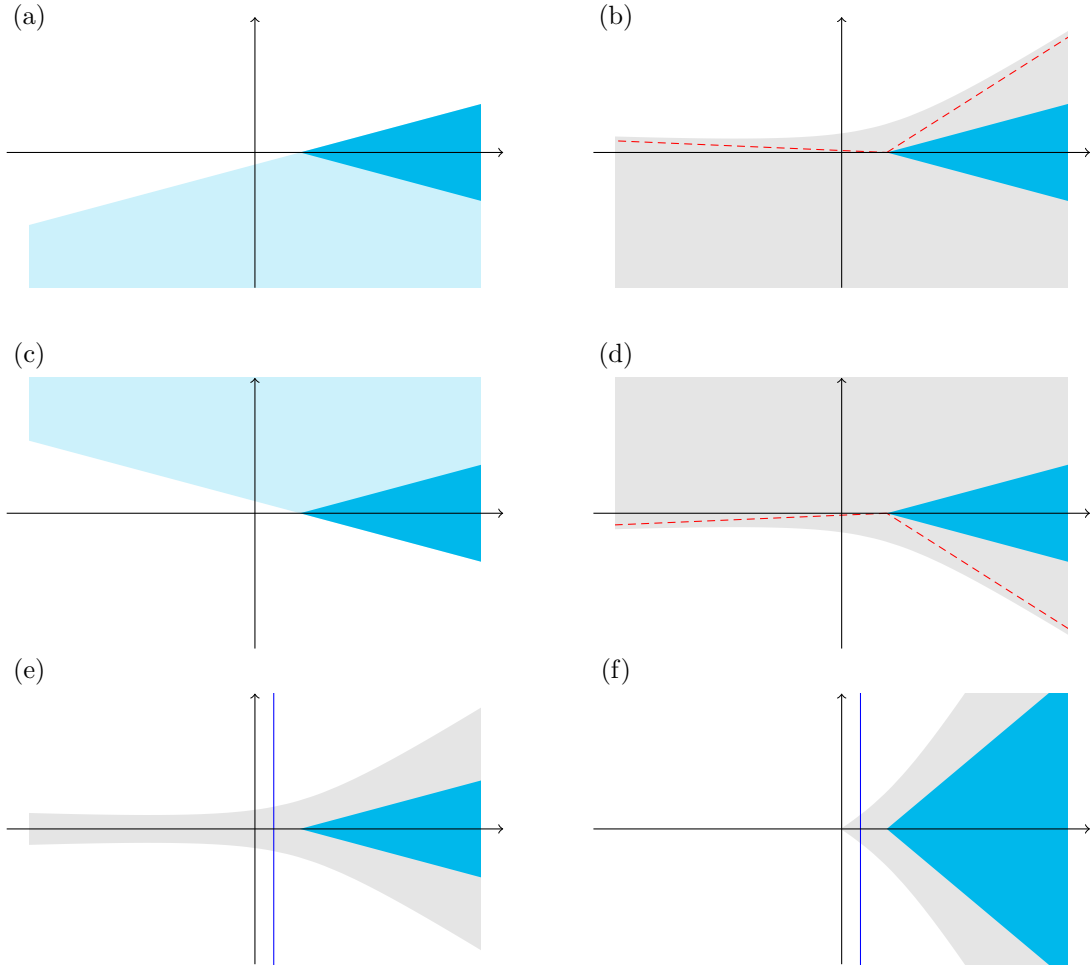

\begin{figure}[H] 
   \centering
   \begin{tikzpicture}[ 
      declare function={ 
      hyperbolarotb(\x) = sqrt( (\abound^2 + \bbound^2*(\betaT)^2  + \bbound^2*(\x)^2)/(1-\bbound^2) ); 
      hyperbolarota(\x) = sqrt( (\abound^2 + \bbound^2*(\alphaT)^2 + \bbound^2*(\x)^2)/(1-\bbound^2) ); 
      } ]

      \tikzmath{
      \gT = 0;  
      \abound = .5;   
      \bbound = .3;  
      \alphaT = -1;   
      \betaT = 2;      
      \thetaT = 15;      
      \alphaTA = \alphaT + sqrt( \abound^2 + \bbound^2*\gT^2 + \bbound^2*(\alphaT)^2);   
      \betaTA = \betaT - sqrt( \abound^2 + \bbound^2*\gT^2 + \bbound^2*(\betaT)^2);      
      }

      \begin{scope}[scale=0.6, transform shape] 

	 \begin{scope}
	    \clip (-5, -3) rectangle (\alphaTA, 3);


	    \begin{scope}
	       \clip plot[domain=-7.0:7.0, smooth, shift={(\alphaT,0)}, rotate=-\thetaT] ({\x},{hyperbolarota(\x)}) -- (7,-7) -- (-7,-7) -- cycle;
	       \fill [colTAspec, opacity=.2]
	       plot[domain=7.0:-7.0, smooth, shift={(\alphaT,0)}, rotate=\thetaT] ({\x},{-hyperbolarota(\x)}) -- (-7,7) -- (7,7) -- cycle;
	    \end{scope}

	 \end{scope}

	 \begin{scope}
	    \clip (\betaTA,-3) rectangle (5, 3);

	    \begin{scope}
	       \clip plot[domain=7:-7, smooth, shift={(\betaT,0)}, rotate=-\thetaT] ({\x},{-hyperbolarotb(\x)}) -- (-7,7) -- (7,7) -- cycle;
	       \fill [colTAspec, opacity=.2]
	       plot[domain=-7:7, smooth, shift={(\betaT,0)} , rotate=\thetaT] ({\x},{hyperbolarotb(\x)}) -- (7,-7) -- (-7,-7) -- cycle;
	    \end{scope}

	 \end{scope}

	 \begin{scope}
	    \clip (-5,-5) rectangle (5, 5);

	    \path [fill=colTspec, opacity=1, shift={(\alphaT,0)}] (\thetaT:-6) -- (0,0) --  (-\thetaT:-6) -- cycle;
	    \path [fill=colTspec, opacity=1, shift={(\betaT,0)}] (\thetaT:6) -- (0,0) --  (-\thetaT:6) -- cycle;
	 \end{scope}

	 \draw (\alphaTA,-.1) -- node [below] {$\alpha_{T+A}$} (\alphaTA,.1);
	 \draw (\betaTA,-.1) -- node [below] {$\beta_{T+A}$} (\betaTA,.1);

	 \draw[->] (-5.5,0)--(5.5,0);
	 \draw[->] (0,-3)--(0,3);
      \end{scope}

      \begin{scope}[scale=0.6, transform shape, xshift=13cm] 
      \tikzmath{
      \thetaT = 35;      
      \alphaTA = \alphaT + sqrt( \abound^2 + \bbound^2*\gT^2 + \bbound^2*(\alphaT)^2);   
      \betaTA = \betaT - sqrt( \abound^2 + \bbound^2*\gT^2 + \bbound^2*(\betaT)^2);      
      }

	 \begin{scope}
	    \clip (-5, -3) rectangle (\alphaTA, 3);


	    \begin{scope}
	       \clip plot[domain=-7.0:7.0, smooth, shift={(\alphaT,0)}, rotate=-\thetaT] ({\x},{hyperbolarota(\x)}) -- (7,-7) -- (-7,-7) -- cycle;
	       \fill [colTAspec, opacity=.2]
	       plot[domain=7.0:-7.0, smooth, shift={(\alphaT,0)}, rotate=\thetaT] ({\x},{-hyperbolarota(\x)}) -- (-7,7) -- (7,7) -- cycle;
	    \end{scope}

	 \end{scope}

	 \begin{scope}
	    \clip (\betaTA,-3) rectangle (5, 3);

	    \begin{scope}
	       \clip plot[domain=7:-7, smooth, shift={(\betaT,0)}, rotate=-\thetaT] ({\x},{-hyperbolarotb(\x)}) -- (-7,7) -- (7,7) -- cycle;
	       \fill [colTAspec, opacity=.2]
	       plot[domain=-7:7, smooth, shift={(\betaT,0)} , rotate=\thetaT] ({\x},{hyperbolarotb(\x)}) -- (7,-7) -- (-7,-7) -- cycle;
	    \end{scope}

	 \end{scope}

	 \begin{scope}
	    \clip (-5,-5) rectangle (5, 5);

	    \path [fill=colTspec, opacity=1, shift={(\alphaT,0)}] (\thetaT:-6) -- (0,0) --  (-\thetaT:-6) -- cycle;
	    \path [fill=colTspec, opacity=1, shift={(\betaT,0)}] (\thetaT:6) -- (0,0) --  (-\thetaT:6) -- cycle;
	 \end{scope}

	 \draw (\alphaTA,-.1) -- node [below] {$\alpha_{T+A}$} (\alphaTA,.1);
	 \draw (\betaTA,-.1) -- node [below] {$\beta_{T+A}$} (\betaTA,.1);

	 \draw[->] (-5.5,0)--(5.5,0);
	 \draw[->] (0,-3)--(0,3);
      \end{scope}

   \end{tikzpicture}
   
   \caption{Illustration of the spectral inclusion from Theorem~\ref{thm:bisecgap} for different angles $\theta$.
   The blue sectors $\Omega(\alpha_T, \beta_T, \theta)$ contain the spectrum of $T$ and the spectrum of $T+A$ is contained in the grey area.
   }
  \label{fig:bisectorgap}

\end{figure}
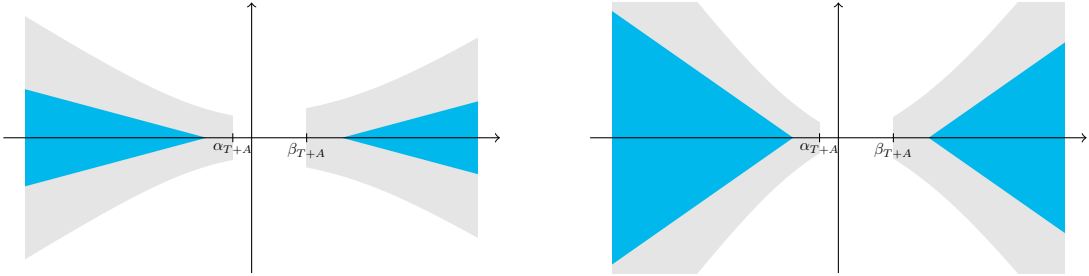

The proof of Theorem~\ref{thm:bisecgap} leads to estimates for the norm of the resolvent of $T+A$.

\begin{corollary}
   \label{cor:resolventestimate2}
   Let $T$ and $A$ be as in Proposition~\ref{thm:bisecgap}
   with $\alpha_{T+A}$ and $\beta_{T+A}$ as in \eqref{eq:betaTAsect}.
   Then, for $z\in \C$ with $\alpha_{T+A}<\re z<\beta_{T+A}$, we have that 
   \begin{align*}
      &\|(T+A-z)^{-1}\|\\
      &\leq \frac{1}{\dist(z,\Omega(\alpha_T, \beta_T, \theta))}
      \frac{1}{1- \sqrt{ b^2\tan^2\theta + \max\left\{
      b^2,\ \frac{{a^2+b^2\beta_T^2}}{(\beta_T-\re z)^2},\ \frac{{a^2+b^2\alpha_T^2}}{(\re z-\alpha_T)^2}
      \right\}} }.
   \end{align*}
\end{corollary}
\begin{proof} 
   Let $z= \mu + \I\nu$ with $\alpha_{T+A} < \mu < \beta_{T+A}$.
   Then, by \eqref{eq:thm:bisecgap:1} and the analogous estimate for $-\sector{\theta}{-\alpha}$, we obtain
   \begin{equation}
      \label{eq:thm:bisecgap:cor}
      \sup_{t\in\sigma(T)} H_z(t)
      \le
      \max\left\{b^2,\ \frac{{a^2+b^2\beta_T^2}}{(\beta_T-\re z)^2},\ \frac{{a^2+b^2\alpha_T^2}}{(\re z-\alpha_T)^2} \right\}
      + b^2\tan^2\theta.
   \end{equation}
   Note that 
   \begin{align}
      \label{eq:cor:resolventestimate2:T}
      \|(T-z)^{-1}\|
      &\leq \frac{1}{\dist(z,\sigma(T))}
      \leq \frac{1}{\dist(z,\Omega(\alpha_T,\beta_T,\theta))}.
   \end{align}
   Now the resolvent estimate follows from 
   Proposition~\ref{prop:ourfirsprop3}\ref{item:ourfirstprop3:iii}, 
   \eqref{eq:thm:bisecgap:cor}
   and \eqref{eq:cor:resolventestimate2:T}.
\end{proof}
There are various ways to estimate \eqref{eq:cor:resolventestimate2:T}.
The most straight forward is 
\begin{align*}
   \frac{1}{\dist(z,\Omega(\alpha_T,\beta_T,\theta))}
   \le
   \frac{1}{{\min\{\re z-\alpha_T,\beta_T-\re z\}} }.
\end{align*}
We can also use \eqref{eq:Hsectorb} and \eqref{eq:Hsectora} with $b=0$ and $a=1$.
Then we obtain for $\im z > 0$ that 
\begin{align*}
   \frac{1}{\dist(z,\Omega(\alpha_T,\beta_T,\theta))}
   \le
   \frac{1}{{\min\{ |\im(\e^{-\I\theta}(z-\beta_T) )|,\  |\im(\e^{\I\theta}(z-\alpha_T) )|  \}} }
\end{align*}
which gives a decay of order $(\im z)^{-1}$ for $\im(z)\to\pm\infty$.
\smallskip

The next proposition combines the results from Theorem~\ref{thm:bisecgap} and Proposition~\ref{prop:bddstrip}.
Of course, an analogous result can be formulated for a strip that is not symmetric to the real axis.

\begin{proposition}  
   \label{prop:bisectorstrip} 
   Let $T$ and $A$ and the constants $a,b$ be as in Assumption~\ref{ass:relbounded}.
   Assume that there are $\alpha_T < \beta_T$ and $\theta\in [0, \pi/2)$ such that 
   $b < \cos\theta$ and that 
   \begin{equation*}
      \sigma(T)\subseteq \sector{\theta}{\beta_{T}}\cup-\sector{\theta}{-\alpha_{T}} \cup ( \R + \I[-\gamma_T,\gamma_T]).
   \end{equation*}
   Let $\alpha_{T+A}$ and $\beta_{T+A}$ as in \eqref{eq:betaTAsect} and $\hyperbola_{\Hind}$ as in \eqref{eq:defHyper}.
   Then
   \begin{multline*}
      \Big(
      \e^{\I\theta}(\beta_T + \hyperbola_{\beta_T}) \cup \e^{-\I\theta}(\beta_T - \hyperbola_{\beta_T})
      \Big)
      \cap
      \Big(
      \e^{-\I\theta}(\alpha_T + \hyperbola_{\alpha_T}) \cup \e^{\I\theta}(\alpha_T - \hyperbola_{\alpha_T})
      \Big)
      \\
      \cap
      \Big(
      (\I\gamma + \hyperbola_{\gamma_T})
      \cup
      (-\I\gamma - \hyperbola_{\gamma_T})
      \Big)
      \subseteq \rho(T+A),
   \end{multline*}
   see Figure~\ref{fig:bisectorstrip}.

\end{proposition}
\begin{proof} 
   The proof is analogous to that of Theorem~\ref{thm:bisecgap}.
   We define the function $H_z$ as in \eqref{eq:Hdef}.
   In order to localise $\rho(T+A)$ we need to find those $z\in\C$ for which
   $\sup_{t\in\sigma(T)} H_z(t) < 1$.
   As before, we use the estimate
   \begin{align*}
      \sup_{t\in\sigma(T)} H_z(t)
      & \le \sup \left\{ H_z(t) :
      t\in \Omega(\alpha_T,\beta_T, \theta) \cup (\R + \I [-\gamma_T,\gamma_T]) \right\}
      \\
      & \le \max\left\{ 
      \sup_{t\in \Omega(\alpha_T,\beta_T, \theta)} H_z(t),\
      \sup_{t\in \R + \I [-\gamma_T,\gamma_T]} H_z(t)
      \right\}.
   \end{align*}
   We know from the proof of Theorem~\ref{thm:bisecgap}
   that
   $\sup\limits_{t\in \Omega(\alpha_T,\beta_T, \theta)} H_z(t) < 1$ if 
   $z\in 
   \Big(
   \e^{\I\theta}(\beta_T + \hyperbola_{\beta_T}) \cup \e^{-\I\theta}(\beta_T - \hyperbola_{\beta_T}) 
   \Big)
   \cap
   \Big(
   \e^{-\I\theta}(\alpha_T + \hyperbola_{\alpha_T}) \cup \e^{\I\theta}(\alpha_T - \hyperbola_{\alpha_T}) 
   \Big)$
   while
   $\sup\limits_{t\in \R + \I [-\gamma_T,\gamma_T]} H_z(t) < 1$ if 
   $z\in (\I\gamma + \hyperbola_{\gamma_T}) \cup (-\I\gamma - \hyperbola_{\gamma_T})$
   by the proof of Proposition~\ref{prop:bddstrip}.
\end{proof}

\begin{figure}[H] 
   \begin{tikzpicture}[ 
      declare function={ 
      hyperbolarotb(\x) = sqrt( (\abound^2 + \bbound^2*\betaT^2 + \bbound^2*(\x)^2)/(1-\bbound^2) ); 
      hyperbolarota(\x) = sqrt( (\abound^2 + \bbound^2*(\alphaT)^2 + \bbound^2*(\x)^2)/(1-\bbound^2) ); 
      hyperbola(\x) = \gT + sqrt( (\abound^2 + \bbound^2*\gT^2 + \bbound^2*\x*\x)/(1-\bbound^2) );
      } ]

      \tikzmath{
      \gT = 1.5;  
      \abound = .5;   
      \bbound = .3;  
      \alphaT = -1;   
      \betaT = 2;      
      \thetaT = 25;      
      \alphaTA = \alphaT + sqrt( \abound^2 + \bbound^2*\gT^2 + \bbound^2*(\alphaT)^2);   
      \betaTA = \betaT - sqrt( \abound^2 + \bbound^2*\gT^2 + \bbound^2*(\betaT)^2);      
      }

      \begin{scope}[xshift=-8cm, scale=0.3, transform shape] 

	 \begin{scope}
	    \clip (-10,-5) rectangle (10, 5);

	    \path [fill=colTspec, shift={(\alphaT,0)}] (\thetaT:-10) -- (0,0) --  (-\thetaT:-10) -- cycle;
	    \path [fill=colTspec, shift={(\betaT,0)}] (\thetaT:9) -- (0,0) --  (-\thetaT:9) -- cycle;
	    \path [fill=colTspec] (-10,-\gT) rectangle (10,\gT);

	    \draw [colTspec!30!black, dashed, shift={(\alphaT,0)}] (\thetaT:-10) -- (0,0) --  (-\thetaT:-10);
	    \draw [colTspec!30!black, dashed, shift={(\betaT,0)}] (\thetaT:9) -- (0,0) --  (-\thetaT:9);

	 \end{scope}

	 \draw[->] (-10.5,0)--(10.5,0);
	 \draw[->] (0,-5.5)--(0,5.5);
      \end{scope}


      \begin{scope}[scale=0.3, transform shape] 

	 \begin{scope}
	    \clip (-10,-5) rectangle (10, 5);

	    \path [fill=colTspec, shift={(\alphaT,0)}] (\thetaT:-10) -- (0,0) --  (-\thetaT:-10) -- cycle;
	    \path [fill=colTspec, shift={(\betaT,0)}] (\thetaT:9) -- (0,0) --  (-\thetaT:9) -- cycle;
	    \path [fill=colTspec] (-10,-\gT) rectangle (10,\gT);

	 \end{scope}

	 \begin{scope}
	    \clip (-10,-7) rectangle (10, 7);

  	    \fill [colTAspec, opacity=.2, shift={(\alphaT,0)} ]
  	    plot[domain=-15.0:15.0, smooth, rotate=-\thetaT] ({\x},{hyperbolarota(\x)}) 
  	    -- plot[domain=15.0:-15.0, smooth, rotate=\thetaT] ({\x},{-hyperbolarota(\x)}) -- cycle;

	    \fill [colTAspec, opacity=.2, shift={(\betaT,0)} ]
	    plot[domain=-15.0:10.0, smooth, rotate=\thetaT] ({\x},{hyperbolarotb(\x)}) -- plot[domain=10.0:-15.0, smooth, rotate=-\thetaT] ({\x},{-hyperbolarotb(\x)}) -- cycle;

	    \fill [green, opacity=.4 ]
	    plot[domain=-10.0:10.0, smooth] ({\x},{hyperbola(\x)}) -- plot[domain=10.0:-10.0, smooth] ({\x},{-hyperbola(\x)}) -- cycle;

	    \path [fill=colTspec, shift={(\alphaT,0)}] (\thetaT:-10) -- (0,0) --  (-\thetaT:-10) -- cycle;
	    \path [fill=colTspec, shift={(\betaT,0)}] (\thetaT:9) -- (0,0) --  (-\thetaT:9) -- cycle;
	    \path [fill=colTspec] (-10,-\gT) rectangle (10,\gT);

	 \end{scope}

	 \draw[->] (-10.5,0)--(10.5,0);
	 \draw[->] (0,-5.5)--(0,5.5);
      \end{scope}


   \end{tikzpicture}
   
   \caption{Localisation of the spectrum of $T+A$ in the case of Proposition~\ref{prop:bisectorstrip} when $A$ is $T$-bounded and the spectrum of $T$ is contained in $\Omega(\alpha_T,\beta_T,\theta)\cup ( \R + \I [-\gamma_T,\gamma_T]$.
   The dashed lines in the left picture indicate the boundaries of $\Omega(\alpha_T,\beta_T,\theta)$.
   In the picture on the right, the grey area indicate the area that we get from Theorem~\ref{thm:bisecgap}.
   The green area shows the contribution from the strip $\R+\I[-\gamma_T,\gamma_t]$.
   Therefore the spectrum of $T+A$ is contained in the union of the coloured areas.
   }
  \label{fig:bisectorstrip}

\end{figure}
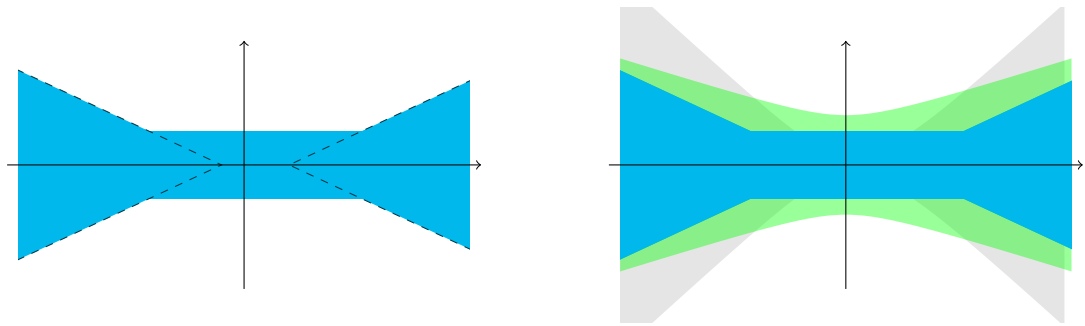

We finish this section with the following theorem for essential spectral gaps.

\begin{theorem}
   \label{thm:essspecsec}
   Let $T$ and $A$ be as in Assumption~\ref{ass:relbounded}.
   Assume that there exist 
   $\alpha < \beta$ and $\theta\in [0,\pi/2)$ such that 
   $\sigmaess(T)\subseteq \Omega(\alpha,\beta,\theta)$
   and 
   $\sigma(T) = \sigmaess(T) \cup \Sigma$ 
   where $\Sigma\subseteq \C\setminus\Omega(\alpha, \beta, \theta)$ is a bounded countable set which may accumulate at most at the boundary of $\sigmaess(T)$.
   Let $\alpha_{T+A}$ and $\beta_{T+A}$ as in \eqref{eq:betaTAsect}.
   If $\alpha_{T+A} < \beta_{T+A}$
   then 
   \begin{multline*}
      \Big(
      \e^{\I\theta}(\beta_T + \hyperbola_{\beta_T}) \cup \e^{-\I\theta}(\beta_T - \hyperbola_{\beta_T})
      \cup \{\re z < \beta_{T+A} \}
      \Big)
      \\
      \cap
      \Big(
      \e^{-\I\theta}(\alpha_T + \hyperbola_{\alpha_T}) \cup \e^{\I\theta}(\alpha_T - \hyperbola_{\alpha_T})
      \cup \{ \re z > \alpha_{T+A} \}
      \Big)
      \subseteq \C\setminus\sigmaess(T+A).
   \end{multline*}
   and the set on the left hand side consists of at most 
   countably many isolated eigenvalues of finite algebraic multiplicity which may accumulate only at its boundary.
\end{theorem}

\begin{proof}
   The proof is analogous to that of Theorem~\ref{thm:essspecbdd} if we use for instance the sets
   $S_n = \Omega(\alpha+\frac{1}{n}, \beta-\frac{1}{n}, \theta)$.
\end{proof}

Clearly, a theorem analogous to Theorem~\ref{thm:EVmultiplicitybdd} about the multiplicity of isolated eigenvalues can be proved.
We omit details.
\bigskip

In the proofs of Theorem~\ref{thm:sectorial}, Theorem~\ref{thm:bisecgap}, Corollary~\ref{cor:resolventestimate2}, Proposition~\ref{prop:bisectorstrip} we used the estimates 
\eqref{eq:Hsectorb} and \eqref{eq:Hsectora}
from Lemma~\ref{lem:app:Hest}.
If we use the alternative estimates 
\eqref{eq:Hsectorb:alt} and \eqref{eq:Hsectora:alt}
from Lemma~\ref{lem:app:Hestalt} then 
we obtain slightly different spectral inclusions for $T+A$. 
For details see Appendix~\ref{app:alternativebounds}.



\section{Other spectral inclusions} 
\label{sec:misc}

Using the techniques from Section~\ref{sec:boundedimaginary} and  Section~\ref{sec:sectorial} we can prove a variety of results for the location of the spectra of perturbed normal operators. 
We only sketch some of their proofs.

\begin{proposition}
   \label{prop:interiorstrip}
   Let $T$ and $A$ be as in Assumption~\ref{ass:relbounded} and assume that there are constants $\gamma_1 < \gamma_2$ such that
   \begin{equation*}
      \{z\in\C : \gamma_1 < \re z < \gamma_2 \} \subseteq \rho(T)
   \end{equation*}
   and set $\wgamma = \max\{ |\gamma_1|,\, |\gamma_2| \}$.
   Then 
   \begin{multline*}
      (\gamma_1 + \I\hyperbola_{\wgamma})
      \cap
      (\gamma_2 - \I\hyperbola_{\wgamma})
      \\
      \begin{aligned}
	 &= \left\{
	 \gamma_1 + \sqrt{ \frac{a^2 + b^2\wgamma^2}{1-b^2} + \frac{b^2}{1-b^2} (\im z)^2}
	 < \re z < 
	 \gamma_2 - \sqrt{ \frac{a^2 + b^2\wgamma^2}{1-b^2} + \frac{b^2}{1-b^2} (\im z)^2}
	 \right\}
	 \\
	 &\subseteq \rho(T+A).
      \end{aligned}
   \end{multline*}
\end{proposition}
\begin{proof}
   We write
   $\sigma(T) \subseteq \{\re z \le \gamma_1 \} \cup \{\re z \ge \gamma_2 \}$.
   We know that a sufficient condition for $z\in \rho(T+A)$ is that
   $\sup_{t\in\partial M} H_z(t) < 1$ for 
   $M = \{\re z \le \gamma_2 \}$ and $M = \{\re z \ge \gamma_1 \}$.
   The claim then follows as in Proposition~\ref{prop:imupperbdd}.
   See Figure~\ref{fig:intstrip}.
\end{proof}

\begin{figure}[H] 
   \centering
   \begin{tikzpicture}[ 
      scale=0.5, transform shape, 
      declare function={ 
      hyperbolax(\x)     = \gx      - sqrt( (\abound^2 + \bbound^2*(\gx)^2      + \bbound^2*\x*\x)/(1-\bbound^2) ); 
      hyperbolaleft(\x)  = \gxupper - sqrt( (\abound^2 + \bbound^2*(\gxupper)^2 + \bbound^2*\x*\x)/(1-\bbound^2) ); 
      hyperbolaright(\x) = \gxlower + sqrt( (\abound^2 + \bbound^2*(\gxlower)^2 + \bbound^2*\x*\x)/(1-\bbound^2) ); 
      }
      ]

      \tikzmath{
      \gx = 2;  
      \gxlower = -2.0;  
      \gxupper = 2.0;  
      \abound = .5;   
      \bbound = .2;  
      }

      \begin{scope} 
	 \clip (-5,-4) rectangle (5,4);

	 \path [fill=colTspec, opacity=1] (-5, -5) rectangle (-\gx, 5);
	 \path [fill=colTspec, opacity=1] (\gx, -5) rectangle (5, 5);

	 \draw[blue] (\gx,-0.1) --  node[below left] {$\gamma_{2}$} (\gx,0.1);
	 \draw[blue] (-\gx,-0.1) --  node[below right] {$\gamma_{1}$} (-\gx,0.1);

	 \draw[->] (-5.5,0)--(5.5,0);
	 \draw[->] (0,-4)--(0,4);
      \end{scope}

      \begin{scope}[xshift=15cm] 
	 \clip (-5,-4) rectangle (5,4);

	 \draw [name path=leftcurve,red,dashed, thick] plot[domain=-5.0:5, smooth] ({hyperbolax(\x)} ,{\x});
	 \draw [name path=rightccurve,red,dashed, thick] plot[domain=-5.0:5, smooth] ({-hyperbolax(\x)},{\x});

	 \fill [colTAspec, opacity=.3] plot[domain=-5.0:5, smooth] ({hyperbolax(\x)} ,{\x}) -- (7,7) -- (7, -7) -- cycle; 
	 \fill [colTAspec, opacity=.3] plot[domain=-5.0:5, smooth] ({-hyperbolax(\x)} ,{\x}) -- (-7,7) -- (-7, -7) -- cycle;

	 \path [fill=colTspec, opacity=1] (-5, -5) rectangle (-\gx, 5);
	 \path [fill=colTspec, opacity=1] (\gx, -5) rectangle (5, 5);
	 
	 \draw[->] (-5.5,0)--(5.5,0);
	 \draw[->] (0,-4)--(0,4);

      \end{scope}

      \tikzmath{
      \gxlower = 1.0;  
      \gxupper = 5.0;  
      \abound = .5;   
      \bbound = .2;  
      }

      \begin{scope}[yshift=-10cm, xshift=-3cm] 
	 \clip (-2,-4) rectangle (8,4);

	 \path [fill=colTspec, opacity=1] (-5, -5) rectangle (\gxlower, 8);
	 \path [fill=colTspec, opacity=1] (\gxupper, -5) rectangle (8, 8);

	 \draw[blue] (\gxupper,-0.1) --  node[below=.2] {$\gamma_{2}$} (\gxupper,0.1);
	 \draw[blue] (\gxlower,-0.1) --  node[below=.2] {$\gamma_{1}$} (\gxlower,0.1);

	 \draw[->] (-5.5,0)--(8,0);
	 \draw[->] (0,-4)--(0,4);
      \end{scope}

      \begin{scope}[yshift=-10cm, xshift=12cm] 
	 \clip (-2,-4) rectangle (8,4);

	 \draw [name path=rightcurve,red,dashed, thick] plot[domain=-5.0:8, smooth] ({hyperbolaleft(\x)},{\x});
	 \draw [name path=leftcurve,red,dashed, thick] plot[domain=-5.0:8, smooth] ({hyperbolaright(\x)},{\x});

	 \fill [gray, opacity=.3] plot[domain=-5.0:8, smooth] ({hyperbolaleft(\x)} ,{\x}) -- (10,8) -- (10, -7) -- cycle; 
	 \fill [gray, opacity=.3] plot[domain=-5.0:8, smooth] ({hyperbolaright(\x)} ,{\x}) -- (-3,7) -- (-3, -7) -- cycle; 

	 \path [fill=colTspec, opacity=1] (-5, -5) rectangle (\gxlower, 8);
	 \path [fill=colTspec, opacity=1] (\gxupper, -5) rectangle (8, 8);

	 \draw[->] (-5.5,0)--(8,0);
	 \draw[->] (0,-4)--(0,4);
      \end{scope}

   \end{tikzpicture}

   \caption{
   Illustration of Proposition~\ref{prop:interiorstrip} when $\sigma(T)\subseteq \C\setminus
   \{ \gamma_{1}\le \re(z) \le \gamma_{2} \}$.
   In the upper row we have the special case $\gamma_1 = -\gamma_2$.
   The dashed red lines show the boundaries of the corresponding hyperbolas.
   The spectrum of $T+A$ is contained in coloured regions.
   }
   \label{fig:intstrip}
\end{figure}
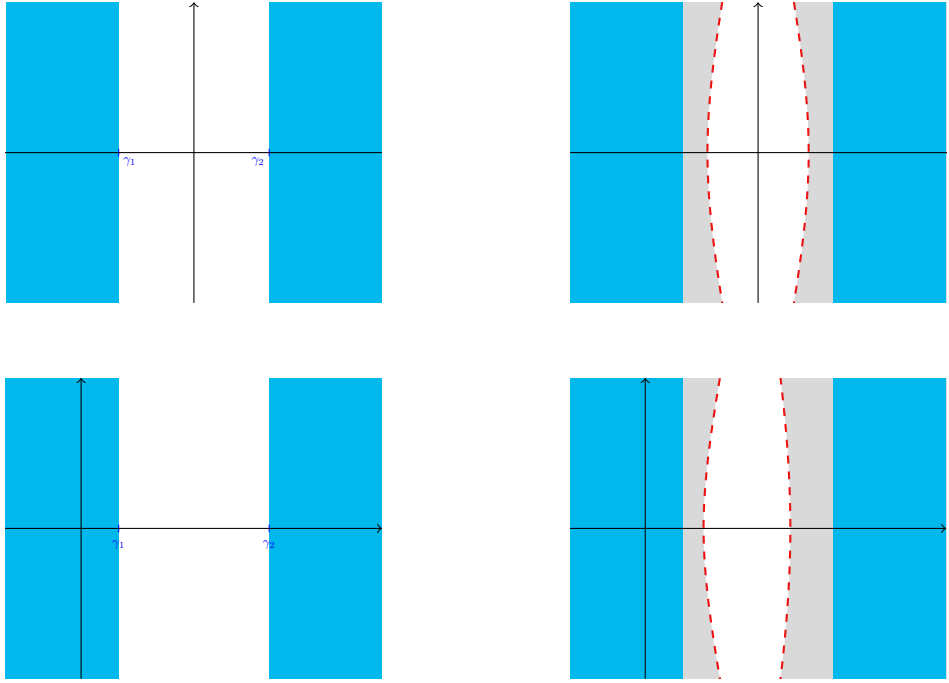

\begin{proposition}
   \label{prop:square}
   Let $T$ and $A$ be as in Assumption~\ref{ass:relbounded} and assume that there are constants $\gamma_j$ and $\eta_j$ $(j=1, 2)$ such that
   \begin{equation*}
      \sigma(T)\subseteq \C\setminus \{z\in\C : \gamma_1 < \re z < \gamma_2,\  \eta_1 < \im z < \eta_2 \}
   \end{equation*}
   and set $\wgamma = \max\{ |\gamma_1|,\, |\gamma_2| \}$, $\weta = \max\{ |\eta_1|,\, |\eta_2| \}$.
   Then 
   \begin{equation*}
      (\gamma_1 + \I\hyperbola_{\wgamma})
      \cap
      (\gamma_2 - \I\hyperbola_{\wgamma})
      \cap
      (\I\eta_1 - \hyperbola_{\weta})
      \cap
      (\I\eta_2 + \hyperbola_{\weta})
      \subseteq \rho(T+A).
   \end{equation*}
\end{proposition}
\begin{proof}
   We write 
   $\sigma(T)$ as union of the half-planes
   $\{\im z \ge \eta_2 \},\, \{\im z \le \eta_2 \},\, \{\re z \le \gamma_1 \}$ and $\{\re z \ge \gamma_2 \}$.
   A sufficient condition for $z\in \rho(T+A)$ is that
   $\sup_{t\in\partial M} H_z(t) < 1$ for all of the four half-planes $M$.
   The claim then follows as in Proposition~\ref{prop:imupperbdd}.
   See Figure~\ref{fig:square}.
\end{proof}

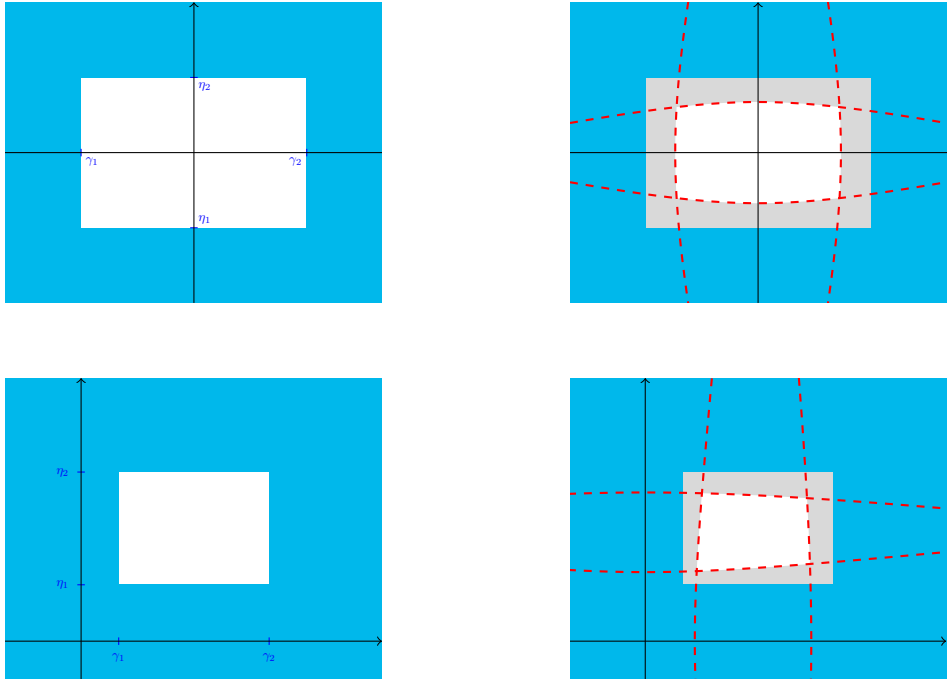
\begin{figure}[H] 
   \centering
   \begin{tikzpicture}[ 
      scale=0.5, transform shape, 
      declare function={ 
      hyperbolax(\x) = \gx - sqrt( (\abound^2 + \bbound^2*\gx^2 + \bbound^2*\x*\x)/(1-\bbound^2) ); 
      hyperbolay(\x) = \gy - sqrt( (\abound^2 + \bbound^2*\gy^2 + \bbound^2*\x*\x)/(1-\bbound^2) ); 
      hyperboladown(\x) = \gyupper - sqrt( (\abound^2 + \bbound^2*(\gyupper)^2 + \bbound^2*\x*\x)/(1-\bbound^2) ); 
      hyperbolaup(\x) = \gylower + sqrt( (\abound^2 + \bbound^2*(\gylower)^2 + \bbound^2*\x*\x)/(1-\bbound^2) ); 
      hyperbolaleft(\x) = \gxupper - sqrt( (\abound^2 + \bbound^2*(\gxupper)^2 + \bbound^2*\x*\x)/(1-\bbound^2) ); 
      hyperbolaright(\x) = \gxlower + sqrt( (\abound^2 + \bbound^2*(\gxlower)^2 + \bbound^2*\x*\x)/(1-\bbound^2) ); 
      },
      ]

      \tikzmath{
      \gx = 3;  
      \gy = 2;  
      \abound = .5;   
      \bbound = .2;  
      }

      \begin{scope} 
	 \clip (-5,-4) rectangle (5,4);

	 \path [fill=colTspec, opacity=1] (\gx, -5) rectangle (5, 5);
	 \path [fill=colTspec, opacity=1] (-\gx, -5) rectangle (-5, 5);
	 \path [fill=colTspec, opacity=1] (-5, \gy) rectangle (5, 5);
	 \path [fill=colTspec, opacity=1] (-5, -\gy) rectangle (5, -5);

	 \draw[blue] (\gx,-0.1) --  node[below left] {$\gamma_{2}$} (\gx,0.1);
	 \draw[blue] (-\gx,-0.1) --  node[below right] {$\gamma_{1}$} (-\gx,0.1);
	 \draw[blue] (-0.1,\gy) --  node[below right] {$\eta_{2}$} (0.1,\gy);
	 \draw[blue] (-0.1,-\gy) --  node[above right] {$\eta_{1}$} (0.1,-\gy);

	 \draw[->] (-5.5,0)--(5.5,0);
	 \draw[->] (0,-4)--(0,4);
      \end{scope}

      \begin{scope}[xshift=15cm] 
	 \clip (-5,-4) rectangle (5,4);

	 \path [save path=\myuppercurve, name path=uppercurve] plot[domain=-5.0:5, smooth] ({\x},{hyperbolay(\x)});
	 \path [save path=\mylowercurve, name path=lowercurve] plot[domain=-5.0:5, smooth] ({\x},{-hyperbolay(\x)});
	 \path [save path=\myleftcurve, name path=leftcurve] [rotate=90] plot[domain=-5.0:5, smooth] ({\x},{hyperbolax(\x)});
	 \path [save path=\myrightcurve, name path=rightcurve] [rotate=90] plot[domain=-5.0:5, smooth] ({\x},{-hyperbolax(\x)});

	 \path [name intersections={of=uppercurve and leftcurve, by=UL}] ;
	 \path [name intersections={of=uppercurve and rightcurve, by=UR}];
	 \path [name intersections={of=lowercurve and rightcurve, by=LR}];
	 \path [name intersections={of=lowercurve and leftcurve, by=LL}] ;


	 \tikzset{
	 declare function={pointstocm(\t) = \t*0.03514598;}, 
	 }

	 \tikzmath{
	 coordinate \UL; coordinate \UR; coordinate \LL; coordinate \LR;
	 \UL = (UL); \UR = (UR); \LL = (LL); \LR = (LR); 
	 \ULconvx = pointstocm(\ULx);
	 \ULconvy = pointstocm(\ULy);
	 \URconvx = pointstocm(\URx);
	 \URconvy = pointstocm(\URy);
	 \LLconvx = pointstocm(\LLx);
	 \LLconvy = pointstocm(\LLy);
	 \LRconvx = pointstocm(\LRx);
	 \LRconvy = pointstocm(\LRy);
	 }

	 \begin{scope} [even odd rule]
	    \clip
	    (-8,-8) rectangle (8,8)
	    plot[domain=\ULconvx:\URconvx, smooth] ({\x},{hyperbolay(\x)})
	    -- plot[domain=\URconvy:\LRconvy, smooth] ({hyperbolax(\x)},{\x})
	    -- plot[domain=\LRconvx:\LLconvx, smooth] ({\x},{-hyperbolay(\x)})
	    -- plot[domain=\LLconvy:\ULconvy, smooth] ({-hyperbolax(\x)}, {\x})
	    -- cycle;
	    
	    \fill[gray, opacity=.3](-8,-8) rectangle (8,8);
	 \end{scope}

	 \path [fill=colTspec, opacity=1] (\gx, -5) rectangle (5, 5);
	 \path [fill=colTspec, opacity=1] (-\gx, -5) rectangle (-5, 5);
	 \path [fill=colTspec, opacity=1] (-5, \gy) rectangle (5, 5);
	 \path [fill=colTspec, opacity=1] (-5, -\gy) rectangle (5, -5);

	 \draw [use path=\myuppercurve,red,dashed, thick];
	 \draw [use path=\mylowercurve,red,dashed, thick];
	 \draw [use path=\myleftcurve, red,dashed, thick];
	 \draw [use path=\myrightcurve,red,dashed, thick];

	 \draw[->] (-5.5,0)--(5.5,0);
	 \draw[->] (0,-4)--(0,4);
      \end{scope}

      \tikzmath{
      \gxlower = 1.0;  
      \gxupper = 5.0;  
      \gyupper = 4.5;  
      \gylower = 1.5;  
      \abound = 1;   
      \bbound = .5;  
      }

      \begin{scope}[yshift=-13cm, xshift=-3cm] 
	 \clip (-2,-1) rectangle (8,7);

	 \path [fill=colTspec, opacity=1] (\gxupper, -5) rectangle (10, 10);
	 \path [fill=colTspec, opacity=1] (\gxlower, -5) rectangle (-5, 10);
	 \path [fill=colTspec, opacity=1] (-5, \gyupper) rectangle (10, 10);
	 \path [fill=colTspec, opacity=1] (-5, \gylower) rectangle (10, -5);

	 \draw[blue] (\gxupper,-0.1) --  node[below=.2] {$\gamma_{2}$} (\gxupper,0.1);
	 \draw[blue] (\gxlower,-0.1) --  node[below=.2] {$\gamma_{1}$} (\gxlower,0.1);
	 \draw[blue] (-0.1,\gyupper) --  node[left=.2] {$\eta_{2}$} (0.1,\gyupper);
	 \draw[blue] (-0.1,\gylower) --  node[left=.2] {$\eta_{1}$} (0.1,\gylower);

	 \draw[->] (-5.5,0)--(8,0);
	 \draw[->] (0,-4)--(0,7);
      \end{scope}

      \begin{scope}[yshift=-13cm, xshift=12cm] 
	 \clip (-2,-1) rectangle (8,7);




	 \tikzmath{\abound=0.3; \bbound=0.1;}
	 \path [save path=\myuppercurve, name path=uppercurve] plot[domain=-5.0:8, smooth] ({\x},{hyperboladown(\x)});
	 \path [save path=\mylowercurve, name path=lowercurve] plot[domain=-5.0:8, smooth] ({\x},{hyperbolaup(\x)});
	 \path [save path=\myrightcurve, name path=rightcurve] plot[domain=-5.0:8, smooth] ({hyperbolaleft(\x)},{\x});
	 \path [save path=\myleftcurve , name path=leftcurve ] plot[domain=-5.0:8, smooth] ({hyperbolaright(\x)},{\x});

	 \path [name intersections={of=uppercurve and leftcurve, by=UL}] ;
	 \path [name intersections={of=uppercurve and rightcurve, by=UR}];
	 \path [name intersections={of=lowercurve and rightcurve, by=LR}];
	 \path [name intersections={of=lowercurve and leftcurve, by=LL}] ;


	 \tikzset{
	 declare function={pointstocm(\t) = \t*0.03514598;}, 
	 }

	 \tikzmath{
	 coordinate \UL; coordinate \UR; coordinate \LL; coordinate \LR;
	 \UL = (UL); \UR = (UR); \LL = (LL); \LR = (LR); 
	 \ULconvx = pointstocm(\ULx);
	 \ULconvy = pointstocm(\ULy);
	 \URconvx = pointstocm(\URx);
	 \URconvy = pointstocm(\URy);
	 \LLconvx = pointstocm(\LLx);
	 \LLconvy = pointstocm(\LLy);
	 \LRconvx = pointstocm(\LRx);
	 \LRconvy = pointstocm(\LRy);
	 }

	 \begin{scope} [even odd rule]
	    \clip
	    (-8,-8) rectangle (8,8)
	    plot[domain=\ULconvx:\URconvx, smooth] ({\x},{hyperboladown(\x)})
	    -- plot[domain=\URconvy:\LRconvy, smooth] ({hyperbolaleft(\x)},{\x})
	    -- plot[domain=\LRconvx:\LLconvx, smooth] ({\x},{hyperbolaup(\x)})
	    -- plot[domain=\LLconvy:\ULconvy, smooth] ({hyperbolaright(\x)}, {\x})
	    -- cycle;
	    
	    \fill[gray, opacity=.3](-8,-8) rectangle (8,8);

	 \end{scope}

	 \path [fill=colTspec, opacity=1] (\gxupper, -5) rectangle (10, 10);
	 \path [fill=colTspec, opacity=1] (\gxlower, -5) rectangle (-5, 10);
	 \path [fill=colTspec, opacity=1] (-5, \gyupper) rectangle (10, 10);
	 \path [fill=colTspec, opacity=1] (-5, \gylower) rectangle (10, -5);

	 \tikzmath{\abound=0.3; \bbound=0.1;}
	 \draw [use path=\myuppercurve,red,dashed, thick];
	 \draw [use path=\mylowercurve,red,dashed, thick];
	 \draw [use path=\myrightcurve,red,dashed, thick];
	 \draw [use path=\myleftcurve,red,dashed, thick];

	 \draw[->] (-5.5,0)--(8,0);
	 \draw[->] (0,-4)--(0,7);
      \end{scope}

   \end{tikzpicture}

   \caption{
   Illustration of Proposition~\ref{prop:square} when $\sigma(T)\subseteq \C\setminus
   \{ \gamma_{1}\le \re(z) \le \gamma_{2},\ \eta_{1} \le \im(z) \le \eta_{2} \}$.
   In the upper row we have the special case $\gamma_1 = -\gamma_2, \eta_1 = -\eta_2$.
   The dashed red lines show the boundaries of the corresponding hyperbolas.
   The white area is contained in $\rho(T+A)$.
   }
   \label{fig:square}
\end{figure}

\begin{proposition}
   \label{prop:disk}
   Let $T$ and $A$ be as in Assumption~\ref{ass:relbounded} and assume that there is a constant $R>0$ such that 
   \begin{equation*}
      \sigma(T)\subseteq \C\setminus \{z\in\C : |z|\le R \}.
   \end{equation*}
   If $r := R - \sqrt{a^2 + b^2 R^2} > 0$, then $B_r(0)\subseteq\rho(T+A)$.
\end{proposition}
\begin{proof}
   Again we will use that 
   $\sup_{t\in\sigma(T)} H_z(t) < 1$ is
   a sufficient condition for $z\in \rho(T+A)$.
   We estimate
   \begin{equation*}
   \sup_{t\in\sigma(T)} H_z(t)
   \le \sup_{|t|\ge R} H_z(t)
   = \max\left\{ b^2, \sup_{|t|= R} H_z(t) \right\}.
   \end{equation*}
   Note that
   $\sup_{|t|= R} H_z(t)
   = \sup_{|t|= R} \frac{a^2 + b^2|t|^2}{|z-t|^2}
   = \sup_{|t|= R} \frac{a^2 + b^2R^2}{|z-t|^2}$.
   Clearly, the supremum is attained when $\arg t = \arg z$, so
   $\sup_{|t|= R} H_z(t)
   =  \frac{a^2 + b^2R^2}{(R - |z|)^2}$.
   This is smaller than 1 if and only if $|z| < R - \sqrt{a^2 + b^2 R^2} =: r$.
   Therefore, such $z$ belong to $\rho(T+A)$ by Proposition~\ref{prop:ourfirsprop3}.
   See Figure~\ref{fig:disk}.
\end{proof}

\begin{figure}[H] 
   \centering
   \begin{tikzpicture}[ scale=0.5, transform shape ]

      \tikzmath{
      \abound = .5;   
      \bbound = .2;  
      \RR = 3;  
      \rr = \RR - sqrt( 1/(\abound^2 + \bbound^2*\RR^2) );  
      }

      \begin{scope} 
	 \clip (-5,-4) rectangle (5,4);

	 \path [fill=colTspec, opacity=1] (-5, -5) rectangle (5, 5);
	 \path [fill=white, opacity=1] circle[radius=\RR];

	 \draw[->] (-5.5,0)--(5.5,0);
	 \draw[->] (0,-4)--(0,4);

	 \draw[-] (\RR,.2)--(\RR ,-.2) node[below]{$R$};
      \end{scope}

      \begin{scope}[xshift=15cm] 
	 \clip (-5,-4) rectangle (5,4);

	 \path [fill=colTspec, opacity=1] (-5, -5) rectangle (5, 5);
	 \path [fill=white, opacity=1] circle[radius=\RR];

	 \begin{scope} 
	    \clip circle[radius=\RR];
	    \path [fill=gray, opacity=.3] (-5, -5) rectangle (5, 5);
	    \path [fill=white, opacity=1] circle[radius=\rr];
	 \end{scope}

	 \draw[->] (-5.5,0)--(5.5,0);
	 \draw[->] (0,-4)--(0,4);

	 \draw[-] (\RR,.2)--(\RR ,-.2) node[below]{$R$};
	 \draw[-] (\rr,.2)--(\rr ,-.2) node[below]{$r$};

      \end{scope}

   \end{tikzpicture}

   \caption[$\sigma(T)\subseteq \C\setminus \{ |z| < R\}$.]{
   Illustration of Proposition~\ref{prop:disk}.
   If $\sigma(T)\subseteq \C\setminus \{ |z| < R\}$, then $\sigma(T+A)\subseteq \C\setminus \{ |z| < r\}$ 
   with $r = R - \sqrt{a^2 + b^2 R^2}$.
   }
   \label{fig:disk}
\end{figure}
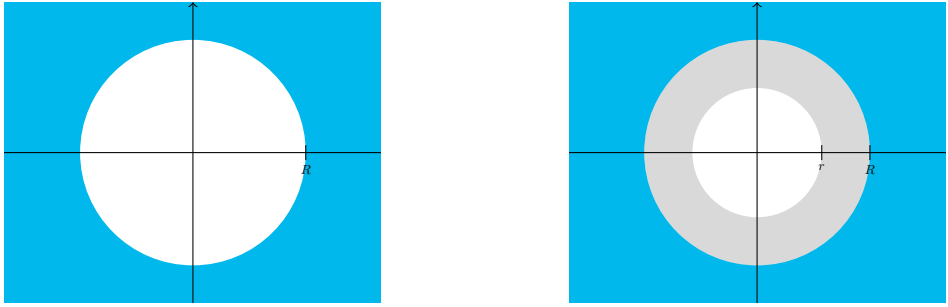

In Figure~\ref{fig:differentangles} we show the spectral enclosures that can be obtained as in Theorem~\ref{thm:bisecgap} when the sectors have different angles.
Figure~\ref{fig:manysectors} shows the spectral enclosures if we have more than two sectors.
In Figure~\ref{fig:manysectors}, all sectors have the same angle and the same distance from the origin, however it is not difficult to adapt the method of proof of Theorem~\ref{thm:bisecgap} to situations with different angles or different distances from the origin.

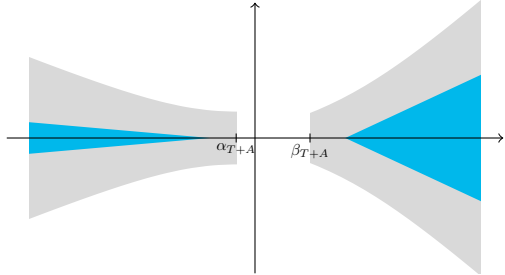
\begin{figure}%
   [H] 

   \begin{center}
      \begin{tikzpicture}[  
	 declare function={ 
	 hyperbolarotb(\x) = sqrt( (\abound^2 + \bbound^2*(\betaT)^2  + \bbound^2*(\x)^2)/(1-\bbound^2) ); 
	 hyperbolarota(\x) = sqrt( (\abound^2 + \bbound^2*(\alphaT)^2 + \bbound^2*(\x)^2)/(1-\bbound^2) ); 
	 },
	 scale=0.6, transform shape
	 ]

	 \tikzmath{
	 \gT = 0;  
	 \abound = .5;   
	 \bbound = .3;  
	 \alphaT = -1;   
	 \betaT = 2;      
	 \thetaA = 5;      
	 \thetaB = 25;      
	 \alphaTA = \alphaT + sqrt( \abound^2 + \bbound^2*\gT^2 + \bbound^2*(\alphaT)^2);   
	 \betaTA = \betaT - sqrt( \abound^2 + \bbound^2*\gT^2 + \bbound^2*(\betaT)^2);      
	 }

	 \begin{scope}
	    \clip (-5, -3) rectangle (\alphaTA, 3);

	    \fill [colTAspec, opacity=.3, shift={(\alphaT,0)} ]
	    plot[domain=-7.0:7.0, smooth, rotate=-\thetaA] ({\x},{hyperbolarota(\x)}) 
	    -- plot[domain=7.0:-7.0, smooth, rotate=\thetaA] ({\x},{-hyperbolarota(\x)}) -- cycle;
	 \end{scope}

	 \begin{scope}
	    \clip (\betaTA,-3) rectangle (5, 5);

	    \fill [colTAspec, opacity=.3, shift={(\betaT,0)} ]
	    plot[domain=-7.0:7.0, smooth, rotate=\thetaB] ({\x},{hyperbolarotb(\x)}) 
	    --  plot[domain=7.0:-7.0, smooth, rotate=-\thetaB] ({\x},{-hyperbolarotb(\x)}) -- cycle;

	 \end{scope}

	 \draw (\alphaTA,-.1) -- node [below] {$\alpha_{T+A}$} (\alphaTA,.1);
	 \draw (\betaTA,-.1) -- node [below] {$\beta_{T+A}$} (\betaTA,.1);

	 \begin{scope}
	    \clip (-5,-3) rectangle (5, 5);
	    \path [fill=colTspec, opacity=1, shift={(\alphaT,0)}] (\thetaA:-6) -- (0,0) --  (-\thetaA:-6) -- cycle;
	    \path [fill=colTspec, opacity=1, shift={(\betaT,0)}] (\thetaB:6) -- (0,0) --  (-\thetaB:6) -- cycle;
	 \end{scope}

	 \draw[->] (-5.5,0)--(5.5,0);
	 \draw[->] (0,-3)--(0,3);


      \end{tikzpicture}
   \end{center}

   \caption{Spectral inclusion for $T+A$ if $\sigma(T)\subseteq -\sector{\theta_\alpha}{-\alpha} \cup -\sector{\theta_\beta}{\beta}$ with $\theta_\alpha\neq \theta_\beta$.}
   \label{fig:differentangles}

\end{figure}

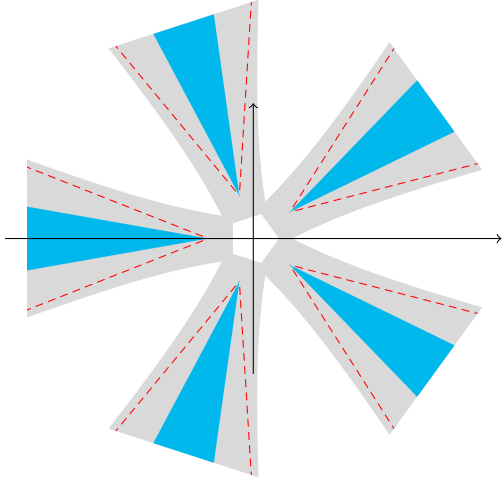
\begin{figure}
   \begin{center}

   \begin{tikzpicture}[  
      declare function={ 
      hyperbolarotb(\x) = sqrt( (\abound^2 + \bbound^2*(\betaT)^2  + \bbound^2*(\x)^2)/(1-\bbound^2) ); 
      hyperbolarota(\x) = sqrt( (\abound^2 + \bbound^2*(\alphaT)^2 + \bbound^2*(\x)^2)/(1-\bbound^2) ); 
      },
      scale=0.6, transform shape]

      \tikzmath{
      \gT = 0;  
      \abound = .5;   
      \bbound = .2;  
      \alphaT = -1;   
      \betaT = 2;      
      \thetaT = 10;      
      \alphaTA = \alphaT + sqrt( \abound^2 + \bbound^2*\gT^2 + \bbound^2*(\alphaT)^2);   
      \betaTA = \betaT - sqrt( \abound^2 + \bbound^2*\gT^2 + \bbound^2*(\betaT)^2);      
      \thetaB = atan(\bbound/sqrt(1-\bbound^2));
      }

      \tikzset{mywedge/.pic={
      \clip (-5, -3) rectangle (\alphaTA, 3);


      \fill [colTAspec!30, opacity=1, shift={(\alphaT,0)} ]
      plot[domain=-7.0:7.0, smooth, rotate=-\thetaT] ({\x},{hyperbolarota(\x)}) 
      -- plot[domain=7.0:-7.0, smooth, rotate=\thetaT] ({\x},{-hyperbolarota(\x)}) -- cycle;

      \draw [red, densely dashed] plot [domain=-7.0:\alphaT, smooth] ({\x},{(\x-\alphaT) * tan(\thetaT+\thetaB)});
      \draw [red, densely dashed] plot [domain=-7.0:\alphaT, smooth] ({\x},{-(\x-\alphaT) * tan(\thetaT+\thetaB)});

      \path [fill=colTspec, opacity=1, shift={(\alphaT,0)}] (\thetaT:-6) -- (0,0) --  (-\thetaT:-6) -- cycle;
      }}

      \draw [rotate=0] (0,0) pic {mywedge};
      \draw [rotate=72] (0,0) pic {mywedge};
      \draw [rotate={2*72}] (0,0) pic {mywedge};
      \draw [rotate={3*72}] (0,0) pic {mywedge};
      \draw [rotate={4*72}] (0,0) pic {mywedge};


      \draw[->] (-5.5,0)--(5.5,0);
      \draw[->] (0,-3)--(0,3);

   \end{tikzpicture}

   \end{center}
   
   \caption{Spectral inclusion for $T+A$ if the $\sigma(T)$ is contained in several sectors symmetrically distributed around $0$}
  \label{fig:manysectors}

\end{figure}

\subsection*{Infinitely many gaps} 

In this subsection we apply the results from Section~\ref{sec:boundedimaginary} to normal operators with infinitely many spectral gaps along the real axis.
First we consider the case when the imaginary part of the spectrum of $T$ is bounded.
The next theorem can be considered a generalisation of Theorem~\ref{thm:bddstripgap} to the case of infinitely many spectral free strips.

\begin{theorem} 
   \label{thm:infgapsbdd}
   Let $T$ and $A$ be as in Assumption~\ref{ass:relbounded}. 
   As in Proposition~\ref{prop:bddstrip} we assume that $\sigma(T)\subseteq \{z\in\C: \gamma_1\le \im z \le \gamma_2\}$ for some $\gamma_1< \gamma_2$.
   Let us set $\widetilde \gamma = \max\{|\gamma_1|,\, |\gamma_2|\}$.
   Assume in addition that there are unbounded sequences $(\alpha_n)_n$ and  $(\beta_n)_n$ such that $\alpha_n < \beta_n \le \alpha_{n+1}$ and $\{z\in\C : \alpha_n < \re z < \beta_n\}\subseteq\rho(T)$ for every $n$.

   \begin{enumerate}[label={\upshape(\roman*)}]

      \item 
      \label{item:infgapsbdd:limsup}
      The resolvent set of the operator $T+A$ contains infinitely many strips parallel to the imaginary axis if
      \begin{equation}
	 \label{eq:infgapsbdd:i}
	 \liminf_{n\to\infty}
	 \frac{ \sqrt{ a^2 + b^2 \widetilde \gamma^2 + b^2\alpha_n^2} + \sqrt{ a^2 + b^2 \widetilde \gamma^2 + b^2\beta_n^2} }{ \beta_n-\alpha_n } 
	 < 1.
      \end{equation}
      If $b=0$, then inequality \eqref{eq:infgapsbdd:i} holds if and only if
      $\limsup \beta_n - \alpha_n > 2a$;
      if $0<b<1$, then it holds if and only if 
      \begin{equation}
	 \label{eq:infgapsbdd:ii}
	 \limsup_{n\to\infty} \frac{\beta_n}{\alpha_n} > \frac{1+b}{1-b}.
      \end{equation}

      \item 
      \label{item:infgapsbdd:liminf}
      At most finitely many spectral gaps of $T$ close if
      \begin{equation}
	 \label{eq:infgapsbdd:iii}
	 \limsup_{n\to\infty}
	 \frac{ \sqrt{ a^2 + b^2 \widetilde \gamma^2 + b^2\alpha_n^2} + \sqrt{ a^2 + b^2 \widetilde \gamma^2 + b^2\beta_n^2} }{ \beta_n-\alpha_n } 
	 < 1.
      \end{equation}
      If $b=0$, then inequality \eqref{eq:infgapsbdd:iii} holds if and only if
      $\liminf \beta_n - \alpha_n > 2a$;
      if $0<b<1$, then it holds if and only if 
      \begin{equation}
	 \label{eq:infgapsbdd:vi}
	 \liminf_{n\to\infty} \frac{\beta_n}{\alpha_n} > \frac{1+b}{1-b}.
      \end{equation}

   \end{enumerate}
\end{theorem}

The theorem can easily be modified to the situation when $\alpha_n < \beta_n \le \alpha_{n+1}$ for $n\in\Z$ or for $n\in -\N$.

\begin{proof} 
   Let $\alpha_n' = \alpha_n + \sqrt{ a^2 + b^2 \widetilde \gamma^2 + b^2\alpha_n^2}$ and 
   $\beta_n' = \beta_n - \sqrt{ a^2 + b^2 \widetilde \gamma^2 + b^2\beta_n^2}$.
   Proposition~\ref{thm:bddstripgap} guarantees that $\{ \alpha_n' < \re z < \beta_n'\}\subseteq \rho(T+A)$.
   This strip is not empty if and only if $\beta_n' > \alpha_n'$.
   This is the case if and only if 
   \begin{align*}
      \frac{ \sqrt{ a^2 + b^2 \widetilde \gamma^2 + b^2\alpha_n^2} + \sqrt{ a^2 + b^2 \widetilde \gamma^2 + b^2\beta_n^2} }{ \beta_n-\alpha_n }
      < 1.
   \end{align*}
   Therefore the first part of \ref{item:infgapsbdd:limsup} and \ref{item:infgapsbdd:liminf} follows.

   If $b=0$, then $\beta_n' < \alpha_n'$ if and only if $\beta_n-\alpha_n > 2a$.
   Therefore infinitely many spectral gaps remain if 
   $\limsup_{n\to\infty} \beta_n-\alpha_n > 2a$
   and at most finitely many gaps close if 
   $\liminf_{n\to\infty} \beta_n-\alpha_n > 2a$.

   Now let us assume that $b>0$ and set 
   $L := \limsup_{n\to\infty} \frac{\beta_n}{\alpha_n}$.
   Note that $L\ge 1$.
   Then
   \begin{multline*}
      \liminf_{n\to\infty}
      \frac{ \sqrt{ a^2 + b^2 \widetilde \gamma^2 + b^2\alpha_n^2} 
      + \sqrt{ a^2 + b^2 \widetilde \gamma^2 + b^2\beta_n^2} }{ \beta_n-\alpha_n }
      \\
      \begin{aligned}
	 &= \liminf_{n\to\infty}
	 \frac{ \sqrt{ (a^2 + b^2 \widetilde \gamma^2)/\alpha_n^2 + b^2} 
	 + \sqrt{ (a^2 + b^2 \widetilde \gamma^2)/\alpha_n^2 + b^2 \beta_n^2/\alpha_n^2} }{ \frac{\beta_n}{\alpha_n} - 1 }
	 \\
	 &= \liminf_{n\to\infty}
	 \frac{ b (\frac{\beta_n}{\alpha_n} + 1) }{ \frac{\beta_n}{\alpha_n} - 1 }
	 = b \liminf_{n\to\infty}
	 \left( 1 + \frac{ 2 }{ \frac{\beta_n}{\alpha_n} - 1 } \right)
	 = b\left( 1 + \frac{ 2 }{ L - 1 } \right)
	 = \frac{  b( L+1) }{ L - 1 }
      \end{aligned}
   \end{multline*}
   with the convention that the latter is $\infty$ if $L=1$.
   This is smaller than 1 if and only if \eqref{eq:infgapsbdd:ii} holds hence \ref{item:infgapsbdd:limsup} is proved.
   The calculation for \ref{item:infgapsbdd:liminf} is analogous.
\end{proof}

If in Theorem~\ref{thm:infgapsbdd} we drop the assumption that the imaginary part of the spectrum of $T$ is bounded, we cannot expect that the resolvent set of $T+A$ contains any strip.
The case of one such spectral strip is dealt with in Proposition~\ref{prop:interiorstrip}.
There we saw that the area between two hyperbolas, which bend towards each other, is contained in the resolvent set of $T+A$.
Although they touch for $|\im z|$ large enough, a neighbourhood of a real interval is contained in the resolvent set of $T+A$.
For infinitely many such strips, we obtain the following theorem.

\begin{theorem} 
   \label{thm:infgaps}
   Let $T$ and $A$ be as in Assumption~\ref{ass:relbounded} and assume that $b < \frac{1}{\sqrt{5}}$ where $b$ is as in \eqref{eq:definicion2}.
   In addition we assume that there are unbounded sequences $(\alpha_n)_n$ and  $(\beta_n)_n$ such that $\alpha_n < \beta_n \le \alpha_{n+1}$ and $\{z\in\C : \alpha_n < \re z < \beta_n\}\subseteq\rho(T)$ for every $n$.

   \begin{enumerate}[label={\upshape(\roman*)}]

      \item 
      \label{item:infgaps:liminf}
      The resolvent set of the operator $T+A$ contains infinitely many rectangular regions around real intervals if 
      \begin{equation}
	 \label{eq:infgaps:i}
	 \liminf_{n\to\infty}
	 \frac{ 2\sqrt{ \frac{ a^2 + b^2\beta_n^2 }{1-b^2}} }{ \beta_n-\alpha_n } 
	 < 1.
      \end{equation}
      If $b=0$, then inequality \eqref{eq:infgaps:i} holds if and only if
      $\limsup \beta_n - \alpha_n > 2a$;
      if $0<b< \frac{1}{\sqrt 5} $, then it holds if and only if 
      \begin{equation}
	 \label{eq:infgaps:ii}
	 \limsup_{n\to\infty} \frac{\beta_n}{\alpha_n} >  1 + \frac{ 2b }{\sqrt{1-b^2} - 2b}.
      \end{equation}

      \item 
      \label{item:infgaps:limsup}
      At most finitely many spectral gaps of $T$ disappear completely if
      \begin{equation}
	 \label{eq:infgaps:iii}
	 \limsup_{n\to\infty}
	 \frac{ 2\sqrt{ \frac{ a^2 + b^2\beta_n^2 }{1-b^2}} }{ \beta_n-\alpha_n } 
	 < 1.
      \end{equation}
      If $b=0$, then inequality \eqref{eq:infgaps:iii} holds if and only if
      $\liminf \beta_n - \alpha_n > 2a$;
      if $0<b< \frac{1}{\sqrt 5} $, then it holds if and only if 
      \begin{equation}
	 \label{eq:infgaps:iv}
	 \liminf_{n\to\infty} \frac{\beta_n}{\alpha_n} > 1 + \frac{ 2b }{\sqrt{1-b^2} - 2b}.
      \end{equation}

   \end{enumerate}
   See Figure~\ref{fig:infgaps}.
\end{theorem}
The theorem can easily be modified to the situation when $\alpha_n < \beta_n \le \alpha_{n+1}$ for $n\in\Z$ or for $n\in -\N$.
\begin{proof} 
   Let $\alpha_n' = \alpha_n + \sqrt{ \frac{ a^2 + b^2\beta_n^2 }{1-b^2}}$,
   $\beta_n' = \beta_n - \sqrt{ \frac{ a^2 + b^2\beta_n^2 }{1-b^2}}$ and
   \begin{equation*}
      I_n := 
      \left\{ x\in\R  :
      \alpha_n + \sqrt{ \frac{a^2 + b^2\beta_n^2}{1-b^2} + \frac{b^2}{1-b^2} (\im z)^2}
      < x < 
      \beta_n - \sqrt{ \frac{a^2 + b^2\beta_n^2}{1-b^2} + \frac{b^2}{1-b^2} (\im z)^2}
      \right\}.
   \end{equation*}
   Note that for $n$ large enough, $\alpha_n < \beta_n$, therefore
   Proposition~\ref{prop:interiorstrip} guarantees that a open complex neighourhood of $I_n$ is contained in $\rho(T+A)$ for every $n$.
   Clearly, $I_n$ is nonempty if and only if $\alpha_n' < \beta_n'$
   and in this case, for every $\alpha_n' < \alpha_n'' < \beta_n'' < \beta_n'$ we can find $\epsilon_n>0$ such that 
   $\{z\in\C : \alpha_n'' < \re z < \beta_n'',\ -\epsilon_n < \im z < \epsilon_n \}\subseteq\rho(T+A)$.
   It remains to check for how many $n$ the inequality $\alpha_n' < \beta_n'$ holds.
   This can be done as in Theorem~\ref{thm:infgapsbdd}.
   For the proof that 
   \eqref{eq:infgaps:i} and \eqref{eq:infgaps:ii} are equivalent,
   we set $L := \limsup_{n\to\infty} \frac{\beta_n}{\alpha_n}$. 
   Note that $L\ge 1$. 
   If $b=0$, then \eqref{eq:infgaps:i} is equivalent to $\limsup (\beta_n - \alpha_n) > 2a$.
   If $0 < b < \frac{1}{\sqrt{2}}$, then 
   \begin{align*}
      \liminf_{n\to\infty}
      \frac{ 2 \sqrt{ \frac{ a^2 + b^2\beta_n^2 }{1-b^2}} }{ \beta_n-\alpha_n } 
      &= \liminf_{n\to\infty} \frac{ 2 \sqrt{ a^2 + b^2\beta_n^2 } }{ \sqrt{1-b^2} (\beta_n-\alpha_n) } 
      = \liminf_{n\to\infty} \frac{ 2 \sqrt{ a^2/\alpha_n^2 + b^2\beta_n^2/\alpha_n^2 } }{ \sqrt{1-b^2} (\beta_n/\alpha_n - 1) } 
      \\
      &= \liminf_{n\to\infty} \frac{ b }{ \sqrt{1-b^2} } \cdot \frac{ 2\beta_n/\alpha_n }{ \beta_n/\alpha_n - 1 } 
      = \frac{ b }{ \sqrt{1-b^2} } \cdot \frac{ 2L }{ L - 1 } .
   \end{align*}
   Note that this is smaller than 1 if and only if $L( \sqrt{1-b^2} - 2b) > \sqrt{1-b^2}$.
   By assumption, $b< \frac{1}{\sqrt{5}}$, so the term in brackets on the left hand side is positive and we have that \eqref{eq:infgaps:i} holds if and only if \eqref{eq:infgaps:ii} holds.
   The equivalence of \eqref{eq:infgaps:iii} and \eqref{eq:infgaps:iv} is shown analogously.
\end{proof}

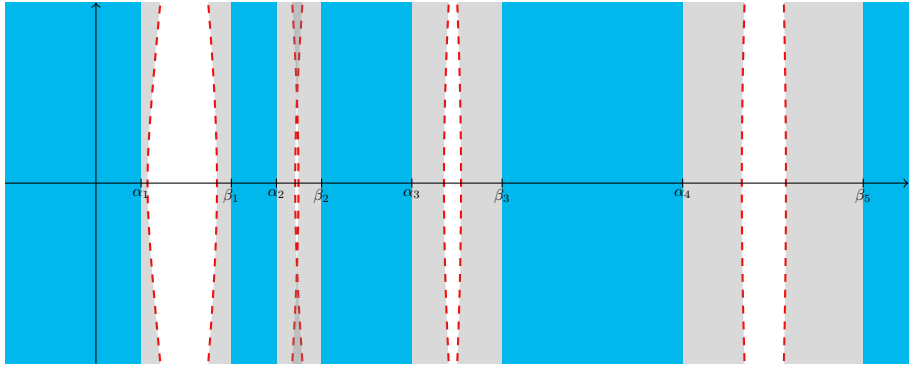
\begin{figure}[H] 
   \centering
   \begin{tikzpicture}[ 
      scale=0.6, transform shape, 
      declare function={ 
      hyperbolax(\x)     = \gx      - sqrt( (\abound^2 + \bbound^2*(\gx)^2      + \bbound^2*\x*\x)/(1-\bbound^2) ); 
      hyperbolaleft(\x)  = \gxupper - sqrt( (\abound^2 + \bbound^2*(\gxupper)^2 + \bbound^2*\x*\x)/(1-\bbound^2) ); 
      hyperbolaright(\x) = \gxlower + sqrt( (\abound^2 + \bbound^2*(\gxlower)^2 + \bbound^2*\x*\x)/(1-\bbound^2) ); 
      rightfun(\x, \gxupper) = \gxupper - sqrt( (\abound^2 + \bbound^2*(\gxupper)^2 + \bbound^2*\x*\x)/(1-\bbound^2) ); 
      leftfun(\x, \gxlower) = \gxlower + sqrt( (\abound^2 + \bbound^2*(\gxlower)^2 + \bbound^2*\x*\x)/(1-\bbound^2) ); 
      parabola(\x,\yy) = \x + \yy;
      }
      ]

      \tikzmath{
      \abound = .1;   
      \bbound = .1;  
      \gxloweri = 1.0;  
      \gxupperi = 3.0;  
      \gxlowerii = 4.0;  
      \gxupperii = 5.0;  
      \gxloweriii = 7.0;  
      \gxupperiii = 9.0;  
      \gxloweriv = 13.0;  
      \gxupperiv = 17.0;  
      \gxlowerv = 28.0;  
      \gxupperv = 28.0;  
      \yy = 1;
      }

      \begin{scope} 
	 \clip (-2,-4) rectangle (18,4);

	 \path [fill=colTspec, opacity=1] (-5, -5) rectangle (\gxloweri, 8);
	 \path [fill=colTspec, opacity=1] (\gxupperi, -5) rectangle (\gxlowerii, 8);
	 \path [fill=colTspec, opacity=1] (\gxupperii, -5) rectangle (\gxloweriii, 8);
	 \path [fill=colTspec, opacity=1] (\gxupperiii, -5) rectangle (\gxloweriv, 8);
	 \path [fill=colTspec, opacity=1] (\gxupperiv, -5) rectangle (\gxlowerv, 8);
	 \path [fill=colTspec, opacity=1] (\gxupperv, -5) rectangle (18, 8);

	 \tikzmath{ \gxlower = \gxloweri; \gxupper = \gxupperi;  }
	 \draw [red,dashed, thick] plot[domain=-5.0:8, smooth] ({leftfun(\x, \gxlower)},{\x});
 	 \draw [red,dashed, thick] plot[domain=-5.0:8, smooth] ({rightfun(\x, \gxupper)},{\x});

	 \tikzmath{ \gxlower = \gxlowerii; \gxupper = \gxupperii; }
	 \draw [red,dashed, thick] plot[domain=-5.0:8, smooth] ({leftfun(\x, \gxlower)},{\x});
 	 \draw [red,dashed, thick] plot[domain=-5.0:8, smooth] ({rightfun(\x, \gxupper)},{\x});

	 \tikzmath{ \gxlower = \gxloweriii; \gxupper = \gxupperiii; }
	 \draw [red,dashed, thick] plot[domain=-5.0:8, smooth] ({leftfun(\x, \gxlower)},{\x});
 	 \draw [red,dashed, thick] plot[domain=-5.0:8, smooth] ({rightfun(\x, \gxupper)},{\x});

	 \tikzmath{ \gxlower = \gxloweriv; \gxupper = \gxupperiv; }
	 \draw [red,dashed, thick] plot[domain=-5.0:8, smooth] ({leftfun(\x, \gxlower)},{\x});
 	 \draw [red,dashed, thick] plot[domain=-5.0:8, smooth] ({rightfun(\x, \gxupper)},{\x});

	 \tikzmath{ \gxlower = \gxlowerv; \gxupper = \gxupperv; }
	 \draw [red,dashed, thick] plot[domain=-5.0:8, smooth] ({leftfun(\x, \gxlower)},{\x});
 	 \draw [red,dashed, thick] plot[domain=-5.0:8, smooth] ({rightfun(\x, \gxupper)},{\x});

	 \tikzmath{ \gxlower = \gxloweri; \gxupper = \gxupperi;  }
	 \fill [gray, opacity=.3] plot[domain=-5.0:8, smooth] ({leftfun(\x, \gxlower)} ,{\x}) -- (\gxloweri,8) -- (\gxloweri, -7) -- cycle; 
	 \fill [gray, opacity=.3] plot[domain=-5.0:8, smooth] ({rightfun(\x, \gxupper)} ,{\x}) -- (\gxupperi,8) -- (\gxupperi, -7) -- cycle; 
	 \tikzmath{ \gxlower = \gxlowerii; \gxupper = \gxupperii; }
	 \fill [gray, opacity=.3] plot[domain=-5.0:8, smooth] ({leftfun(\x, \gxlower)} ,{\x}) -- (\gxlowerii,8) -- (\gxlowerii, -7) -- cycle; 
	 \fill [gray, opacity=.3] plot[domain=-5.0:8, smooth] ({rightfun(\x, \gxupper)} ,{\x}) -- (\gxupperii,8) -- (\gxupperii, -7) -- cycle; 
	 \tikzmath{ \gxlower = \gxloweriii; \gxupper = \gxupperiii; }
	 \fill [gray, opacity=.3] plot[domain=-5.0:8, smooth] ({leftfun(\x, \gxlower)} ,{\x}) -- (\gxloweriii,8) -- (\gxloweriii, -7) -- cycle; 
	 \fill [gray, opacity=.3] plot[domain=-5.0:8, smooth] ({rightfun(\x, \gxupper)} ,{\x}) -- (\gxupperiii,8) -- (\gxupperiii, -7) -- cycle; 

	 \tikzmath{ \gxlower = \gxloweriv; \gxupper = \gxupperiv; }
	 \fill [gray, opacity=.3] plot[domain=-5.0:8, smooth] ({leftfun(\x, \gxlower)} ,{\x}) -- (\gxloweriv,8) -- (\gxloweriv, -7) -- cycle; 
	 \fill [gray, opacity=.3] plot[domain=-5.0:8, smooth] ({rightfun(\x, \gxupper)} ,{\x}) -- (\gxupperiv,8) -- (\gxupperiv, -7) -- cycle; 

	 \draw[->] (-5.5,0)--(18,0);
	 \draw[->] (0,-4)--(0,4);

	 \draw (\gxloweri,-.1) -- node [below] {$\alpha_{1}$} (\gxloweri,.1);
	 \draw (\gxlowerii,-.1) -- node [below] {$\alpha_{2}$} (\gxlowerii,.1);
	 \draw (\gxloweriii,-.1) -- node [below] {$\alpha_{3}$} (\gxloweriii,.1);
	 \draw (\gxloweriv,-.1) -- node [below] {$\alpha_{4}$} (\gxloweriv,.1);
	 \draw (\gxupperi,-.1) -- node [below] {$\beta_{1}$} (\gxupperi,.1);
	 \draw (\gxupperii,-.1) -- node [below] {$\beta_{2}$} (\gxupperii,.1);
	 \draw (\gxupperiii,-.1) -- node [below] {$\beta_{3}$} (\gxupperiii,.1);
	 \draw (\gxupperiv,-.1) -- node [below] {$\beta_{5}$} (\gxupperiv,.1);
      \end{scope}

   \end{tikzpicture}

   \caption{
   Illustration of Theorem~\ref{thm:infgaps} when 
   $\sigma(T)\subseteq \bigcup_{n\in\N} \{ \beta_{n}\le \re(z) \le \alpha_{n+1} \}$.
   The blue areas contain the spectrum of $T$, the shaded areas are enclosures for the spectrum of $T+A$.
   }
   \label{fig:infgaps}
\end{figure}

Observe that in Theorem~\ref{thm:infgapsbdd} we assumed that the imaginary part of 
$\sigma(T)\cap \{\re z \ge \alpha_1\}$ is bounded
whereas in Theorem~\ref{thm:infgaps} we had no such assumption.
As an intermediate situation we could assume that the ``height'' of the boxes that contain the spectrum of $T$ is increasing linearly.
For instance, if 
\begin{equation*}
   \sigma(T) \subseteq \bigcup_{n\in\Z} \{ \alpha_n < \re z < \beta_{n}\} \cap (-\sector{\theta}{0}\cup \sector{\theta}{0}),
\end{equation*}
then we can combine the results from Theorem~\ref{thm:infgapsbdd} with those from Theorem~\ref{thm:bisecgap} to obtain spectral inclusions.



\section{$p$-subordinate perturbations} 
\label{sec:psub}

In this section we consider $p$-subordinate perturbations of a normal operator $T$.
In our first theorem we assume that $T$ has spectrum as in Theorem~\ref{thm:bddstripgap}, but the perturbation $A$ is now assumed to be $p$-subordinate.

\begin{theorem}
   \label{thm:psubord1}
   Let $T$ and $A$ be as in Assumption~\ref{ass:psubord}. That is, $T$ is a normal operator in a Hilbert space $H$, $A$ is $p$-subordinate to $T$ with $0\leq p\leq 1$ and $c \geq 0$ as in \eqref{eq:defpsubord}.
   We assume that there exist $\alpha_T < \beta_T$ such that $(\alpha_{T},\beta_{T})+\I\R \subseteq \rho(T)$ and that
   $\im\sigma(T)\subseteq [-\gamma_T,\gamma_T]$ for some $\gamma_T\geq 0$. 
   If
   \begin{align}
      \label{eq:crucial1subordinate}
      2c\cdot\max\left\{|\alpha_T+\I\gamma_T|^p,|\beta_T+\I\gamma_T|^p\right\}<\beta_{T}-\alpha_{T}
   \end{align}
   then the resolvent set of $T+A$ contains a strip parallel to the imaginary axis; more precisely, 
   \begin{align}
      \label{eq:tira1subordinate}
      \{z\in\C: \alpha'_{T,p} <\re z < \beta'_{T,p} \} \subseteq \rho(T+A)
   \end{align}
   with
   \begin{align}
      \label{eq:extremostira1subordinate}
      \alpha'_{T,p}:=\alpha_{T}+c\cdot\max\left\{|\alpha_T+\I\gamma_T|^p,|\beta_T+\I\gamma_T|^p\right\},\\
      \beta'_{T,p}:= \beta_T -c\cdot\max\left\{|\alpha_T+\I\gamma_T|^p,|\beta_T+\I\gamma_T|^p\right\}.
   \end{align}
\end{theorem}
\begin{proof}
   Let $z\in \C$ with $ \alpha_{T} <\re z < \beta_{T}$. 
   We set $\mu = \re z$ and $\nu=\im z$.
   Then $z\in \rho(T)$
   and
   \begin{equation*}
      \|(T-z)^{-1}\|=\frac{1}{\dist (z,\sigma (T))}\leq \max\left\{\frac{1}{\mu-\alpha_T}, \frac{1}{\beta_T-\mu}\right\}.
   \end{equation*}
   If we set $\tau =\re t$ and $\sigma = \im t$ for $t\in\sigma(T)$, then we obtain
   \begin{align*}
      \|T(T-z)^{-1}\|^2
      &= \sup_{t\in\sigma(T)} \frac{|t|^2}{|t-z|^2}
      =\sup_{\tau + \I\sigma \in\sigma(T)} \frac{\tau^2+\sigma^2}{(\tau-\mu)^2+(\sigma-\nu)^2}
      \leq \sup_{\tau + \I\sigma \in\sigma(T)} \frac{\tau^2+\gamma_T^2}{(\tau-\mu)^2}
      \\
      & \le \sup_{\tau\in\R\setminus\ (\alpha_T,\beta_T)} \frac{\tau^2+\gamma_T^2}{(\tau-\mu)^2}
      = \max \left\{\frac{\alpha_T^2+\gamma_T^2}{(\alpha_T-\mu)^2},\,  \frac{\beta_T^2+\gamma_T^2}{(\beta_T-\mu)^2} \right\}
   \end{align*}
   where the last equality is a straightforward calculation.
   Therefore, if $\alpha_{T,p}' < \re z < \beta_{T,p}'$, then 
   \begin{align*}
      \|A(T-z)^{-1}\|&\leq c\|(T-z)^{-1}\|^{1-p}\|T(T-z)^{-1}\|^{p}\\
      &\leq c\cdot 
      \max\left\{\frac{1}{(\mu-\alpha_T)^{1-p}}, \frac{1}{(\beta_T-\mu)^{1-p}}\right\}
      \max\left\{\frac{|\alpha_T+\I\gamma_T|^p}{(\mu-\alpha_T)^p},\ \frac{|\beta_T+\I\gamma_T|^p}{(\beta_T-\mu)^p} \right\}
      \\
      &\leq c\cdot 
      \max\left\{\frac{1}{\mu-\alpha_T}, \frac{1}{\beta_T-\mu}\right\}
      \max\left\{ |\alpha_T+\I\gamma_T|^p,\ |\beta_T+\I\gamma_T|^p \right\}
      \\
      & <  c\cdot 
      \max\left\{\frac{1}{\alpha_{T,p}'-\alpha_T}, \frac{1}{\beta_T-\beta_{T,p}'}\right\}
      \max\left\{ |\alpha_T+\I\gamma_T|^p,\ |\beta_T+\I\gamma_T|^p \right\}
      = 1
   \end{align*}
   and \eqref{eq:tira1subordinate} follows from Proposition~\ref{prop:ourfirsprop3}.
\end{proof}

The theorem can be generalised to the case of a (bi)sectorial normal operator which will be discussed in detail elsewhere.
\medskip

For $\lambda\in \R$ we define the set
\begin{equation*}
   \parabola(\lambda) := \left\{z\in \C: \re z \geq \lambda, \ | \im z | \leq (\re z - \lambda)^2 \right\}.
\end{equation*}
\begin{figure}[H] 
   \centering
   \begin{tikzpicture}[scale=.8, transform shape]
      \begin{scope}
	 \clip (-1,-3) rectangle (4,3);

	 \fill [colTspec!80] plot [domain=3:1, smooth] ( {\x}, {-(\x-1)^2} ) -- plot [domain=1:3, smooth] ( {\x}, {(\x-1)^2} ) --++(2, 0) --++(0, -8) --++ (-2,0) -- cycle; 
      \end{scope}

      \draw[->] (-1, 0) -- (4,0);
      \draw[->] (0,-3) -- (0,3);

      \draw (1,-.1) -- node [below] {$\lambda$} (1,.1);
      \node at (3.5, 1) {$\parabola(\lambda)$};
   \end{tikzpicture}
   \caption{The set $\parabola (\lambda) = \left\{z\in \C: \re z \geq \lambda, \ | \im z | \leq (\re z - \lambda)^2 \right\}$.}
   \label{fig:exteriorparabola}
\end{figure}

\begin{proposition}
   \label{prop:psubord}
   Let $T$ be a normal operator in a Hilbert space $H$ with
   \begin{equation*}
      \sigma(T)\subseteq \parabola(\beta_{T})\cup-\parabola(-\alpha_{T})
   \end{equation*}
   for some $\alpha_T<\beta_T$.
   Let $A$ be $p$-subordinate to $T$ for some  $0\leq p< 1/2$ and let $c>0$ such that \eqref{eq:defpsubord} holds.
   Let $\zeta_T:= \max\{|\beta_T|,|\alpha_T|\}$ and let $\nu_0\geq 1/2$ such that 
   $$\frac{c\left(\sqrt{\widehat{\nu}-\frac{1}{4}}+\sqrt{\zeta_T^2+\widehat{\nu}^2}\right)^p}{\sqrt{\widehat{\nu}-\frac{1}{4}} }<1~~~ \mbox{ for } \widehat{\nu} >\nu_0.$$
   If there exists $M>0$ such that 
   \begin{align}
      \label{c:crucial3subordinate}
      \frac{c\left(M+\sqrt{\zeta_T^2+\nu_0^2}\right)^p}{M}<1 ~\mbox{ and }~ 2M<\beta_{T}-\alpha_{T},
   \end{align}
   then the resolvent set of $T+A$ contains a strip parallel to the imaginary axis; more precisely, 
   \begin{align}
      \label{t:tira3subordinate}
      \{z\in\C: \alpha_{T+A}^{(p)} <\re z < \beta_{T+A}^{(p)}\} \subseteq \rho(T+A)
   \end{align}
   with
   \begin{align}
      \label{e:extremos de la tira3subordinate}
      \alpha_{T+A}^{(p)}:=\alpha_{T}+M,~~~~ \beta_{T+A}^{(p)}:= \beta_T-M.
   \end{align}
\end{proposition}
\begin{proof}
   Let $z\in \rho(T)$ with $\mu=\re z$ and $\nu=\im z$ and let $d(z) :=\dist(z,\sigma(T)$. 
   Since $T$ is normal, we have that
   \begin{equation*}
      \|(T-z)^{-1}\| = \frac{1}{d(z)},
      \qquad
      \|T(T-z)^{-1}\| = \|I+z(T-z)^{-1}\| \leq 1+\frac{|z|}{d(z)}.
   \end{equation*}
   Together with 
   \begin{equation*}
      \|Ax\|\leq c\|x\|^{1-p}\|Tx\|^{p},\qquad x\in \mD(T),
   \end{equation*}
   we obtain for $\alpha_{T+A}^{(p)} < \re z < \beta_{T+A}^{(p)}$
   \begin{multline*}
      \|A(T-z)^{-1}\| 
      \leq c\|(T-z)^{-1}\|^{1-p}\|T(T-z)^{-1}\|^{p}
      \leq c\frac{1}{d(z)^{1-p}} \left(1+\frac{|z|}{d(z)}\right)^p \\
      \leq 
      \begin{cases}
	 \frac{c\left(1+\frac{|z|}{\sqrt{\min\{\mu-\alpha_T, \beta_T-\mu\}^2+\nu^2}}\right)^p}{\left(\sqrt{\min\{\mu-\alpha_T, \beta_T-\mu\}^2+\nu^2}\right)^{1-p}},
	 &\text{if }  |\nu| \leq 1/2,
	 \\[4ex]
	 \frac{c\left(1+\frac{|z|}{\sqrt{\min\{\mu-\alpha_T, \beta_T-\mu\}^2+|\nu|-\frac{1}{4}} }\right)^p}{\left(\sqrt{\min\{\mu-\alpha_T, \beta_T-\mu\}^2+|\nu|-\frac{1}{4}}\right)^{1-p}},
	 &\text{if } |\nu| > 1/2
      \end{cases}
   \end{multline*}
   where in the last step we used Lemma~\ref{lem:app:parabola}.
   If $|\nu|\leq \nu_0$, then 
   \begin{align*}
      \|A(T-z)^{-1}\|&\leq c\left(1+\frac{|z|}{\min\{\mu-\alpha_T, \beta_T-\mu\}}\right)^p\frac{1}{\left(\min\{\mu-\alpha_T, \beta_T-\mu\}\right)^{1-p}}\\
      &\leq c\left(1+\frac{\sqrt{\zeta_T^2+\nu_0^2}}{M}\right)^p\frac{1}{M^{1-p}}=\frac{c\left(M+\sqrt{\zeta_T^2+\nu_0^2}\right)^p}{M}<1
   \end{align*}
   whereas if $|\nu|>\nu_0\geq 1/2$, then
   \begin{align*}
      \|A(T-z)^{-1}\|&\leq c\left(1+\frac{|z|}{\sqrt{|\nu|-\frac{1}{4}} }\right)^p\frac{1}{\left(\sqrt{|\nu|-\frac{1}{4}}\right)^{1-p}}= \frac{c\left(\sqrt{|\nu|-\frac{1}{4}}+|z|\right)^p}{\sqrt{|\nu|-\frac{1}{4}} }\\
      &\leq \frac{c\left(\sqrt{|\nu|-\frac{1}{4}}+\sqrt{\zeta_T^2+\nu^2}\right)^p}{\sqrt{|\nu|-\frac{1}{4}} }<1
   \end{align*}
   by \eqref{c:crucial3subordinate}.
   We conclude that, in any case, $T+A-z$ is bijective and so $z\in \rho(T+A)$.
\end{proof}

\section{Applications}
\label{sec:applications}

\subsection{Schrödinger operators on a non-selfadjoint quantum star graph} 
\label{sec:nsaquantumgraph}
In this last section we apply some of the above results to a non-selfadjoint Schrödinger operator on a quantum star graph $\Gamma$ with a non-selfadjoint Robin condition at the central vertex and Dirichlet boundary conditions on the set of outer vertices.
Set $\J=\{1,2,\ldots,n\}$.
Our Hilbert space is
$\mathcal{H}=\bigoplus_{j=1}^n L^2(0,a_j; \C)$, with $a_j>0$ for $j=1,\ldots,n$,
and for $c\in \C\cup \{\infty\}$, we define the operator $L_c$ by 
\begin{equation*}
   D(L_c):= \left\{\psi=(\psi_j)_{j=1}^n\in \bigoplus_{j=1}^n H^2(0,a_j;\C):
   \begin{array}{ll}
      \psi_j(a_j)=0,\qquad &j \in \J\\
      \psi_j(0)=\psi_i(0),\qquad &i,j\in\J\\
      \multicolumn{2}{l}{\sum_{j\in\J} \big(\psi_j(0)+ c\psi_j'(0)\big)=0}
   \end{array}
   \right\}
\end{equation*}
and 
\begin{equation}
   \label{eq:Lc}
   L_c\psi:=L_c(\psi_j)_{j=1}^n=(-\psi_j'')_{j=1}^n,\qquad\psi\in D(L_c).
\end{equation}
If $c=\infty$, the conditions at the central vertex are understood as Neumann-Kirchhoff conditions
\begin{equation*}
   \sum_{j\in\J} \psi_j'(0)=0
   \quad\text{and}\quad
   \psi_j(0)=\psi_i(0)\ \text{ for all }\ i,j\in\J.
\end{equation*}

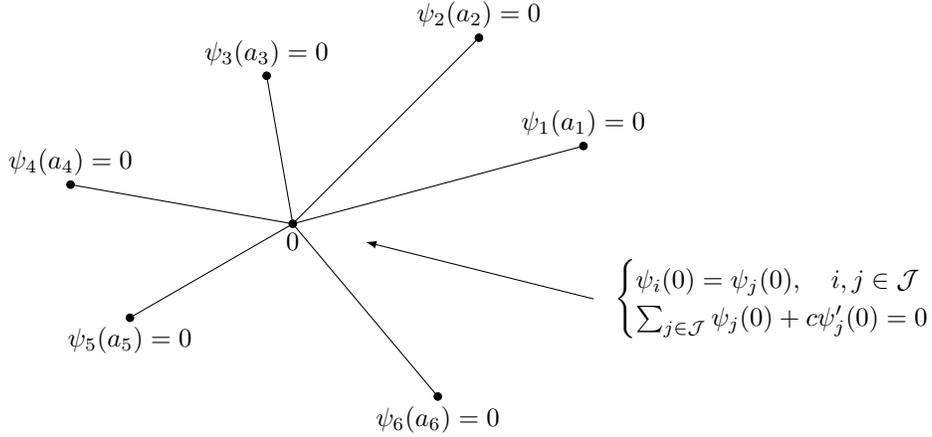
\begin{figure}[H] 
  \centering
   \begin{tikzpicture}
      \coordinate (O) at (0,0);
      \coordinate (A) at (15:4);
      \coordinate (B) at (45:3.5);
      \coordinate (C) at (100:2);
      \coordinate (D) at (170:3);
      \coordinate (E) at (210:2.5);
      \coordinate (F) at (310:3);
      \foreach \x in {A,B,C,D,E,F} \draw (O) -- (\x);
      \foreach \x in {A,B,C,D,E,F,O}  \draw[fill] (\x) circle (0.05cm);
      \node at (A) [above] {$\psi_1(a_1) = 0$};
      \node at (B) [above] {$\psi_2(a_2) = 0$};
      \node at (C) [above] {$\psi_3(a_3) = 0$};
      \node at (D) [above] {$\psi_4(a_4) = 0$};
      \node at (E) [below] {$\psi_5(a_5) = 0$};
      \node at (F) [below] {$\psi_6(a_6) = 0$};
      \node at (O) [below] {$0$};

     \draw[latex-, shorten <=1cm]  (0, 0) -- (4,-1) node[right=1ex]{
      $ \begin{cases}
	 \psi_i(0) = \psi_j(0),\quad i,j\in\J\\
	 \sum_{j\in\J} \psi_j(0) + c\psi'_j(0) = 0
      \end{cases}$};
   \end{tikzpicture}
  \caption{Quantum star graph with Robin condition at the central vertex.}
  \label{fig:figuragrafo}
\end{figure}

In the next proposition we collect some well-known properties of $L_c$.
\begin{proposition} \label{g:generalfact}
   \begin{enumerate}[label={\upshape(\roman*)}]
      \item\label{g:generalfact1}
      $L_c$ is selfadjoint if and only if $c\in \R\cup \{\infty\}$.
      
      \item\label{g:generalfact2}
      $L_c$ is m-sectorial for all $c\in \C $, i.e. $L_c$ is closed, its spectrum is contained in a closed sector symmetric to the real axis, and its numerical range is also contained in this sector. 

      \item\label{g:generalfact3}
      $L_c$ has purely discrete spectrum for all $c\in \C \cup \{\infty\}$.

      \item\label{g:generalfact4}
      $L_c$ has a Riesz basis of generalised eigenfunctions, for all $c\in \C$. 

   \end{enumerate}   
\end{proposition} 

For property \ref{g:generalfact1} we refer to \cite[Theorem 1.4.4]{Berkolaiko2013}. 
The proof of \ref{g:generalfact2} is contained in \cite[Proposition 2.1]{riviere2019}. 
The properties \ref{g:generalfact3} and \ref{g:generalfact4} are consequences of \cite[Theorem 5.1]{Krejcirik2015}.\\

Here we study perturbations of the form 
\begin{equation}
   \label{eq:Sc}
   S_{c}=L_{c}+\mathcal{M}_V
\end{equation}
where $V\in \mathcal{H}$ and $\mathcal{M}_V$ is the operator of multiplication defined by $V$. 
Our goal is to determine conditions on $V$ that ensure that the gaps in the spectrum of $L_c$ result in gaps in the spectrum of $S_c$.
For $c\in\C\cup\{\infty\}$ we denote the eigenvalues of $L_c$ by $\lambda_m(c)$, $m\in \N$, repeated according to their algebraic multiplicities and ordered such that 
\begin{equation}
   \label{eq:evaluesordered}
   \re (\lambda_m(c))\leq \re(\lambda_{m+1}(c)), \qquad m\in \N.
\end{equation}
Rivi{\`e}re and Royer proved the following spectral properties of $L_c$.

\begin{proposition}[\cite{riviere2019}, Proposition 2.5]
   \label{b:boundedimaginarypart}
   For $c\in\C$, there exists $\gamma_c\geq 0$ such that 
   \begin{equation*}
      \re(\lambda_m(c))\geq -\gamma_c
      \quad\text{ and }\quad 
      |\im (\lambda_m(c))|\leq  \gamma_c \quad\text{for all } m\in \N.
   \end{equation*}
   Moreover,
   \begin{equation*}
      \lim_{R\rightarrow\infty}\sup_{\re(\lambda_{m}(c))\geq R} |\im (\lambda_m(c))|
      \leq \frac{2n|\im (c^{-1})|}{|\Gamma|},
   \end{equation*}
   where $|\Gamma|:= \sum_{j=1}^n a_j$ is the total length of the graph.
\end{proposition}
\begin{definition}
   We say that the lengths $a_1,a_2,\ldots,a_n$ are \define{incommensurable over} $\{-1,0,1\}$ if only a trivial linear combination of $a_1,a_2,\ldots,a_n$ with these coefficients vanishes.
\end{definition}
\begin{proposition}[\cite{riviere2019}, Proposition 1.1 and Lemma 2.7]
   \label{m:multiplicities}
   For all $c\in \C$ the geometric and algebraic multiplicity of the eigenvalue $\lambda_{m}(c)$ are equal, for all $m\in \N$. Moreover, if $\{a_j\}_{j=1}^n$ are incommensurable over $\{-1,-0, 1\}$
   then the multiplicity of each eigenvalue of $L_c$ is 1.
\end{proposition}
\begin{corollary}
   \label{s:similarity}
   If $c\in \C\setminus \R$, then $L_c$ is similar to a non-selfadjoint normal operator.     
\end{corollary}
\begin{proof}
   The claim follows from Proposition \ref{g:generalfact}\ref{g:generalfact4}, Proposition \ref{m:multiplicities} and \cite[Chapter~XV, Theorem 6.4]{DunfordSchwartzIII}.    
\end{proof}   

The next lemma shows that $L^q$ integrability of $V$ ensures its $p$-subordination to $L_{c}$ with $p=\frac{1}{q}$.

\begin{lemma}
   Let $c_\Gamma := n \left( \sum_{n\in\mathcal J} \frac{1}{a_j} \right)^{-1}$
   and $c\in \C\setminus \{c_\Gamma\}$. 
   If $V=(V_j)_{j=1}^n\in L^q(\Gamma):=\bigoplus_{j=1}^n L^q(0,a_j; \C)$ for some $q\geq 2$, then there exists a constant $C>0$ which depends only on $c$, $q$ and $\{a_j\}_{j=1}^n$, such that
   \[
   \|V u\|_{L^2(\Gamma)}\leq C\|V\|_{L^q(\Gamma)}\|u\|_{L^2(\Gamma)}^{1-\frac{1}{q}}\|L_{c}u\|_{L^2(\Gamma)}^{\frac{1}{q}},
   \qquad
   u\in D(L_c).
   \]
\end{lemma}
\begin{proof}
   First we show that $0\in \rho(L_c)$. 
   Assume for the sake of a contradiction that $0\in \sigma(L_c)$. 
   Then $0$ is an eigenvalue of $L_c$ and we can pick 
   $\psi := (\psi_j)_{j=1}^n\in \mD(L_c)$ with $L_c\psi =0$.
   Hence for $j=1,\,\dots,\, n$ there exist $A_j, B_j$ such that 
   \begin{equation*}
      \psi_j(x) = A_j + B_j x,
      \qquad x\in [0,a_j].
   \end{equation*}
   By the continuity of $\psi$ at the central vertex, we obtain that $A_j= A_1$ for all $j$.
   The Dirichlet conditions at the outer vertices imply that $B_j = -\frac{A_j}{a_j} = -\frac{A_1}{a_j}$.
   Therefore it follows that 
   $0 = nA_1 + c\sum_{j=1}^n B_j = nA_1 (1 - c\sum_{j=1}^n \frac{1}{a_j}) = nA_1 (1 - \frac{c}{c_\Gamma})$.
   Since $c\neq c_\Gamma$, we must have $A_1=0$ and consequently $\psi=0$.

   For every $u=(u_j)_{j=1}^n\in  D(L_c)\subseteq\bigoplus_{j=1}^n H^2(0,a_j;\C)$
   we have, by \cite[ Chapter~III, Example~2.2]{EngelNagel2000}),
   \begin{align*}
      \|u_j'\|_{L^2(0,a_j)}\leq  3\|u_j\|_{L^2(0,a_j)}+3\|u_j''\|_{L^2(0,a_j)}.  
   \end{align*}
   The Gagliardo-Nirenberg inequality, see for instance \cite{Brezis},
   shows that for each $j=1,2,\ldots, n$ there exists $d_j>0$ such that
   \[
   \|u_j\|_{L^r(0,a_j)}\leq d_j \|u_j\|_{L^2(0,a_j)}^{1-\frac{1}{q}}\|u_j\|_{H^1(0,a_j)}^{\frac{1}{q}}
   \]
   where 
   \[
   \frac{1}{r}+\frac{1}{q}=\frac{1}{2}
   \]
   and from H\"older's inequality we infer that
   \[
   \|V_ju_j\|_{L^2(0,a_j)}\leq \|V_j\|_{L^q(0,a_j)}\|u_j\|_{L^r(0,a_j)}.
   \]
   Therefore, combining these three inequalities, we can find constants $d, C>0$ such that  
   \begin{align*}
      \|Vu\|_{ L^2(\Gamma)}&= \sum_{j=1}^n  \|V_ju_j\|_{L^2(0,a_j)}  \leq\sum_{j=1}^n \|V_j\|_{L^q(0,a_j)}\|u_j\|_{L^r(0,a_j)}  \\
      &\leq \sum_{j=1}^n d_j\|V_j\|_{L^q(0,a_j)} \|u_j\|_{L^2(0,a_j)}^{1-\frac{1}{q}}\|u_j\|_{H^1(0,a_j)}^{\frac{1}{q}}\\
      &\leq \sum_{j=1}^n d_j\|V_j\|_{L^q(0,a_j)} \|u_j\|_{L^2(0,a_j)}^{1-\frac{1}{q}}\left(4\|u_j\|_{L^2(0,a_j)}+3\|u_j''\|_{L^2(0,a_j)}\right)^{\frac{1}{q}} \\
      &\leq n d \max_{1\leq j\leq  n}\|V_j\|_{L^q(0,a_j)} \left(\max_{1\leq j\leq  n}\|u_j\|_{L^2(0,a_j)}\right)^{1-\frac{1}{q}}\left(4\max_{1\leq j\leq  n}\|u_j\|_{L^2(0,a_j)}+3\max_{1\leq j\leq  n}\|u_j''\|_{L^2(0,a_j)}\right)^{\frac{1}{q}} \\
      &\leq n d \|V\|_{L^q(\Gamma)}\|u\|_{L^2(\Gamma)}^{1-\frac{1}{q}}\left(4\|u\|_{L^2(\Gamma)}+3\|L_cu\|_{L^2(\Gamma)}\right)^{\frac{1}{q}}\\
      &\leq n d \|V\|_{L^q(\Gamma)}\|u\|_{L^2(\Gamma)}^{1-\frac{1}{q}}\left(3+4\|L_c^{-1}\|\right)^{\frac{1}{q}}\|L_cu\|_{L^2(\Gamma)}^{\frac{1}{q}} \\
      &\leq C\|V\|_{L^q(\Gamma)}\|u\|_{L^2(\Gamma)}^{1-\frac{1}{q}}\|L_{c}u\|_{L^2(\Gamma)}^{\frac{1}{q}}.
      \qedhere
   \end{align*}
\end{proof}

By Corollary \ref{s:similarity}, for every $c\in\C\cup\{\infty\}$ there exists a normal operator $T_c$ on $\mathcal{H}$ and an isomorphism $J_c\in B(\mathcal{H})$ such that 
$$T_c= J_cL_cJ_c^{-1}.$$
Note that $T_c$ and $L_c$ have the same spectra.
Set $A_{V}:= J_c\M_{V}J_c^{-1}$ and $\kappa_c:= \|J_c\|\|J_c^{-1}\|$.
If $V\in L^q(\Gamma)$, then the above lemma shows that $A_V$ is $\frac{1}{q}$-subordinate to $T_c$ and that
\begin{align}
   \label{p:subordinate} 
   \|A_V w\|_{L^2(\Gamma)}
   \leq \kappa_cC\|V\|_{L^q(\Gamma)}\|w\|_{L^2(\Gamma)}^{1-\frac{1}{q}}\|T_{c}w\|_{L^2(\Gamma)}^{\frac{1}{q}}
\end{align}
for all $w\in D(T_c)$. 
The Weyl law proved by Rivière and Royer in \cite[Proposition 1.2]{riviere2019} shows that 
$$\re(\lambda_m(c))= \frac{\pi^2}{|\Gamma|^2}m^2+O(m) \text{ as } m\rightarrow \infty$$
and 
\begin{equation*}
   \re(\lambda_{m+1}(c))-\re(\lambda_m(c))=O(m+1)\mbox{ as } m\rightarrow \infty.
\end{equation*}
This implies that, for $m\in \N$ such that $\re(\lambda_{m}(c))\neq \re(\lambda_{m+1}(c))$, 
\begin{equation*}
   \frac{ \max\{|\re(\lambda_{m}(c))+\I\gamma_c|^{\frac{1}{q}},|\re(\lambda_{m+1}(c))+\I\gamma_c|^{\frac{1}{q}}\} }{ \re(\lambda_{m+1}(c))-\re(\lambda_{m}(c)) }
   \longrightarrow 0
   \quad\text{for } m\to\infty.
\end{equation*}
Therefore, hypothesis \eqref{eq:crucial1subordinate} is satisfied for $T_c$ and $A_V$ for large enough $m$ and hence 
Theorem~\ref{thm:psubord1} leads to the following theorem.
\begin{theorem}
   Let $L_c$ and $S_c$ be as in \eqref{eq:Lc} and \eqref{eq:Sc}, respectively, for some 
   $c\in \C\setminus \{c_\Gamma\}$ and $V\in L^q(\Gamma)$ with $q>2$.
   Then at most finitely many of the vertical spectral free strips of $L_c$ close under perturbation by $V$.
\end{theorem}

\begin{remark}
   By Proposition \ref{prop:bddstrip}  and \ref{cor:ourconjecture1}, we have that, for all $\theta\in (0,\pi/2]$, there is $\beta\in \R$ such that $\sigma(L_c+\M_V)\subseteq S_{\theta}(\beta)$. 
\end{remark}
\begin{remark}
   We can also admit more general perturbations.
   Consider the operator $\mathcal{A}$ defined by
   $$D(\mathcal{A}):=D(L_c),$$
   and
   $$\mathcal{A}\psi(x):= \zeta(x)\frac{d}{dx}\psi(x)+V(x)\psi(x),~~~~\psi(x)\in D(L_c)$$
   with $\zeta, V\in L^{2}(\Gamma)$. 
   Then the following is true.
   \begin{itemize}
      \item  If $\zeta\in L^{\infty}(\Gamma)$ then $J_{c}\mathcal{A}J^{-1}_c$ is $\frac{1}{2}$-subordinate to $T_{c}$.
      \item If $\zeta\in L^{2}(\Gamma)$ then $J_{c}\mathcal{A}J^{-1}_c$is $T_{c}$-bounded.
   \end{itemize}
\end{remark}

\subsection{ 
Perturbations of normal operators generated by first order
systems}
We consider the following first order system of ODEs
\begin{equation*}
   \mathcal{L}y:= \mathcal{L}(B,Q)y:= -\I B y'+Q(x)y=\lambda y,\qquad y
   = 
   \begin{pmatrix}
      y_1\\
      y_2\\
      \vdots\\
      y_n  
   \end{pmatrix}
\end{equation*}
with 
\begin{equation*}
   B:=\diag(b_1,b_2,\ldots, b_n)\in \C^{n\times n} \quad\text{and}\quad Q:= \left(q_{jk}\right)_{j,k=1}^n\in L^1([0,1]; \C^{n\times n}).
\end{equation*}
We associate to $\mathcal{L}$ the maximal operator $L_{\max}:= L_{\max}(B,Q)$ acting in $ L^2([0,1]; \C^{n})$ on the domain 
$$D(L_{\max}):=\left\{y\in AC([0,1]; \C^{n}): \mathcal{L}y \in L^2([0,1]; \C^{n})  \right\}$$
and with action 
$$L_{\max}y:= \mathcal{L}y,\qquad y\in D(L_{\max}). $$
The minimal operator $L_{\min}:= L_{\min}(B,Q)$ is defined by 
$$D(L_{\min}):=\left\{y\in D(L_{\max}): y(0)=y(1)=0 \right\}$$
and
$$L_{\min}y:= \mathcal{L}y,~~~y\in D(L_{\min}). $$

Let us first consider the case when $Q=0$.
The next proposition describes the situation when the operators on the different vertices are decoupled and satisfy so-called \define{quasi-periodic boundary conditions} on each edge.
We omit the proof.

\begin{proposition}
   Let $D:=\diag(d_1,d_2,\ldots, d_n)\in \C^{n\times n} $ and consider the operator $L_{D}(B,0)$ defined by 
   \begin{align}
      \label{q:quasiperiodic}
      D(L_{D}(B,0)):=\left\{y\in D(L_{\max}(B,0)): Dy(1)=y(0) \right\}  
   \end{align}
   and
   $$L_{D}(B,0)y:= \mathcal{L}(B,0)y,~~~y\in D(L_{D}(B,0)). $$
   Then the following is true:
   \begin{enumerate}[label={\upshape(\roman*)}]
      \item If $|d_i|=1$ for all $i=1,\ldots,n$, then $L_D(B,0)$ is normal.
      \item If $|d_i|=1$ and $b_i\in \C\setminus \{0\}$ for all $i=1,\ldots,n$, then $L_D(B,0)$ is normal and discrete.
      \item If $|d_i|=1$ and $b_i\in \R\setminus \{0\}$ for all $i=1,\ldots,n$, then $L_D(B,0)$ is selfadjoint and discrete.
   \end{enumerate}
\end{proposition}

Now we consider the following more general boundary conditions 
\begin{align}
   \label{b:boundaryconditions} 
   Cy(0)+Dy(1)=0,\qquad C:= \left(c_{jk}\right),\ D:= \left(d_{jk}\right)\in  \C^{n\times n},
\end{align}
where we always assume that $\rank(C~D) = n$.
We define the operator $L_{C,D}(B,Q)$ as restriction of $L_{\max}$ to the domain
\begin{align*}
   D(L_{C,D}(B,Q)):=\left\{y\in D(L_{\max}(B,Q)): Cy(0)+Dy(1)=0 \right\} .
\end{align*}
This class of operators has appeared in numerous papers (see e.g. \cite{BirkhoffLanger1923}, \cite{MO2012}, \cite{LM2015} and \cite{ALMO2019}). 
For instance, Lunyov and Malamud found in \cite[Lemma 5.1]{LM2015} all normal realizations $L_{C,D}(B,0)$ of the minimal operator $L_{\min}(B,0)$ under so-called regular boundary conditions (see \cite[pag. 89]{BirkhoffLanger1923} or \cite[Definition 5.1]{ALMO2019}). 
Later, Agibalova, Lunyov, Malamud and Oridoroga in \cite[Lemma 4.4]{ALMO2019} characterized these normal extensions for the case $n=2$.
\begin{proposition}[\cite{LM2015}, Lemma 5.1]
   Let $B$ be a nonsingular diagonal $n \times n$ matrix with complex entries and let $C, D\in  \C^{n\times n}$ such that the boundary conditions given in \eqref{b:boundaryconditions} are regular. 
   Then the operator $L_{C,D}(B,0)$ is normal if and only if 
   \begin{equation*}
      CBC^*= DBD^*.
   \end{equation*}
\end{proposition}

\begin{proposition}[\cite{ALMO2019}, Lemma 4.4]
   Let $n=2$,  $C, D\in  \C^{2\times 2}$ such that $\rank(C~D) = 2$ and $b_1b_2^{-1}\notin \R$ . Then the operator $L_{C,D}(B,0)$ is normal if and only if 
   \begin{equation*}
      D=C\diag(d_1,d_2),~~|d_1|=|d_2|=1\ \text{ and }\ \det C \neq 0.
   \end{equation*}
\end{proposition}
\begin{remark}
   The operator $L_{C,D}(B,0)$ is normal if and only if there exist $d_1,d_2\in \C$ with $|d_1|=|d_2|=1$ such that $L_{C,D}(B,0)=L_{D'}(B,0)$ with $D'=\diag(d_1,d_2)$.   
\end{remark}

Now we focus on the case $n=2$. 
Our goal is to obtain spectral inclusions for the operator $L_D(B,Q)$ with $B=\diag(b_1,b_2)$ and $D=\diag(d_1,d_2)$. The next lemma shows that $L^q$ integrability of $Q$ with $q\geq 2$ ensures that its corresponding operator of multiplication is $L_D(B,0)$-bounded with $L_D(B,0)$-bound equal to zero.
\begin{lemma}
   Let $Q\in L^q([0,1]; \C^{2\times 2})$ for some $q \ge 2$.
   As before, let $B=\diag(b_1,b_2)$ and $D=\diag(d_1,d_2)$ with $b_1,b_2\neq 0$, $d_1,d_2\in \C\setminus\{0,1\}$.
   Then $\mathcal{M}_{Q}$, the operator of multiplication defined by $Q$, is $q$-subordinate to $L_D(B,0)$, in particular it has $L_D(B,0)$-bound equal to zero.
\end{lemma}
\begin{proof}
   Let $ y=\begin{pmatrix}
      y_1 \\ y_2
   \end{pmatrix} \in D(L_{D}(B,0))$
   Then $y_1,y_2\in H^1(0,1)$ and 
   \begin{equation*}
      d_jy_j(1)= y_j(0)\quad \text{ for } j=1,2,
   \end{equation*}
   and $L_D(B,0)$ can be written as the direct sum $L_D(B,0) = b_1 L_{d_1}\oplus b_2 L_{d_2}$.
   Note that the condition $d_1, d_2 \notin \{0,1\}$ guarantees that $L_{d_1}$ and $L_{d_2}$ are boundedly invertible.
   The Gagliardo-Nirenberg inequality shows that there exists a constant $c_q>0$ such that for $j=1,2$ 
   \[
   \|y_j\|_{L^r(0,1)}\leq c_q \|y_j\|_{L^2(0,1)}^{1-\frac{1}{q}}\|y_j\|_{H^1(0,1)}^{\frac{1}{q}}
   \]
   where 
   \[
   \frac{1}{r}+\frac{1}{q}=\frac{1}{2}
   \]
   and from H\"older's inequality we infer that, for any $\tilde q\in L^q(0,1)$, 
   \[
   \|\tilde q y_j\|_{L^2(0,1)}\leq \|\tilde q\|_{L^q(0,1)}\|y_j\|_{L^r(0,1)}.
   \]
   Therefore, combining these two inequalities, we can find constants $C_j>0$ such that  
   \begin{align*}
      \|\tilde q y_j\|_{L^2(0,1)}
      & \leq \|\tilde q\|_{L^q(0,1)}\|y_j\|_{L^r(0,1)} 
      \leq c_q \|\tilde q\|_{L^q(0,1)}\|y_j\|_{L^2(0,1)}^{1-\frac{1}{q}}\|y_j\|_{H^1(0,1)}^{\frac{1}{q}}  \\
      &= c_q \|\tilde q\|_{L^q(0,1)}\|y_j\|_{L^2(0,1)}^{1-\frac{1}{q}}
      \left(\|y_j\|_{L^2(0,1)}+\|y_j'\|_{L^2(0,1)}\right)^{\frac{1}{q}} 
      \\
      &\leq c_q \|\tilde q\|_{L^q(0,1)}\|y_j\|_{L^2(0,1)}^{1-\frac{1}{q}}
      \left(1+\|(L_{d_j})^{-1}\|\right)^{\frac{1}{q}}\|L_{d_j}y_j\|_{L^2(0,1)}^{\frac{1}{q}} 
      \\
      &= C_j\|\tilde q\|_{L^q(0,1)}\|y_j\|_{L^2(0,1)}^{1-\frac{1}{q}}
      \|L_{d_j}y_j\|_{L^2(0,1)}^{\frac{1}{q}}.
   \end{align*} 
   Consequently, there is a constant $C_0>0$ such that 
   \begin{align*}
      \|Qy\|_{L^2(0,1)}^2
      &= \| q_{11} y_1 + q_{12} y_2 \|_{L^2(0,1)}^2 + \| q_{21} y_1 + q_{22} y_2 \|_{L^2(0,1)}^2
      \\
      & \le 2\left\{ \| q_{11} y_1 \|_{L^2(0,1)}^2 + \|q_{12} y_2 \|_{L^2(0,1)}^2 + \| q_{21} y_1 \|_{L^2(0,1)}^2 + \|q_{22} y_2 \|_{L^2(0,1)}^2
      \right\}
      \\
      &\le 
      C_0\, \|Q\|_{L^q(0,1)}^2
      \sum_{j=1}^2
      \|y_j\|_{L^2(0,1)}^{2(1-\frac{1}{q})} \|L_{d_j}(1,0)y_j\|_{L^2(0,1)}^{\frac{2}{q}}
      \\
      &\le 
      C_0\, \|Q\|_{L^q(0,1)}^2
      \bigg( \sum_{j=1}^2 \|y_j\|_{L^2(0,1)}^2 \bigg)^{1-\frac{1}{q}} 
      \bigg( \sum_{j=1}^2 \|L_{d_j}(1,0)y_j\|_{L^2(0,1)}^2 \bigg)^{\frac{1}{q}}
      \\
      & =
      C_0\, \left( 
      \|Q\|_{L^q(0,1)}
      \|y\|_{L^2(0,1)}^{1-\frac{1}{q}} 
      \|L_{D}(1,0)y\|_{L^2(0,1)}^{\frac{1}{q}}
      \right)^2.
      \qedhere
   \end{align*}

\end{proof}
\begin{remark}
   Let $\epsilon>0$ and $y\in D( L_{D}(B,0))$, then 
   $$\|Q(x)y\|_{L^2([0,1]; \C^2)}\leq a'_{\epsilon}\|y\|_{L^2([0,1]; \C^2)}+ b'_{\epsilon}\|L_{D}(B,0)y\|_{L^2([0,1]; \C^2)},$$
   where 
   \begin{align*}
      a'_{\epsilon} := 2\left(1-\frac{1}{q}\right) \epsilon^{\frac{q}{1-q}}C_0
      \max_{i,j=1}^2 \|q_{ij}\|_{L^q(0,1)} \quad\text{ and }\quad
      b'_{\epsilon} := \frac{2}{q}\, {\epsilon^{q}}C_0
      \max_{i,j=1}^2 \|q_{ij}\|_{L^q(0,1)}.
   \end{align*}
\end{remark}
\begin{proposition}
   Let $Q\in L^q([0,1]; \C^{2\times 2})$, $b_1,b_2\neq 0$, $d_1,d_2\in \C\setminus\{1\}$ with $|d_1|=|d_2|=1$ and $q\geq 2$. 
   If $B:=\diag(b_1,b_2)$ and $D:=\diag(d_1,d_2)$ then the operator $L_{D}(B,Q)$ has compact resolvent. 
   Moreover, if $ \theta_{\mathfrak s}:= \frac{1}{2}\left|\arg\left(\frac{\mathfrak s b_2}{b_1}\right)\right| $, with $\mathfrak s = \pm 1$, then for all $\epsilon>0$, there exists $r_\epsilon >0$ such that 
   \begin{equation}
      \label{eq:inlcusionsystem}
      \sigma(\e^{\I\varphi_{\mathfrak s}} L_D(B,Q))
      \subseteq {K_{r_\epsilon}(0)}\cup S_{\theta_{\mathfrak s}+\epsilon}(0)\cup -S_{\theta_{\mathfrak s}+\epsilon}(0),
   \end{equation}
   for some $\varphi_{\mathfrak s} \in[0,2\pi)$ and $K_{r_\epsilon}(0)$ is the closed disk of radius $r_\epsilon$ centred in $0$.  
\end{proposition}
\begin{proof}
   Let $\varphi_{\mathfrak s}\in [0,2\pi) $  such that 
   \begin{equation*}
      \sigma(\e^{\I\varphi_{\mathfrak s}} L_D(B,0))
      \subseteq  S_{\theta_{\mathfrak s}}(0)\cup -S_{\theta_{\mathfrak s}}(0).
   \end{equation*}
   The spectral inclusion \eqref{eq:inlcusionsystem}
   follows from Proposition \ref{prop:bisectorstrip}. 
   Moreover, if $T:= \e^{\I\varphi_{\mathfrak s}} L_D(B,0)$ we have that 
   $\sup_{t\in\sigma(T)} H_z(t) < 1$, for $z\in {K_{r_\epsilon}(0)}\cup S_{\theta_{\mathfrak s}+\epsilon}(0)\cup -S_{\theta_{\mathfrak s}+\epsilon}(0) $, where $H_z$ is as in \eqref{eq:Hdef}. 
   So, by  Proposition~\ref{prop:ourfirsprop3}\ref{item:ourfirstprop3:iii} we have that 
   \begin{equation*}
      ( L_D(B,Q)-w)^{-1}=( L_D(B,0)-w)^{-1} \sum_{n=0}^\infty (-\mathcal{M}_{Q} (L_D(B,Q)-w)^{-1})^n,
   \end{equation*}
   for $w\notin\e^{-\I\varphi_{\mathfrak s}}\left( {K_{r_\epsilon}(0)}\cup S_{\theta_{\mathfrak s}+\epsilon}(0)\cup -S_{\theta_{\mathfrak s}+\epsilon}(0)\right)$. 
   Since $L_{D}(B,0)$ is a discrete normal operator, we conclude that $L_{D}(B,Q)$ has compact resolvent.
\end{proof}
\begin{remark}
   Resolvent estimates for $L_{D}(B,Q)$ will be given in a forthcoming work. For instance, we will prove that, even if $|d_1|,|d_2|\neq 1$, this operator is bisectorial for some angle $\phi\in [0, \pi/2)$ and radius $r\geq 0$, see \cite[Definition 2.7]{TretterWyss2014} .   
\end{remark}

\appendix
\section{Estimates in the complex plane} 
\label{app:estimates}
In this section we prove some technical estimates for functions in the complex plane that are used in Sections~\ref{sec:boundedimaginary} and \ref{sec:sectorial}.

\begin{remark}
   Let $w\in\C\setminus\R$ and $\gamma\in\R\setminus\{\im w\}$.
   Then
   \begin{align}
      \label{eq:app:distrealaxis}
      \sup_{\tau\in \R}\frac{|\tau|}{|\tau-w|} &= \frac{|w|}{|\im w|}
      \intertext{and}
      \label{eq:app:distrealaxisshifted}
      \sup_{t\in \R + \I\gamma }\frac{|t|^2}{|t-w|^2} 
      & \le 1 + \frac{(\re w)^2}{ (\gamma - \im w)^2 } + \frac{\gamma^2}{ ( \gamma-  \im  w )^2} .
   \end{align}
\end{remark}
\begin{proof}
   In fact, if $\re w \neq 0$, then the supremum in \eqref{eq:app:distrealaxis} is a maximum and it is attained at $\tau = \frac{|z|^2}{\mu}$.
   If $\re w  = 0$, then the function on the left hand side has no maximum and
   $\sup_{\tau\in \R}\frac{|\tau|}{|\tau-w|}= \im_{|\tau|\to\infty}\frac{|\tau|}{|\tau-w|} = 1 = \frac{|w|}{|\im w|}$.

   For the proof of \eqref{eq:app:distrealaxisshifted} observe that 
   \begin{align*}
      \sup_{t\in \R + \I\gamma }\frac{|t|^2}{|t-w|^2}
       & = \sup_{\tau\in \R}\frac{|\tau + \I\gamma|^2}{|\tau + \I\gamma-w|^2} 
       = \sup_{\tau\in \R} \frac{\tau^2}{|\tau + \I\gamma-w|^2} + \frac{\gamma^2}{|\tau + \I\gamma-w|^2} 
       \\
       & \le \frac{|\I\gamma - w|^2}{ (\im(\I\gamma-w))^2 } + \frac{\gamma^2}{ ( \im( \I\gamma-w ) )^2} 
   \end{align*}
   where we used \eqref{eq:app:distrealaxis} to estimate the first term.
   The second term clearly attains its maximum when $\tau = \re w$.
   The claim follows if we observe that 
   $\frac{|\I\gamma - w|^2}{ (\im(\I\gamma-w))^2 } = 1 + \frac{(\re w)^2}{ (\im(\I\gamma-w))^2 }$.
\end{proof}

\begin{lemma}
   \label{lem:nolocalextrema}
   For $a,b\ge0$ and $z\in\C$ fixed, the function
   \begin{equation*}
      H_z:\C\setminus\{z\}\to\C,
      \qquad
      H_z(t) = \frac{a^2 + b^2 |t|^2}{|t-z|^2} 
   \end{equation*}
   does not have any local maximum and
   \begin{equation*}
      \lim_{|t|\to\infty }H_z(t) = b^2.
   \end{equation*}
   Consequently, for every set $U\subseteq\C$ with unbounded boundary $\partial U$, we have that
   \begin{equation*}
      \sup_{t\in U} H_z(t) = \sup_{t\in \partial U} H_z(t).
   \end{equation*}
   If $U$ is unbounded but $\partial U$ is bounded, then
   \begin{equation*}
      \sup_{t\in U} H_z(t) = \max\{ b^2, \sup_{t\in \partial U} H_z(t) \}.
   \end{equation*}
\end{lemma}
\begin{proof}
   The cases when $z=0$ or $b=0$ are clear. So let us now assume that $b>0$ and $z\neq 0$.
   Let us restrict $H_z$ to $t\in\C$ with $t = rz$ for $r\in\R\setminus\{1\}$.
   Then 
   \begin{equation*}
      f(r) := H_z(r z) = \frac{ a^2 + b^2 r^2|z|^2 }{ |z|^2(r-1)^2} 
      = b^2 + \frac{ a^2 + |z|^2b^2}{|z|^2(r-1)^2} + \frac{2b^2}{r-1} .
   \end{equation*}
   It is easy to see that $f$ has only one local extremum; it is a global minimum 
   and it is attained at $r_0 = -\frac{a^2}{b^2 |z|^2}$ and $f(r_0) = \frac{b^2 a^2}{b^2|z|^2 + a^2} < b^2$.
   Therefore $H_z$ cannot have a local maximum in any $t\in\C\setminus\{z\}$ with $\arg t = \arg(\pm z)$.
   If $\arg t\neq \arg \pm z$, then we write $t= z + R\e^{\I\phi}$ with $R=|t-z|$ and appropriate $\phi\in\R$.
   Then
   $H_z(z + R\e^{\I(\phi+h)}) = \frac{a^2 + b^2|z + R\e^{\I(\phi+h)}|^2}{R^2}$ is strictly monotonic in $h$ for $h\in\R$ in a small enough neighbourhood of zero, see Figure~\ref{fig:app:nolocalextrema}.
   Therefore $H_z$ cannot have a local extremum in $t$.
   \begin{figure}
      \begin{center}
	 \begin{tikzpicture}[scale=0.6] 
	    \draw[->] (-3,0) -- (3,0);
	    \draw[->] (0,-3) -- (0,3);

	    \coordinate (T) at (50:2);
	    \coordinate (Z) at (3,1);

	    \tikzmath{coordinate \RRR;
	    \RRR = (Z)-(T); 
	    \distZT = sqrt((\RRRx)^2+(\RRRy)^2); 
	    }

	    \node[above right]  at (Z) {$z$};
	    \draw [fill] (Z) circle [radius=.05];

	    \node[above=.3]  at (T) {$t$};
	    \draw [fill] (T) circle [radius=.05];

	    \draw [dashed] (0,0) circle [radius=2];


	    \draw[red]
	    ($(Z) + (-10:-\distZT pt)$) arc (-10:350:-\distZT pt);

	 \end{tikzpicture}
      \end{center}
      \caption{Illustration for the proof of Lemma~\ref{lem:nolocalextrema}. 
      If $h\in\R$, then $t' = z+ R\e^{\I(\phi+h)}$ varies on the red circle leaving $|z-t'|$ constant while $|t'|$ is strictly monotonic for $h$ in a neighbourhood of $0$.}
      \label{fig:app:nolocalextrema}
   \end{figure}
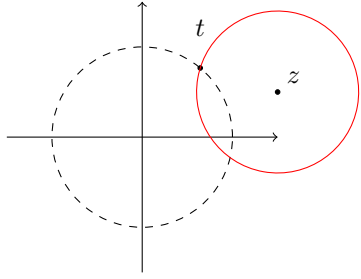
\end{proof}

Our next lemma gives bounds for the supremum of $H_z$ when restricted to a strip around the real axis or a (bi)sector centred in $0$. 
It is used in Proposition~\ref{prop:bddstrip}, Proposition~\ref{prop:bisectorstrip} and Theorem~\ref{thm:bisecgap} respectively.
Recall that for $\beta\in\R$ and $\theta\in [0, \pi/2)$ we define the closed sector 
$\sector{\theta}{\beta} = \{z\in\C : |\arg(z-\beta)| \le \theta \} \cup \{\beta\}$.

\begin{lemma}
   \label{lem:app:Hest} 
   Let $a,b\ge 0$.
   Let $H_z(t) = \frac{a^2 + b^2 |t|^2}{ |t-z|^2 } $ be the function from Lemma~\ref{lem:nolocalextrema}.
   \begin{enumerate}[label={\upshape(\roman*)}]
      \item
      \label{item:lem:app:distanceline}
      Let $\gamma\in\R$ and let $L = \{ \tau + \I \gamma : \tau\in\R\}$.
      Then for $z \in\C\setminus L$ we have that 
      \begin{align}
	 \label{eq:distanceline}
	 \sup_{t\in L} H_z(t)
	 &\le 
	 b^2 + \frac{b^2 (\re z)^2}{(\im z - \gamma)^2} + \frac{a^2 + b^2\gamma^2}{(\im z - \gamma)^2}
	 \\
	 \nonumber
	 &= \frac{b^2|z-\I\gamma|^2 + a^2 + b^2\gamma^2}{ (\im(z - \I\gamma))^2}.
      \end{align}

      \item
      \label{item:lem:app:distancestrip}
      Let $\gamma_1 < \gamma_2$ and let 
      $\strip_{\gamma_1,\gamma_2} 
      = \{ \tau + \I \sigma : \tau\in\R,\, \gamma_1 \le \sigma\le \gamma_2\}$ 
      be the strip parallel to the real axis.
      Then for $z\in\C\setminus \strip_{\gamma_1,\gamma_2}$ we have that 
      \begin{equation}
	 \label{eq:distancestrip}
	 \sup_{t\in \strip_{\gamma_1,\gamma_2}} H_z(t)
	 \le 
	 b^2 + 
	 \max\left\{ 
	 \frac{b^2 (\re z)^2}{(\im z - \gamma_1)^2} + \frac{a^2 + b^2\gamma_1^2}{(\im z - \gamma_1)^2},\ 
	 \frac{b^2 (\re z)^2}{(\im z - \gamma_2)^2} + \frac{a^2 + b^2\gamma_2^2}{(\im z - \gamma_2)^2}
	 \right\}.
      \end{equation}

      \item
      \label{item:lem:app:Hsectorb}
      Let $\theta\in [0, \pi/2)$, $\beta\in\R$ and let $z\in \C \setminus \sector{\theta}{\beta}$.
      Then either $\arg(z-\beta)\in (\theta, \theta+\pi)$ or $\arg(z-\beta)\in (-\theta-\pi, -\theta)$ and
      \begin{multline}
	 \label{eq:Hsectorb}
	 \sup_{t\in \sector{\theta}{\beta}}  H_z(t)
	 \\
	 \le b^2 + 
	 \begin{cases}
	    b^2\, \frac{ (\re( \e^{-\I\theta}(z-\beta)) )^2 }{ ( \im (\e^{-\I\theta} (z-\beta)) )^2 }
	    + \frac{a^2 + b^2\beta^2}{ ( \im (\e^{-\I\theta} (z-\beta)) )^2 }
	    &\qquad \text{if } \arg(z-\beta)\in(\theta, \theta+\pi),
	    \\[1ex]
	    b^2\, \frac{ (\re( \e^{\I\theta}(z-\beta)) )^2 }{ ( \im (\e^{\I\theta} (z-\beta)) )^2 }
	    + \frac{a^2 + b^2\beta^2}{ ( \im (\e^{\I\theta} (z-\beta)) )^2 }
	    &\qquad \text{if } \arg(z-\beta)\in(-\theta-\pi, -\theta).
	 \end{cases}
      \end{multline}

      \item
      \label{item:lem:app:Hsectora}
      Let $\theta\in [0, \pi/2)$, $\alpha\in\R$ and let $z\in \C \setminus -\sector{\theta}{-\alpha}$.
      Then either $\arg(z-\alpha)\in (-\theta, -\theta+\pi)$ or $\arg(z-\alpha)\in (\theta-\pi, \theta)$ and
      \begin{multline}
	 \label{eq:Hsectora}
	 \sup_{t\in -\sector{\theta}{-\alpha}} H_z(t)
	 \le b^2 +
	 \\
	 \begin{cases}
	    b^2\, \frac{ (\re( \e^{\I\theta}(z-\alpha)) )^2 }{ ( \im (\e^{\I\theta} (z-\alpha)) )^2 }
	    + \frac{a^2 + b^2\alpha^2}{ ( \im (\e^{\I\theta} (z-\alpha)) )^2 }
	    &\qquad \text{if } \arg(z-\alpha)\in(-\theta, -\theta+\pi),
	    \\[2ex]
	    b^2\, \frac{ (\re( \e^{-\I\theta}(z-\alpha)) )^2 }{ ( \im (\e^{-\I\theta} (z-\alpha)) )^2 }
	    + \frac{a^2 + b^2\alpha^2}{ ( \im (\e^{-\I\theta} (z-\alpha)) )^2 }
	    &\qquad \text{if } \arg(z-\alpha)\in(\theta, \theta-\pi).
	 \end{cases}
      \end{multline}

   \end{enumerate}
\end{lemma}

\begin{proof} 
   We know from Lemma~\ref{lem:nolocalextrema} that the supremum of $H_z$ on the given sets is equal to its supremum on their respective boundaries.
   Therefore, in 
   \eqref{eq:distanceline} -- \eqref{eq:Hsectora}
   we can replace the supremum by the supremum over the boundary of the respective sets.
   Let $\mu = \re z$ and $\nu = \im z$.

   \ref{item:lem:app:distanceline}
   For every $t=\tau + \I\gamma\in L$ we obtain with the help of \eqref{eq:app:distrealaxis} for the first summand
   \begin{align*}
      H_z(t) 
      &= \frac{a^2 + b^2|t|^2}{|t-z|^2} 
      = \frac{b^2\tau^2}{|\tau-(z - \I\gamma)|^2} + \frac{a^2 + b^2\gamma^2}{|\tau-(z - \I\gamma)|^2}
      \\ &\le \frac{b^2 |z - \I\gamma|^2}{|\im (z - \I\gamma)|^2} + \frac{ a ^2 + b^2 \gamma^2}{|\im(z - \I\gamma)|^2}
      = b^2 + \frac{b^2 (\re z)^2}{(\im z - \gamma)^2} + \frac{a^2 + b^2\gamma^2}{(\im z - \gamma)^2}.
   \end{align*}

   \ref{item:lem:app:distancestrip}
   It follows from \eqref{eq:distanceline} that
   \begin{align*}
      \sup_{t\in \strip_{\gamma_1, \gamma_2}} H_z(t)
      = \max_{\gamma \in\{\gamma_1, \gamma_2\}}\
      \sup_{t\in \R + \I\gamma} H_z(t)
      & \le 
      b^2 + \max_{\gamma \in\{\gamma_1, \gamma_2\}}\
      \frac{b^2 (\re z)^2}{(\im z - \gamma)^2} + \frac{a^2 + b^2\gamma^2}{(\im z - \gamma)^2}.
   \end{align*}

   \ref{item:lem:app:Hsectorb}
   Suppose that $\arg(z-\beta)\in(\theta,\theta+\pi)$.
   Then $H_z$ is well defined in the half-plane $P = \beta + \e^{\I\theta} \{ \zeta \in\C : \im \zeta \le 0\}$  and $\sector{\theta}{\beta}\subseteq P$.
   See Figure~\ref{fig:rotatedhyperbolas}.
   Therefore
   \begin{align}
      \nonumber
      \sup_{t\in \sector{\theta}{\beta} } H_z(t)
      & \le \sup_{t\in P } H_z(t)
      = \sup_{t\in \partial P } H_z(t)
      = \sup_{ r\in\R } H_z(\beta + r\e^{\I\theta})
      \\
      \nonumber
      & = \sup_{ r\in\R } \frac{ a^2 + b^2|\beta +  r\e^{\I\theta}|^2 }{ | \beta + r\e^{\I\theta} - z |^2 }
      = \sup_{ r\in\R } \frac{ a^2 + b^2 | r + \e^{-\I\theta}\beta |^2 }{ | r + \e^{-\I\theta} (\beta - z) |^2 }
      \\
      \nonumber
      & \le \sup_{ r\in\R } 
      \frac{ b^2 |r|^2 }{ | r + \e^{-\I\theta} (\beta - z) |^2 }
      + \frac{ a^2 + b^2 \beta^2 }{ | r + \e^{-\I\theta} (\beta - z) |^2 }
      \\
      \label{eq:halfplaneestimate}
      & \le b^2 + b^2\, \frac{ (\re( \e^{-\I\theta}(\beta-z)) )^2 }{ ( \im (\e^{-\I\theta} (\beta - z)) )^2 }
      + \frac{a^2 + b^2\beta^2}{ ( \im (\e^{-\I\theta} (\beta - z)) )^2 }
   \end{align}
   where we used \eqref{eq:app:distrealaxis} in the last inequality.
   This shows the first line in \eqref{eq:Hsectorb}.
   If $\arg(z-\beta)\in(-\theta-\pi, -\theta)$, then $H_z$ is well defined in the half-plane 
   $Q = \beta + \e^{-\I\theta} \{ \zeta\in\C : \im \zeta \ge 0\}$ and $\sector{\theta}{\beta}\subseteq Q$
   and an analogous calculation shows the second line in \eqref{eq:Hsectorb}.

   \smallskip

   \ref{item:lem:app:Hsectora}
   Note that 
   \begin{align*}
      \sup_{t\in-\sector{\theta}{-\alpha}} \frac{|t|^2}{|t-z|^2} 
      = \sup_{t\in-\sector{\theta}{-\alpha}} \frac{|-t|^2}{|-t+z|^2} 
      = \sup_{t\in \sector{\theta}{\alpha}} \frac{|t|^2}{|t+z|^2} .
   \end{align*}
   Now the claim follows from \ref{item:lem:app:Hsectorb} applied to $\beta = -\alpha$ and $-z$ instead of $z$.
\end{proof}

The next lemma provides alternative estimates for $H_z$.
Instead of the half-plane $P= \beta + \e^{\I\theta} \{ \zeta\in\C:  \im \zeta \le 0 \}$ from the proof of Lemma~\ref{lem:app:Hest}, we use the half-plane $P' = \e^{\I\theta} \{  \zeta\in\C: \im \zeta \le \beta\sin\theta \}$.

\begin{lemma}
   \label{lem:app:Hestalt} 
   Let $a,b\ge 0$.
   Let $H_z(t) = \frac{a^2 + b^2 |t|^2}{ |t-z|^2 } $ be the function from Lemma~\ref{lem:nolocalextrema}.
   \begin{enumerate}[label={\upshape(\roman*')}]
      \setcounter{enumi}{2}

      \item
      \label{item:lem:app:distsectorb:alt}
      Let $\theta \in [0,\pi/2)$, $\beta\in\R$ and let $z\in \C \setminus \sector{\theta}{\beta}$.
      Then
      \begin{multline}
	 \label{eq:Hsectorb:alt}
	 \tag{\ref{eq:Hsectorb}'}
	 \sup_{t\in \sector{\theta}{\beta}} H_z(t)
	 \\
	 \le b^2 + 
	 \begin{cases}
	    b^2 \frac{ (\re( \e^{-\I\theta}z) )^2 }{ (\im (  \e^{-\I\theta}(z - \beta )))^2 } 
	    + \frac{ a^2 + b^2 \beta^2\sin^2\theta}{ (\im (  \e^{-\I\theta}(z - \beta )))^2 } 
	    &\qquad \text{if } \arg(z-\beta)\in(\theta, \theta+\pi).
	    \\[2ex]
	    b^2 \frac{ (\re( \e^{\I\theta}z) )^2 }{ (\im (  \e^{\I\theta}(z - \beta )))^2 } 
	    + \frac{ a^2 + b^2 \beta^2\sin^2\theta}{ (\im (  \e^{\I\theta}(z - \beta )))^2 } 
	    &\qquad \text{if } \arg(z-\beta)\in(-\pi-\theta, -\theta).
	 \end{cases}
      \end{multline}

      \item
      \label{item:lem:app:distsectora:alt}
      Let $\theta \in [0,\pi/2)$, $\alpha\in\R$ and let $z\in \C \setminus -\sector{\theta}{-\alpha}$.
      Then
      \begin{multline}
	 \label{eq:Hsectora:alt}
	 \tag{\ref{eq:Hsectora}'}
	 \sup_{t\in \sector{\theta}{\alpha}} H_z(t)
	 \\
	 \le b^2 + 
	 \begin{cases}
	    b^2 \frac{ (\re( \e^{-\I\theta}z) )^2 }{ (\im (  \e^{-\I\theta}(z - \alpha )))^2 } 
	    + \frac{ a^2 + b^2 \alpha^2\sin^2\theta}{ (\im (  \e^{-\I\theta}(z - \alpha )))^2 } 
	    &\qquad \text{if } \arg(z-\alpha)\in(-\theta, -\theta+\pi).
	    \\[2ex]
	    b^2 \frac{ (\re( \e^{\I\theta}z) )^2 }{ (\im (  \e^{\I\theta}(z - \alpha )))^2 } 
	    + \frac{ a^2 + b^2 \alpha^2\sin^2\theta}{ (\im (  \e^{\I\theta}(z - \alpha )))^2 } 
	    &\qquad \text{if } \arg(z-\alpha)\in(\theta, \theta-\pi).
	 \end{cases}
      \end{multline}

   \end{enumerate}
\end{lemma}
\begin{proof} 
   We only  show \eqref{eq:Hsectorb:alt}.
   Its proof is similar to the one of \eqref{eq:Hsectorb}.

   Suppose that $\arg(z-\beta)\in(\theta,\theta+\pi)$.
   Then $H_z$ is well defined in the half-plane $P' = \e^{\I\theta} \{ \zeta\in\C : \im \zeta \le -\beta\sin\theta\}$  and $\sector{\theta}{\beta}\subseteq P'$.
   Therefore
   \begin{align*}
      \sup_{t\in \sector{\theta}{\beta} } H_z(t)
      & \le \sup_{t\in P' } H_z(t)
      = \sup_{t\in \partial P' } H_z(t)
      = \sup_{ \tau\in\R } H_z(\e^{\I\theta}(\tau-\I\beta\sin\theta))
      \\
      & = \sup_{ \tau\in\R} \frac{ a^2 + b^2 | \tau - \I\beta\sin\theta|^2 }{ | \tau - \I\beta\sin\theta - \e^{-\I\theta} z |^2 }
      \\
      & \le b^2 + \frac{ b^2 (\re( \e^{-\I\theta})z )^2 }{ ( \beta\sin\theta + \im( \e^{-\I\theta}z ) )^2 } 
      + \frac{ a^2 + b^2 \beta^2\sin^2\theta}{ ( \beta\sin\theta + \im( \e^{-\I\theta}z ) )^2 } 
      \\
      & = b^2 + b^2 \frac{ (\re( \e^{-\I\theta})z )^2 }{ (\im (  \e^{\-I\theta}(z - \beta )))^2 } 
      + \frac{ a^2 + b^2 \beta^2\sin^2\theta}{ (\im (  \e^{\-I\theta}(z - \beta )))^2 } 
   \end{align*}
   where we used \eqref{eq:app:distrealaxis} in the last inequality.
   See Figure~\ref{fig:rotatedhyperbolas:alt}.
   If $\arg(z-\beta)\in(-\theta-\pi, -\theta)$, then $H_z$ is well defined in the half-plane 
   $Q' = \e^{-\I\theta} \{ \zeta\in\C : \im \zeta \ge \beta\sin\theta\}$  and $\sector{\theta}{\beta}\subseteq Q'$
   and we can use an analogous calculation to estimate $H_z(t)$ for such $z$.
   \smallskip
\end{proof}

The situation in the next lemma is as in Lemma~\ref{lem:app:Hest}~\ref{item:lem:app:distanceline} but for the special case when $\re(z) =0$.
\begin{lemma}
   \label{lem:app:lemma0}
   Let $a,b\ge 0$ and $\mu > 0$.
   Let $H_z(t) = \frac{a^2 + b^2 |t|^2}{ |t-z|^2 } $ be the function from Lemma~\ref{lem:nolocalextrema}.
   \begin{enumerate}[label={\upshape(\roman*)}]

      \item
      For all $x \neq \mu \in\R$,
      \begin{equation}
	 \label{eq:supline}
	 \sup_{t\in x+\I\R} H_\mu(t)
	 =\max\left\{b^2,\ \frac{ a^2+b^2 x^2 }{ (\mu- x)^2} \right\}.
      \end{equation}

      \item
      Let $\alpha, \beta \in \R$ with $\alpha<\beta$ and let $a, b\geq 0$. 
      Then, for all $\mu\in (\alpha, \beta)$,
      \begin{equation}
	 \label{eq:supstrip}
	 \sup_{t\in\C\setminus\big((\alpha,\beta)+\I \R\big)} H_\mu(t)
	 =\max\left\{b^2,\ \frac{ a^2+b^2\alpha^2 }{ (\mu-\alpha)^2},\ \frac{a^2+b^2\beta^2}{(\beta-\mu)^2}\right\}.
      \end{equation}

   \end{enumerate}
\end{lemma}
\begin{proof} 
   For $t=x + \I\sigma$ we obtain
   \begin{equation*}
      H_\mu(t)
      = H_\mu(x + \I\sigma)
      = \frac{ a^2 + b^2 x^2 + b^2 \sigma^2 }{ (x-\mu)^2 + \sigma^2} 
      = b^2 + \frac{ a^2 + b^2 [ x^2 - (x-\mu)^2] }{ (x-\mu)^2 + \sigma^2}.
   \end{equation*}
   If $a^2 + b^2 [ x^2 - (x-\mu)^2] \ge 0$, then the function has a global maximum when $\sigma=0$, that is in $t = x$;
   otherwise it has no global maximum and its supremum is $b^2$ which proves \eqref{eq:supline}.
   Note that $H_\mu$ has no local maxima by Lemma~\ref{lem:nolocalextrema}.
   \begin{equation*}
      \sup_{t\in\C\setminus\big((\alpha,\beta)+\I \R\big)}\frac{ a^2+b^2|t|^2 }{ |t-\mu|^2 }
      = \sup_{t\in\{\alpha, \beta\} +\I \R} H_\mu(t)
      = \max_{x \in \{\alpha, \beta\} } \sup_{\sigma\in\R} H_\mu( x + \I\sigma).
   \end{equation*}
   Now \eqref{eq:supstrip} follows from \eqref{eq:supline}.
\end{proof}

The next lemma is used in Theorem~\ref{thm:bisecgap}.
\begin{lemma}
   Let $\alpha < \beta$ and $\theta\in \left[0, \frac{\pi}{2}\right)$.
   Then, for all $z\in \C$ we obtain
   \begin{alignat}{4}
      \label{eq:Hsector2b}
      \sup_{t\in \sector{\theta}{\beta}} H_z(t)
      \le 
      \max\left\{b^2,\ \frac{{a^2+b^2\beta^2}}{(\beta- \re z)^2} \right\} + \tan^2\theta.
      \qquad
      &&\text{ if }  \re z < \beta,
      \\
      \label{eq:Hsector2a}
      \sup_{t\in -\sector{\theta}{\alpha}} H_z(t)
      \le 
      \max\left\{b^2,\ \frac{{a^2+b^2\alpha^2}}{( \re z-\alpha)^2} \right\} + \tan^2\theta.
      \qquad
      &&\text{ if }  \re z > \alpha.
   \end{alignat}
\end{lemma}
\begin{proof}
   We only prove \eqref{eq:Hsector2b}.
   Let us write $z = \mu + \I \nu$ and 
   $t = \tau + \I \sigma$.
   Observe that 
   \begin{align*}
      H_z(t) 
      &= \frac{a^2+b^2|t|^2}{|t-z|^2}
      = \frac{a^2+b^2|t-\I\nu|^2}{|t-z|^2}
      + \frac{ b^2((\sigma^2 - (\sigma-\nu)^2}{|t-z|^2 }
      \\
      &= H_{\mu}(\tau + \I(\sigma-\nu)) + b^2\frac{\sigma^2-(\sigma-\nu)^2}{|t-z|^2}
      \\
      &= H_{\mu}(t - \I\nu) + b^2\frac{\sigma^2-(\sigma-\nu)^2}{|t-z|^2}.
   \end{align*}
   Therefore Lemma~\ref{lem:nolocalextrema} implies that 
   \begin{align*}
      \sup_{t\in \sector{\theta}{\beta}} H_z(t)
      = \sup_{t\in \partial \sector{\theta}{\beta}} H_z(t)
      \le \sup_{t\in \partial \sector{\theta}{\beta}} H_{\mu}(t - \I\nu)
      + \sup_{t\in \partial \sector{\theta}{\beta}} b^2\frac{\sigma^2-(\sigma-\nu)^2}{|t-z|^2}.
   \end{align*}
   For the first term we obtain 
   \begin{align*}
      \sup_{t\in \partial \sector{\theta}{\beta}} H_{\mu}(t - \I\nu)
      & \le \sup_{\re t \ge \beta} H_{\mu}(t - \I\nu)
      = \sup_{\re t = \beta} H_{\mu}(t)
       \le \max\left\{b^2,\ \frac{{a^2+b^2\beta^2}}{(\beta-\mu)^2} \right\}
   \end{align*}
   by \eqref{eq:supline}.
   For the second term, we use that 
   $t = \tau + \I\sigma = (\beta + |\sigma|\cot\theta) + \I\sigma$, $\sigma\in \R$, on $\partial \sector{\theta}{\beta}$.
   Therefore we obtain
   \begin{align*}
      \sup_{t\in \partial \sector{\theta}{\beta}} b^2\frac{\sigma^2-(\sigma-\nu)^2}{|t-z|^2}
      & = b^2 \sup_{\sigma\in \R} \frac{\sigma^2-(\sigma-\nu)^2}{(\beta-\mu + |\sigma|\cot\theta)^2 + (\nu-\sigma)^2}
      \le b^2 \sup_{\sigma\in \R} \frac{\sigma^2}{(\beta-\mu + |\sigma|\cot\theta)^2}
      \\
      & \le b^2 \tan^2\theta.
   \end{align*}
   In the last step we used that $\beta-\mu = \beta-\re z > 0$.
\end{proof}

The next lemma is used in Proposition~\ref{prop:psubord}.
Recall that $\parabola(\beta) = \{ z\in\C : \re z > \beta,\, |\im (z)| \le (\re z - \beta)^2 \}$.
\begin{lemma}
   \label{lem:app:parabola}
   Let $\alpha, \beta\in\R$ and $z\in\C$.
   Then
   \begin{align}
      \label{eq:lem:parabola:b}
      \dist(z,\parabola(\beta)) &\ge 
      \begin{cases}
	 \sqrt{(\beta- \re z)^2 + (\im z)^2},\quad & |\im z|\le \frac{1}{2}, \\
	 \sqrt{(\beta- \re z)^2 + |\im z| - \frac{1}{4}},\quad & |\im z| > \frac{1}{2},
      \end{cases}
      \qquad\text{if } \re z < \beta,
      \\[2ex]
      \label{eq:lem:parabola:a}
      \dist(z,-\parabola(-\alpha)) &\ge 
      \begin{cases}
	 \sqrt{(\alpha- \re z)^2 + (\im z)^2},\quad & |\im z|\le \frac{1}{2}, \\
	 \sqrt{(\alpha- \re z)^2 + |\im z| - \frac{1}{4}},\quad & |\im z| > \frac{1}{2}.
      \end{cases}
      \qquad\text{if } \re z > \alpha.
   \end{align}
\end{lemma}
\begin{proof}
   It suffices to show \eqref{eq:lem:parabola:b} because
   $\dist(z,-\parabola(-\alpha)) = \dist(-z,\parabola(\alpha))$.
   Let $z = \mu + \I \nu\in\C$ and assume $\re z < \beta$.
   For $t= \tau + \I\sigma \in \parabola(\beta)$ we define
   \begin{equation*}
      \delta_z(t):=|z-t|^2.
   \end{equation*}
   Clearly, $\delta$ has no critical points in the interior of $\parabola(\beta)$, hence
   \begin{equation*}
      \inf_{t\in \parabola(\beta) }\delta_z(t) = \inf_{t\in \partial(\parabola(\beta)) }\delta_z(t).
   \end{equation*}
   Note that the infimum is in fact a minimum.
   Let $t = \tau + \I \sigma \in\parabola(\beta)$.
   Hence $\sigma = \pm (\tau - \beta)^2$.
   Let us first assume that $\sigma = (\tau-\beta)^2$.
   Then
   \begin{align*}
      \delta_z(t)= (\tau-\mu)^2+(\nu-(\tau-\beta)^2)^2=  (\tau-\mu)^2-2\nu(\tau-\beta)^2+(\tau-\beta)^4+\nu^2.
   \end{align*}
   If $\nu \leq 0$, then
   \begin{equation*}
      \delta_z(t)\geq (\tau-\mu)^2+\nu^2\geq  (\beta-\mu)^2+\nu^2.
   \end{equation*}
   If $\nu \geq 0$, let $\Delta(\tau):=(\tau-\mu)^2-(\tau-\beta)^2$. Then
   \begin{align*}
      \delta_z(t)
      &= \left[(\tau-\beta)^2+\frac{1-2\nu}{2}\right]^2+\nu^2-\left(\frac{1-2\nu}{2}\right)^2+\Delta(\tau).
   \end{align*}
   Suppose that $\nu\leq 1/2$, so that $1-2\nu \geq 0$. Then 
   $$\delta_z(t)\geq \left(\frac{1-2\nu}{2}\right)^2+\nu^2-\left(\frac{1-2\nu}{2}\right)^2+\Delta(\tau)= \nu^2+\Delta(\tau). $$
   If $\nu > 1/2$, we have that 
   $$\delta_z(t)\geq \nu^2-\left(\frac{1-2\nu}{2}\right)^2+\Delta(\tau)= \nu-\frac{1}{4}+\Delta(\tau).$$
   Note that 
   \begin{align*}
      \Delta(\tau)
      =(2\tau-(\beta+\mu))(\beta-\mu)
      \geq (2\beta-(\beta+\mu))(\beta-\mu)\geq (\beta-\mu)^2.
   \end{align*}
   Therefore
   \begin{align}
      \label{e:eq1subordinate}
      \delta_z(t)\geq 
      \begin{cases}
	 (\beta-\mu)^2+\nu^2,&\text{if } \nu \leq 1/2, \\
	 (\beta-\mu)^2+ \nu-\frac{1}{4},&\text{if } \nu > 1/2.
      \end{cases}
      \qquad\text{if } \im t = (\re t -\beta)^2.
   \end{align}
   Since $\delta_z(t) = \delta_{\overline z}(\overline t)$, we obtain 
   \begin{align}
      \label{e:eq2subordinate}
      \delta_z(t)\geq 
      \begin{cases}
	 (\beta-\mu)^2 + \nu^2,&\text{if } \nu \geq -1/2, \\
	 (\beta-\mu)^2 - \nu-\frac{1}{4},&\text{if } \nu < 1/2.
      \end{cases}
      \qquad\text{if } \im t = -(\re t -\beta)^2.
   \end{align}
   Since the right hand sides of \eqref{e:eq1subordinate} and \eqref{e:eq2subordinate} are independent of $t$, the claim follows. 
\end{proof}


\section{Alternative bounds for the resolvent set of $T+A$ for sectorial operators $T$} 
\label{app:alternativebounds}

\cite{werw}

In the proofs of Theorem~\ref{thm:sectorial},
Theorem~\ref{thm:bisecgap}, Corollary~\ref{cor:resolventestimate2}, Proposition~\ref{prop:bisectorstrip} we used the estimates 
\eqref{eq:Hsectorb} and \eqref{eq:Hsectora} from Lemma~\ref{lem:app:Hest} which resulted in sets of the form $\e^{\I\theta}(\beta+\hyperbola_\beta)$.
If we use the estimates 
\eqref{eq:Hsectorb:alt} and \eqref{eq:Hsectora:alt} from Lemma~\ref{lem:app:Hestalt} then 
\eqref{eq:ourfirsttheo2:b2} is replaced by 
\begin{align}
   \tag{\ref{eq:ourfirsttheo2:b2}'}
   \label{eq:ourfirsttheo2:b2:alt}
   \frac{ a^2  + b^2 \beta_T^2\sin^2\theta + b^2 (\re ( \e^{ -\I\theta}z ) )^2 }{1-b^2}
   < (\im (\e^{ -\I\theta} (z - \beta_T)) )^2.
\end{align}
and consequently $\sup_{t\in \sector{\theta}{\beta_T}} \frac{a^2+b^2|t|^2}{|t-z|^2} < 1$ if
\begin{equation}
   \label{eq:summaryb:alt}
   \re z < \beta_{T+A}
   \quad\text{or}\quad
   z\in \e^{\I\theta}(\beta_T\sin\theta + \hyperbola_{\beta_T\sin\theta}) \cup \e^{-\I\theta}(-\beta_T\sin\theta - \hyperbola_{\beta_T\sin\theta}).
\end{equation}
Therefore, under the conditions of Theorem~\ref{thm:bisecgap}
\begin{multline*}
   \Big(
   \e^{\I\theta}(\beta_T\sin\theta +\hyperbola_{\beta_T\sin\theta}) 
   \cup \e^{-\I\theta}(-\beta_T\sin\theta -\hyperbola_{\beta_T\sin\theta}\sin\theta)
   \cup \{\re z < \beta_{T+A} \}
   \Big)
   \\
   \cap
   \Big(
   \e^{-\I\theta}(\alpha_T\sin\theta +\hyperbola_{\alpha_T\sin\theta}) 
   \cup \e^{-\I\theta}(-\alpha_T\sin\theta -\hyperbola_{\alpha_T\sin\theta})
   \cup \{\re z > \alpha_{T+A} \}
   \Big)
   \\
   \subseteq \rho(T+A).
\end{multline*}
Similarly, the estimates for the spectral inclusions in the proofs of
Theorem~\ref{thm:sectorial},
Corollary~\ref{cor:resolventestimate2}, 
Proposition~\ref{prop:bisectorstrip}  
can be modified.
The hyperbolas that provide the spectral inclusions have the same asymptotic slopes with either of the estimates from Lemma~\ref{lem:app:Hest} or Lemma~\ref{lem:app:Hestalt} but they are shifted differently, thus leading to different spectral inclusions.
See Figure~\ref{fig:TEST}.

We find that the estimates from Lemma~\ref{lem:app:Hest} yield better results 
than those from Lemma~\ref{lem:app:Hestalt} 
if $\alpha<0$ and worse results if $\alpha>0$.
Analogously, they are better if $\beta > 0$ and worse if $\beta < 0$.
To see this, let $\psi = \theta + \arctan \frac{b}{\sqrt{1+b^2}}$ and 
note that for instance for $\re z \to\infty$ the asymptotes are given by the following functions:
\begin{align*}
   \text{Asymptote of } \e^{\I\theta}(\beta+\hyperbola_\beta):\qquad
   & y  = (x-\beta)\tan(\psi),
   \\
   \text{Asymptote of } \e^{\I\theta} (\beta\sin\theta +\hyperbola_{\beta\sin\theta}): \qquad
   & \widetilde y  = x  \tan(\psi) - \beta\sin\theta(\cos\theta + \sin\theta \tan\psi).
\end{align*}

Clearly $y$ gives a better estimate than $\widetilde y$ if and only if $y(x) < \widetilde y(x)$.
Note that 
\begin{align*}
   y(x) < \widetilde y(x)
   &\; \iff \;
   -\beta\tan(\psi) < -\beta\sin\theta(\cos\theta + \sin\theta \tan\psi)
   \\
   &\; \iff \;
   0 < \beta\tan(\psi) - \beta\sin\theta(\cos\theta + \sin\theta \tan\psi).
\end{align*}
Now
\begin{align*}
   \beta\tan(\psi) - \beta\sin\theta(\cos\theta + \sin\theta \tan\psi)
   &=
   \beta ( \tan \psi - \sin\theta\cos\theta - \sin^2\theta \tan\psi )
   \\
   &=
   \beta ( \cos^2\theta \tan \psi - \sin\theta\cos\theta)
   \\
   &=
   \beta \cos^2\theta (\tan\psi - \tan\theta ).
\end{align*}
Since $0 \le \theta \le \psi < \pi/2$, the expression in parenthesis is $\ge 0$.

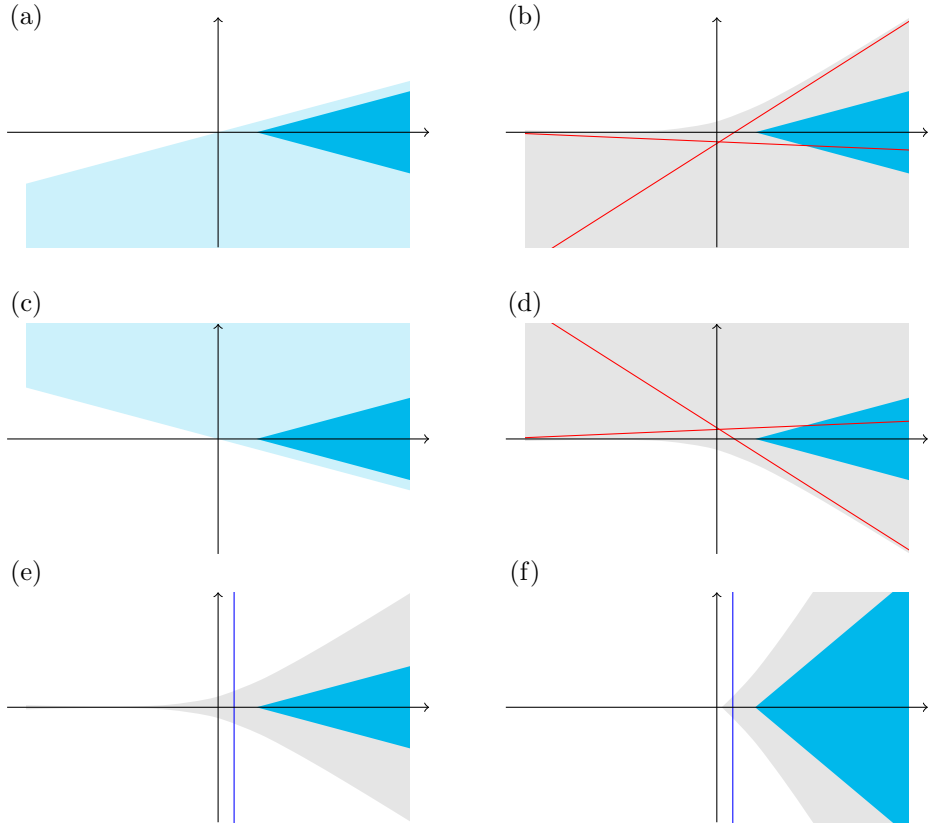
\begin{figure}[H] 
   \begin{tikzpicture}[ 
      declare function={ 
      hyperbolarotb(\x) = sqrt( (\abound^2 + \bbound^2*(\betaT)^2  + \bbound^2*(\x)^2)/(1-\bbound^2) ); 
      altbhyperbola(\x) =  \betaT*(sin(\thetaT)) - sqrt( (\abound^2 + \bbound^2*(\betaT)^2*(sin(\thetaT))^2 + \bbound^2*(\x)^2)/(1-\bbound^2) ); 
      } ,
      scale=.85
      ]

      \tikzmath{
      \gT = 0;  
      \abound = .5;   
      \bbound = .3;  
      \alphaT = -2;   
      \betaT = 1;      
      \thetaT = 15;      
      \alphaTA = \alphaT + sqrt( \abound^2 + \bbound^2*\gT^2 + \bbound^2*(\alphaT)^2);   
      \betaTA = \betaT - sqrt( \abound^2 + \bbound^2*\gT^2 + \bbound^2*(\betaT)^2);      
      \thetaB = atan(\bbound/sqrt(1-\bbound^2));
      }

      \begin{scope}[scale=.6, xshift=-13cm] 

	 \begin{scope}
	    \clip (-5,-3) rectangle (5, 3);

	    \path [fill=colPlane, rotate=\thetaT, shift={(\betaT,0)}] (-8,0) -- (8,0) --  (8, -8) -- (-8,-8) -- cycle;

	    \path [fill=colTspec, shift={(\betaT,0)}] (\thetaT:6) -- (0,0) --  (-\thetaT:6) -- cycle;

	 \end{scope}

	 \draw[->] (-5.5,0)--(5.5,0);
	 \draw[->] (0,-3)--(0,3);

	 \node at (-5,3) {(a)};
      \end{scope}


      \begin{scope}[scale=.6, xshift=0cm] 

	 \begin{scope}
	    \clip (-5,-3) rectangle (5, 3);


	    \fill [colTAspec, opacity=.2 ] plot[domain=-15:15, smooth, rotate=1*\thetaT] ({\x},{-altbhyperbola(\x)}) 
	    -- (15, -15) -- (-15, -15) -- cycle;

	    \path [fill=colTspec, shift={(\betaT,0)}] (\thetaT:6) -- (0,0) --  (-\thetaT:6) -- cycle;

  	    \draw [red, rotate=\thetaT] plot [domain=-7.0:7.0, smooth] ( {\x},{(\x)*tan(\thetaB)-\betaT*(sin(\thetaT))} );
  	    \draw [red, rotate=\thetaT] plot [domain=-7.0:7.0, smooth] ( {\x},{-(\x)*tan(\thetaB)-\betaT*(sin(\thetaT))} );
	 \end{scope}

	 \draw[->] (-5.5,0)--(5.5,0);
	 \draw[->] (0,-3)--(0,3);

	 \node at (-5,3) {(b)};
      \end{scope}


      \begin{scope}[scale=.6, xshift=-13cm, yshift=-8cm] 

	 \begin{scope}
	    \clip (-5,-3) rectangle (5, 3);

	    \path [fill=colPlane, rotate={180-\thetaT}, shift={(\betaT,0)}] (-8,0) -- (8,0) --  (8, -8) -- (-8,-8) -- cycle;

	    \path [fill=colTspec, shift={(\betaT,0)}] (\thetaT:6) -- (0,0) --  (-\thetaT:6) -- cycle;

	 \end{scope}

	 \draw[->] (-5.5,0)--(5.5,0);
	 \draw[->] (0,-3)--(0,3);

	 \node at (-5,3.5) {(c)};
      \end{scope}


      \begin{scope}[scale=.6, xshift=0cm, yshift=-8cm] 

	 \begin{scope}
	    \clip (-5,-3) rectangle (5, 3);


	    \fill [colTAspec, opacity=.2 ] 
	    plot[domain=-15:15, smooth, rotate=-\thetaT] ({\x},{altbhyperbola(\x)}) -- (15, 15) -- (-15, 15) -- cycle;
	    \path [fill=colTspec, shift={(\betaT,0)}] (\thetaT:6) -- (0,0) --  (-\thetaT:6) -- cycle;

  	    \draw [red, rotate=-\thetaT] plot [domain=-7.0:7.0, smooth] ( {\x},{-(\x)*tan(\thetaB)+\betaT*(sin(\thetaT))} );
  	    \draw [red, rotate=-\thetaT] plot [domain=-7.0:7.0, smooth] ( {\x},{(\x)*tan(\thetaB)+\betaT*(sin(\thetaT))} );
	 \end{scope}

	 \draw[->] (-5.5,0)--(5.5,0);
	 \draw[->] (0,-3)--(0,3);

	 \node at (-5,3.5) {(d)};
      \end{scope}


      \begin{scope}[scale=0.6, xshift=-13cm, yshift=-15cm] 

	 \begin{scope}
	    \clip (-5,-3) rectangle (5, 3);

	    \begin{scope}
	       \clip plot[domain=-15:15, smooth, rotate=1*\thetaT] ({\x},{-altbhyperbola(\x)}) -- (15, -15) -- (-15, -15) -- cycle;
	       \fill [colTAspec, opacity=.2 ]
	       plot[domain=-15:15, smooth, rotate=-\thetaT] ({\x},{altbhyperbola(\x)}) -- (15, 15) -- (-15, 15) -- cycle;
	    \end{scope}

	    \path [fill=colTspec, shift={(\betaT,0)}] (\thetaT:6) -- (0,0) --  (-\thetaT:6) -- cycle;

	    \draw [blue] (\betaTA, -5) -- (\betaTA, 5);

	 \end{scope}

	 \draw[->] (-5.5,0)--(5.5,0);
	 \draw[->] (0,-3)--(0,3);

	 \node at (-5,3.5) {(e)};

      \end{scope}


      \begin{scope}[scale=0.6, xshift=0cm, yshift=-15cm] 

	 \begin{scope}

	    \tikzmath{
	    \thetaT = 40;      
	    \betaTA = \betaT - sqrt( \abound^2 + \bbound^2*\gT^2 + \bbound^2*(\betaT)^2);      
	    }

	    \clip (-5,-3) rectangle (5, 3);

	    \begin{scope}
	       \clip plot[domain=-15:15, smooth, rotate=1*\thetaT] ({\x},{-altbhyperbola(\x)}) -- (15, -15) -- (-15, -15) -- cycle;
	       \fill [colTAspec, opacity=.2 ]
	       plot[domain=-15:15, smooth, rotate=-\thetaT] ({\x},{altbhyperbola(\x)}) -- (15, 15) -- (-15, 15) -- cycle;
	    \end{scope}

	    \path [fill=colTspec, shift={(\betaT,0)}] (\thetaT:6) -- (0,0) --  (-\thetaT:6) -- cycle;

	    \draw [blue] (\betaTA, -5) -- (\betaTA, 5);

	 \end{scope}

	 \draw[->] (-5.5,0)--(5.5,0);
	 \draw[->] (0,-3)--(0,3);

	 \node at (-5,3.5) {(f)};

      \end{scope}


   \end{tikzpicture}
   
   \caption[Illustration for the spectral inclusions with alternative bounds.]
   {Illustration for the proof of Lemma~\ref{lem:app:Hestalt} and the boundaries for the spectrum of $T+A$ that we would yield in Theorem~\ref{thm:bisecgap}.
   In (a), the lightblue half-plane is the plane $P'$ above which the estimate \eqref{eq:Hsectorb:alt} for $\arg(z-\beta)\in(\theta, \theta+\pi)$ holds.
   The white regions in (b) and (d) correspond to the sets
   $\e^{\I\theta}(\beta_T\sin\theta + \hyperbola_{\beta_T\sin\theta})$ and
   $\e^{-\I\theta}(-\beta_T\sin\theta - \hyperbola_{\beta_T\sin\theta})$ respectively.

   There, $\sup_{t\in\partial(\sector{\theta}{\beta_T})} H_z(t) < 1$ holds.
   If we combine these results, we find that the spectrum of $T+A$ must be contained in the grey area in (e).
   The graphic in (f) shows the spectral inclusion for a larger angle $\theta$.

   The vertical blue line in (e) and (f) is the lower bound $\beta_{T+A}$ for the real part of the spectrum of $T+A$ in the case $b<\cos\theta$, so $\sigma(T+A)$ is contained in the grey area to its right.
   }
  \label{fig:rotatedhyperbolas:alt}
\end{figure}

\begin{figure}[H] 
   \centering

   \begin{tikzpicture}[ 
      declare function={ 
      hyperbolarota(\x) = sqrt( (\abound^2 + \bbound^2*(\alphaT)^2 + \bbound^2*(\x)^2)/(1-\bbound^2) ); 
      altapphyperbola(\x) = \alphaT*(sin(\thetaT)) + sqrt( (\abound^2 + \bbound^2*(\alphaT)^2*(sin(\thetaT))^2 + \bbound^2*(\x)^2)/(1-\bbound^2) ); 
      altapmhyperbola(\x) = \alphaT*(sin(\thetaT)) - sqrt( (\abound^2 + \bbound^2*(\alphaT)^2*(sin(\thetaT))^2 + \bbound^2*(\x)^2)/(1-\bbound^2) ); 
      hyperbolarotb(\x) = sqrt( (\abound^2 + \bbound^2*(\betaT)^2 + \bbound^2*(\x)^2)/(1-\bbound^2) ); 
      altbhyperbola(\x) = \betaT*(sin(\thetaT)) - sqrt( (\abound^2 + \bbound^2*(\betaT)^2*(sin(\thetaT))^2 + \bbound^2*(\x)^2)/(1-\bbound^2) ); 
      altbpphyperbola(\x) = \betaT*(sin(\thetaT)) + sqrt( (\abound^2 + \bbound^2*(\betaT)^2*(sin(\thetaT))^2 + \bbound^2*(\x)^2)/(1-\bbound^2) ); 
      altbpmhyperbola(\x) = \betaT*(sin(\thetaT)) - sqrt( (\abound^2 + \bbound^2*(\betaT)^2*(sin(\thetaT))^2 + \bbound^2*(\x)^2)/(1-\bbound^2) ); 
      } ,
      scale=.85]

      \tikzmath{
      \gT = 1.5;  
      \abound = 1.0;   
      \bbound = .3;  
      \alphaT = -1;   
      \betaT = 2;      
      \thetaT = 38;      
      \alphaTA = \alphaT + sqrt( \abound^2 + \bbound^2*\gT^2 + \bbound^2*(\alphaT)^2);   
      \betaTA = \betaT - sqrt( \abound^2 + \bbound^2*\gT^2 + \bbound^2*(\betaT)^2);      
      }

      \begin{scope}[scale=0.4, transform shape] 
	 \tikzmath{
	 \betaT = 5;   
	 \betaTA = \betaT - sqrt( \abound^2 + \bbound^2*\gT^2 + \bbound^2*(\betaT)^2);   
	 }

	 \begin{scope}
	    \clip (-5,-5) rectangle (10, 5);

	    \begin{scope}
	       \clip plot[domain=-15:15, smooth, shift={(\betaT,0)},  rotate=\thetaT] ({\x},{hyperbolarotb(\x)}) -- (15,-15) -- (-15,-15)-- cycle;
	       \fill [green, opacity=.2 ]
	       plot[domain=-15:15, smooth, shift={(\betaT,0)},  rotate={180-\thetaT}] ({\x},{hyperbolarotb(\x)}) -- (-15,15) -- (15,15) -- cycle;
	       \draw [teal, densely dotted] plot[domain=-15:15, smooth, shift={(\betaT,0)},  rotate={180-\thetaT}] ({\x},{hyperbolarotb(\x)});
	    \end{scope}
	    \begin{scope}
	       \clip plot[domain=-15:15, smooth, shift={(\betaT,0)},  rotate={180-\thetaT}] ({\x},{hyperbolarotb(\x)}) -- (-15,15) -- (15,15) -- cycle;
	       \draw [teal, densely dotted] plot[domain=-15:15, smooth, shift={(\betaT,0)},  rotate=\thetaT] ({\x},{hyperbolarotb(\x)});
	    \end{scope}


	    \begin{scope}
	       \clip plot[domain=-15:15, smooth, rotate=-\thetaT] ({\x},{altbhyperbola(\x)}) -- (15, 15) -- (-15, 15) -- cycle;
	       \fill [red, opacity=.2 ]
	       plot[domain=-15:15, smooth, rotate=1*\thetaT] ({\x},{-altbhyperbola(\x)}) -- (15, -15) -- (-15, -15) -- cycle;
	       \draw [red, densely dotted] plot[domain=-15:15, smooth, rotate=1*\thetaT] ({\x},{-altbhyperbola(\x)});
	    \end{scope}
	    \begin{scope}
	       \clip plot[domain=-15:15, smooth, rotate=1*\thetaT] ({\x},{-altbhyperbola(\x)}) -- (15, -15) -- (-15, -15) -- cycle;
	       \draw [red, densely dotted] plot[domain=-15:15, smooth, rotate=-\thetaT] ({\x},{altbhyperbola(\x)});
	    \end{scope}

	    \draw[blue] (\betaTA,-5.5) -- (\betaTA,5.5);

	    \path [fill=colTspec, shift={(\betaT,0)}] (\thetaT:10) -- (0,0) --  (-\thetaT:10) -- cycle;

	 \end{scope}

	 \draw[blue] (\betaTA,-5.5) -- (\betaTA,5.5);

	 \draw[->] (-5.5,0)--(10.5,0);
	 \draw[->] (0,-5.5)--(0,5.5);
      \end{scope}


      \begin{scope}[xshift=10cm, scale=0.4, transform shape] 

	 \tikzmath{
	 \thetaT = 10;      
	 \betaT = 5;   
	 \betaTA = \betaT - sqrt( \abound^2 + \bbound^2*\gT^2 + \bbound^2*(\betaT)^2);   
	 }

	 \begin{scope}
	    \clip (-5,-5) rectangle (10, 5);

	    \begin{scope}
	       \clip plot[domain=-15:15, smooth, shift={(\betaT,0)},  rotate=\thetaT] ({\x},{hyperbolarotb(\x)}) -- (15,-15) -- (-15,-15)-- cycle;
	       \fill [green, opacity=.2 ]
	       plot[domain=-15:15, smooth, shift={(\betaT,0)},  rotate={180-\thetaT}] ({\x},{hyperbolarotb(\x)}) -- (-15,15) -- (15,15) -- cycle;
	    \end{scope}
	    \draw [teal, densely dotted] plot[domain=-15:15, smooth, shift={(\betaT,0)},  rotate={180-\thetaT}] ({\x},{hyperbolarotb(\x)});
	    \draw [teal, densely dotted] plot[domain=-15:15, smooth, shift={(\betaT,0)},  rotate=\thetaT] ({\x},{hyperbolarotb(\x)});


	    \begin{scope}
	       \clip plot[domain=-15:15, smooth, rotate=-\thetaT] ({\x},{altbhyperbola(\x)}) -- (15, 15) -- (-15, 15) -- cycle;
	       \fill [red, opacity=.2 ]
	       plot[domain=-15:15, smooth, rotate=1*\thetaT] ({\x},{-altbhyperbola(\x)}) -- (15, -15) -- (-15, -15) -- cycle;
	    \end{scope}
	    \draw [red, densely dotted] plot[domain=-15:15, smooth, rotate=1*\thetaT] ({\x},{-altbhyperbola(\x)});
	    \draw [red, densely dotted] plot[domain=-15:15, smooth, rotate=-\thetaT] ({\x},{altbhyperbola(\x)});


	    \path [fill=colTspec, shift={(\betaT,0)}] (\thetaT:10) -- (0,0) --  (-\thetaT:10) -- cycle;

	 \end{scope}

	 \draw[blue] (\betaTA,-5.5) -- (\betaTA,5.5);

	 \draw[->] (-5.5,0)--(10.5,0);
	 \draw[->] (0,-5.5)--(0,5.5);
      \end{scope}


      \begin{scope}[xshift=0cm, yshift=-5cm, scale=0.4, transform shape] 
	 \tikzmath{
	 \betaT = 2;   
	 \betaTA = \betaT - sqrt( \abound^2 + \bbound^2*\gT^2 + \bbound^2*(\betaT)^2);   
	 }

	 \begin{scope}
	    \clip (-5,-5) rectangle (10, 5);

	    \begin{scope}
	       \clip plot[domain=-15:15, smooth, shift={(\betaT,0)},  rotate=\thetaT] ({\x},{hyperbolarotb(\x)}) -- (15,-15) -- (-15,-15)-- cycle;
	       \fill [green, opacity=.2 ]
	       plot[domain=-15:15, smooth, shift={(\betaT,0)},  rotate={180-\thetaT}] ({\x},{hyperbolarotb(\x)}) -- (-15,15) -- (15,15) -- cycle;
	       \draw [ teal, densely dotted ] plot[domain=-15:15, smooth, shift={(\betaT,0)},  rotate={180-\thetaT}] ({\x},{hyperbolarotb(\x)});
	    \end{scope}
	    \begin{scope}
	       \clip plot[domain=-15:15, smooth, shift={(\betaT,0)},  rotate={180-\thetaT}] ({\x},{hyperbolarotb(\x)}) -- (-15,15) -- (15,15) -- cycle;
	       \draw [ teal, densely dotted ] plot[domain=-15:15, smooth, shift={(\betaT,0)},  rotate={\thetaT}] ({\x},{hyperbolarotb(\x)});
	    \end{scope}



	    \begin{scope}
	       \clip plot[domain=-15:15, smooth, rotate=-\thetaT] ({\x},{altbhyperbola(\x)}) -- (15, 15) -- (-15, 15) -- cycle;
	       \fill [red, opacity=.2 ]
	       plot[domain=-15:15, smooth, rotate=1*\thetaT] ({\x},{-altbhyperbola(\x)}) -- (15, -15) -- (-15, -15) -- cycle;
	       \draw [red, densely dotted ] plot[domain=-15:15, smooth, rotate=1*\thetaT] ({\x},{-altbhyperbola(\x)});
	    \end{scope}
	    \begin{scope}
	       \clip plot[domain=-15:15, smooth, rotate=1*\thetaT] ({\x},{-altbhyperbola(\x)}) -- (15, -15) -- (-15, -15) -- cycle;
	       \draw [red, densely dotted ] plot[domain=-15:15, smooth, rotate=-\thetaT] ({\x},{altbhyperbola(\x)});
	    \end{scope}

	    \path [fill=colTspec, shift={(\betaT,0)}] (\thetaT:10) -- (0,0) --  (-\thetaT:10) -- cycle;

	 \end{scope}

	 \draw[blue] (\betaTA,-5.5) -- (\betaTA,5.5);

	 \draw[->] (-5.5,0)--(10.5,0);
	 \draw[->] (0,-5.5)--(0,5.5);
      \end{scope}


      \begin{scope}[xshift=10cm, yshift=-5cm, scale=0.4, transform shape] 
	 \tikzmath{
	 \thetaT = 10;      
	 \betaT = 2;   
	 \betaTA = \betaT - sqrt( \abound^2 + \bbound^2*\gT^2 + \bbound^2*(\betaT)^2);   
	 }

	 \begin{scope}
	    \clip (-5,-5) rectangle (10, 5);

	    \begin{scope}
	       \clip plot[domain=-15:15, smooth, shift={(\betaT,0)},  rotate=\thetaT] ({\x},{hyperbolarotb(\x)}) -- (15,-15) -- (-15,-15)-- cycle;
	       \fill [green, opacity=.2 ]
	       plot[domain=-15:15, smooth, shift={(\betaT,0)},  rotate={180-\thetaT}] ({\x},{hyperbolarotb(\x)}) -- (-15,15) -- (15,15) -- cycle;
	       \draw [teal, densely dotted ] plot[domain=-15:15, smooth, shift={(\betaT,0)},  rotate={180-\thetaT}] ({\x},{hyperbolarotb(\x)});
	    \end{scope}
	    \begin{scope}
	       \clip plot[domain=-15:15, smooth, shift={(\betaT,0)},  rotate={180-\thetaT}] ({\x},{hyperbolarotb(\x)}) -- (-15,15) -- (15,15) -- cycle;
	       \draw [teal, densely dotted ] plot[domain=-15:15, smooth, shift={(\betaT,0)},  rotate=\thetaT] ({\x},{hyperbolarotb(\x)});
	    \end{scope}



	    \begin{scope}
	       \clip plot[domain=-15:15, smooth, rotate=-\thetaT] ({\x},{altbhyperbola(\x)}) -- (15, 15) -- (-15, 15) -- cycle;
	       \fill [red, opacity=.2 ]
	       plot[domain=-15:15, smooth, rotate=1*\thetaT] ({\x},{-altbhyperbola(\x)}) -- (15, -15) -- (-15, -15) -- cycle;
	       \draw [red, densely dotted] plot[domain=-15:15, smooth, rotate=1*\thetaT] ({\x},{-altbhyperbola(\x)});
	       \draw [red, densely dotted] plot[domain=-15:15, smooth, rotate=-\thetaT] ({\x},{altbhyperbola(\x)});
	    \end{scope}


	    \path [fill=colTspec, shift={(\betaT,0)}] (\thetaT:10) -- (0,0) --  (-\thetaT:10) -- cycle;

	 \end{scope}

	 \draw[blue] (\betaTA,-5.5) -- (\betaTA,5.5);

	 \draw[->] (-5.5,0)--(10.5,0);
	 \draw[->] (0,-5.5)--(0,5.5);
      \end{scope}


      \begin{scope}[xshift=0cm, yshift=-10cm, scale=0.4, transform shape] 
	 \tikzmath{
	 \betaT = -2;   
	 \betaTA = \betaT - sqrt( \abound^2 + \bbound^2*\gT^2 + \bbound^2*(\betaT)^2);   
	 }

	 \begin{scope}
	    \clip (-5,-5) rectangle (10, 5);

	    \begin{scope}
	       \clip plot[domain=-15:15, smooth, shift={(\betaT,0)},  rotate=\thetaT] ({\x},{hyperbolarotb(\x)}) -- (15,-15) -- (-15,-15)-- cycle;
	       \fill [green, opacity=.2 ]
	       plot[domain=-15:15, smooth, shift={(\betaT,0)},  rotate={180-\thetaT}] ({\x},{hyperbolarotb(\x)}) -- (-15,15) -- (15,15) -- cycle;
	       \draw [teal, densely dotted] plot[domain=-15:15, smooth, shift={(\betaT,0)},  rotate={180-\thetaT}] ({\x},{hyperbolarotb(\x)});
	    \end{scope}
	    \begin{scope}
	       \clip plot[domain=-15:15, smooth, shift={(\betaT,0)},  rotate={180-\thetaT}] ({\x},{hyperbolarotb(\x)}) -- (-15,15) -- (15,15) -- cycle;
	       \draw [teal, densely dotted] plot[domain=-15:15, smooth, shift={(\betaT,0)},  rotate=\thetaT] ({\x},{hyperbolarotb(\x)});
	    \end{scope}



	    \begin{scope}
	       \clip plot[domain=-15:15, smooth, rotate=-\thetaT] ({\x},{altbhyperbola(\x)}) -- (15, 15) -- (-15, 15) -- cycle;
	       \fill [red, opacity=.2 ]
	       plot[domain=-15:15, smooth, rotate=1*\thetaT] ({\x},{-altbhyperbola(\x)}) -- (15, -15) -- (-15, -15) -- cycle;
	       \draw [red, densely dotted] plot[domain=-15:15, smooth, rotate=1*\thetaT] ({\x},{-altbhyperbola(\x)});
	       \draw [red, densely dotted] plot[domain=-15:15, smooth, rotate=-\thetaT] ({\x},{altbhyperbola(\x)}); 
	    \end{scope}


	    \path [fill=colTspec, shift={(\betaT,0)}] (\thetaT:15) -- (0,0) --  (-\thetaT:15) -- cycle;

	 \end{scope}

	 \draw[blue] (\betaTA,-5.5) -- (\betaTA,5.5);

	 \draw[->] (-5.5,0)--(10.5,0);
	 \draw[->] (0,-5.5)--(0,5.5);
      \end{scope}


      \begin{scope}[xshift=10cm, yshift=-10cm, scale=0.4, transform shape] 
	 \tikzmath{
	 \thetaT = 10;      
	 \betaT = -2;   
	 \betaTA = \betaT - sqrt( \abound^2 + \bbound^2*\gT^2 + \bbound^2*(\betaT)^2);   
	 }

	 \begin{scope}
	    \clip (-5,-5) rectangle (10, 5);

	    \begin{scope}
	       \clip plot[domain=-15:15, smooth, shift={(\betaT,0)},  rotate=\thetaT] ({\x},{hyperbolarotb(\x)}) -- (15,-15) -- (-15,-15)-- cycle;
	       \fill [green, opacity=.2 ]
	       plot[domain=-15:15, smooth, shift={(\betaT,0)},  rotate={180-\thetaT}] ({\x},{hyperbolarotb(\x)}) -- (-15,15) -- (15,15) -- cycle;
	       \draw [teal, densely dotted] plot[domain=-15:15, smooth, shift={(\betaT,0)},  rotate={180-\thetaT}] ({\x},{hyperbolarotb(\x)});
	       \draw [teal, densely dotted] plot[domain=-15:15, smooth, shift={(\betaT,0)},  rotate=\thetaT] ({\x},{hyperbolarotb(\x)});
	    \end{scope}


	    \begin{scope}
	       \clip plot[domain=-15:15, smooth, rotate=-\thetaT] ({\x},{altbhyperbola(\x)}) -- (15, 15) -- (-15, 15) -- cycle;
	       \fill [red, opacity=.2 ]
	       plot[domain=-15:15, smooth, rotate=1*\thetaT] ({\x},{-altbhyperbola(\x)}) -- (15, -15) -- (-15, -15) -- cycle;
	       \draw [red, densely dotted] plot[domain=-15:15, smooth, rotate=1*\thetaT] ({\x},{-altbhyperbola(\x)});
	       \draw [red, densely dotted] plot[domain=-15:15, smooth, rotate=-\thetaT] ({\x},{altbhyperbola(\x)});
	    \end{scope}

	    \path [fill=colTspec, shift={(\betaT,0)}] (\thetaT:15) -- (0,0) --  (-\thetaT:15) -- cycle;

	 \end{scope}

	 \draw[blue] (\betaTA,-5.5) -- (\betaTA,5.5);

	 \draw[->] (-5.5,0)--(10.5,0);
	 \draw[->] (0,-5.5)--(0,5.5);
      \end{scope}


   \end{tikzpicture}
   
   \caption[Comparison between the estimates for the spectral enclosure obtained with the estimates from Lemma~\ref{lem:app:Hest} and Lemma\ref{lem:app:Hestalt}.]
   {Comparison between the estimates for the spectral enclosure we obtain with the estimates from Lemma~\ref{lem:app:Hest} and Lemma~\ref{lem:app:Hestalt}.
   The blue area is $\sector{\theta}{\beta_T}$ for the angles $\theta=30^\circ$ and $\theta=10^\circ$ and 
   $\beta_T=5$, $\beta_T=2$ and $\beta_T=-2$.

   The vertical blue line is the lower bound $\beta_{T+A}$ for the real part of the spectrum of $T+A$, so only the region to its right is of interest.

   The green area shows the enclosure for $\sigma(T+A)$ which corresponds to the estimates from Lemma~\ref{lem:app:Hest} that we used throughout this work.
   They are given by 
   $\frac{ a^2 +  b^2 \beta_T^2 + b^2 (\re (\e^{ \mp\I\theta}(z{\color{red}-\beta_T})) )^2 }{1-b^2}
   < (\im (\e^{ \mp\I\theta} (z - \beta_T)) )^2$.
   \\
   The red area shows the enclosure for $\sigma(T+A)$ which corresponds to the estimates from Lemma~\ref{lem:app:Hestalt}.
   They are given by 
   $\frac{ a^2  + b^2 \beta_T^2{\color{red}\sin^2\theta} + b^2 (\re ( \e^{ \mp\I\theta}z ) )^2 }{1-b^2}
   < (\im (\e^{ \mp\I\theta} (z - \beta_T)) )^2$.
   \smallskip
   
   As expected, for large $|z|$ the green hyperbolas provide better estimates if $\beta_T>0$ than those given by the red ones.
   If $\beta_T<0$ the converse is true.
   }
  \label{fig:TEST}

\end{figure}
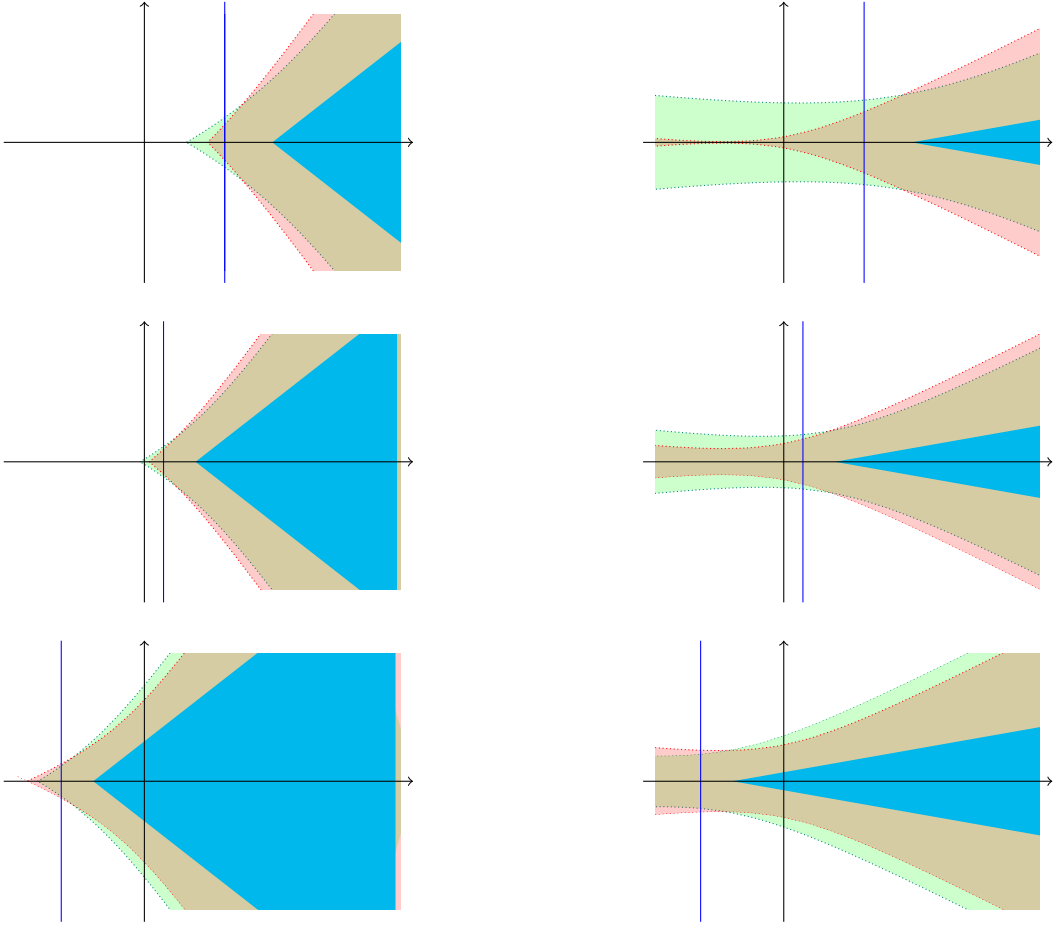


\addcontentsline{toc}{section}{Bibliography}

\begin{thebibliography}{KZPS76}

\bibitem[ALMO19]{ALMO2019}
A.~Agibalova, A.~Lunyov,  M.~Malamud and L.~ Oridoroga.
\newblock{Completeness property of one-dimensional perturbations of normal and spectral operators generated by first order systems.}
\newblock{\em Integr. Equ. Oper. Theory.}, 91:Paper No. 37, 35, 2019.

\bibitem[BK13]{Berkolaiko2013}
G.~Berkolaiko and P.~Kuchment.
\newblock {\em Introduction to quantum graphs}, volume 186 of {\em Mathematical
  Surveys and Monographs}.
\newblock American Mathematical Society, Providence, RI, 2013.

\bibitem[BL23]{BirkhoffLanger1923}
G.~Birkhoff, and R.~Langer.
\newblock{The boundary problems and developments associated with a system of ordinary differential equations of the first order.}
\newblock{\em Proc. Am. Acad. Arts Sci.}, 58:49--128, 1923.

\bibitem[Bre11]{Brezis}
H. Brezis.
\newblock {\em Functional analysis, {S}obolev spaces and partial differential
  equations}.
\newblock Universitext. Springer, New York, 2011.

\bibitem[CT16]{CueninTretter2016}
J.-C. Cuenin and C.~Tretter.
\newblock Non-symmetric perturbations of self-adjoint operators.
\newblock {\em J. Math. Anal. Appl.}, 441(1):235--258, 2016.

\bibitem[DS88a]{DunfordSchwartzI}
N.~Dunford and J.~T. Schwartz.
\newblock {\em Linear operators. {P}art {I}}.
\newblock Wiley Classics Library. John Wiley \& Sons, Inc., New York, 1988.
\newblock General theory, With the assistance of William G. Bade and Robert G.
  Bartle, Reprint of the 1958 original, A Wiley-Interscience Publication.

\bibitem[DS88b]{DunfordSchwartzII}
N.~Dunford and J.~T. Schwartz.
\newblock {\em Linear operators. {P}art {II}}.
\newblock Wiley Classics Library. John Wiley \& Sons, Inc., New York, 1988.
\newblock Spectral theory. Selfadjoint operators in Hilbert space, With the
  assistance of William G. Bade and Robert G. Bartle, Reprint of the 1963
  original, A Wiley-Interscience Publication.

\bibitem[DS88c]{DunfordSchwartzIII}
N.~Dunford and J.~T. Schwartz.
\newblock {\em Linear operators. {P}art {III}}.
\newblock Wiley Classics Library. John Wiley \& Sons, Inc., New York, 1988.
\newblock Spectral operators, With the assistance of William G. Bade and Robert
  G. Bartle, Reprint of the 1971 original, A Wiley-Interscience Publication.


\bibitem[EE87]{EdmundsEvans}
D.~E. Edmunds and W.~D. Evans.
\newblock {\em Spectral theory and differential operators}.
\newblock Oxford Mathematical Monographs. The Clarendon Press, Oxford
  University Press, New York, 1987.
\newblock Oxford Science Publications.


\bibitem[EN00]{EngelNagel2000}
K.-J. Engel and R.~Nagel.
\newblock {\em One-parameter semigroups for linear evolution equations}, volume
  194 of {\em Graduate Texts in Mathematics}.
\newblock Springer-Verlag, New York, 2000.
\newblock With contributions by S. Brendle, M. Campiti, T. Hahn, G. Metafune,
  G. Nickel, D. Pallara, C. Perazzoli, A. Rhandi, S. Romanelli and R.
  Schnaubelt.

\bibitem[GGK90]{GohGolMarKaa90}
I.~Gohberg, S.~Goldberg, and M.~A. Kaashoek.
\newblock {\em Classes of linear operators. {V}ol. {I}}, volume~49 of {\em
  Operator Theory: Advances and Applications}.
\newblock Birkh\"{a}user Verlag, Basel, 1990.

\bibitem[GK69]{GohbergKrein69}
I.~C. Gohberg and M.~G. Kre\u{\i}n.
\newblock {\em Introduction to the theory of linear nonselfadjoint operators}.
\newblock Translations of Mathematical Monographs, Vol. 18. American
  Mathematical Society, Providence, RI, 1969.
\newblock Translated from the Russian by A. Feinstein.

\bibitem[HKS15]{Krejcirik2015}
A.~Hussein, D.~Krej\v{c}i\v{r}\'{\i}k, and P.~Siegl.
\newblock Non-self-adjoint graphs.
\newblock {\em Trans. Amer. Math. Soc.}, 367(4):2921--2957, 2015.

\bibitem[Kat95]{Kato95}
T.~Kato.
\newblock {\em Perturbation theory for linear operators}.
\newblock Classics in Mathematics. Springer-Verlag, Berlin, 1995.
\newblock Reprint of the 1980 edition.

\bibitem[KZPS76]{krasnoselski}
M.~A. Krasnosel'ski\u{\i}, P.~P. Zabre\u{\i}ko, E.~I. Pustyl'nik, and P.~E.
  Sobolevski\u{\i}.
\newblock {\em Integral operators in spaces of summable functions}.
\newblock Monographs and Textbooks on Mechanics of Solids and Fluids,
  Mechanics: Analysis. Noordhoff International Publishing, Leiden, 1976.
\newblock Translated from the Russian by T. Ando.

\bibitem[LM15]{LM2015}
 A.~Lunyov and M.~Malamud. 
 \newblock{On the completeness and Riesz basis property
of root subspaces of boundary value problems for first order systems and
applications.} 
\newblock{\em J. Spectr. Theory.}, 5(1):17--70, 2015.

\bibitem[MO12]{MO2012}
M.~Malamud and L.~ Oridoroga. 
\newblock{On the completeness of root subspaces of boundary value problems for first order systems of ordinary differential equations}.
\newblock{\em J. Funct. Anal.}, 263:1939--1980, 2012.

\bibitem[Mar88]{Markus1988}
A.~S. Markus.
\newblock {\em Introduction to the spectral theory of polynomial operator
  pencils}, volume~71 of {\em Translations of Mathematical Monographs}.
\newblock American Mathematical Society, Providence, RI, 1988.
\newblock Translated from the Russian by H. H. McFaden, With an appendix by M.
  V. Keldysh.

\bibitem[MM81]{MarkusMatsaev1981}
A.~S. Markus and V.~I. Matsaev.
\newblock On the convergence of expansions according to the eigenvectors of an
  operator close to a self adjoint one.
\newblock {\em Mat. Issled.}, 61:104--129, 1981.

\bibitem[RR20]{riviere2019}
G.~Rivi\`ere and J.~Royer.
\newblock Spectrum of a non-selfadjoint quantum star graph.
\newblock {\em J. Phys. A}, 53(49):495202, 30, 2020.

\bibitem[Sch12]{Schmuedgen2012}
K.~Schm\"{u}dgen.
\newblock {\em Unbounded self-adjoint operators on {H}ilbert space}, volume 265
  of {\em Graduate Texts in Mathematics}.
\newblock Springer, Dordrecht, 2012.

\bibitem[Shk16]{shkalikov2016perturbations}
A.~A. Shkalikov.
\newblock Perturbations of self-adjoint and normal operators with discrete
  spectrum.
\newblock {\em Uspekhi Mat. Nauk}, 71(5(431)):113--174, 2016.

\bibitem[TW14]{TretterWyss2014}
C.~Tretter and C.~Wyss.
\newblock Dichotomous Hamiltonians with unbounded entries and solutions of Riccati equations.
\newblock {\em  J. Evol. Equ.}, 14:121--153, 2014.

\bibitem[vN22]{vanneerven}
J.~van Neerven.
\newblock {\em Functional analysis}, volume 201 of {\em Cambridge Studies in
  Advanced Mathematics}.
\newblock Cambridge University Press, Cambridge, 2022.

\bibitem[Wys10]{wyss2010}
C.~Wyss.
\newblock Riesz bases for {$p$}-subordinate perturbations of normal operators.
\newblock {\em J. Funct. Anal.}, 258(1):208--240, 2010.

\end{thebibliography}

\end{document}